\pdfoutput=1 % for arXiv to use pdflatex instead of latex+dvipdf

\documentclass[a4paper,twoside]{article}

\usepackage{amssymb}
\usepackage{amsmath}
\usepackage{amsthm}

\usepackage{microtype}

\usepackage{mathrsfs}

\usepackage{enumitem}

\usepackage{hyperref}

\newtheoremstyle{citing}% name
  {3pt}%      Space above, empty = `usual value'
  {3pt}%      Space below
  {\itshape}% Body font
  {}%         Indent amount (empty = no indent, \parindent = para indent)
  {\bfseries}% Thm head font
  {.}%        Punctuation after thm head
  {.5em}%     Space after thm head: " " = normal interword space;
  {\thmnote{#3}}% Thm head spec
\theoremstyle{citing}
\newtheorem*{citing}{}

\swapnumbers

\theoremstyle{plain}
\newtheorem{theorem}{Theorem}[section]
\newtheorem{lemma}[theorem]{Lemma}
\newtheorem{corollary}[theorem]{Corollary}

\theoremstyle{remark}
\newtheorem{remark}[theorem]{Remark}
\newtheorem{example}[theorem]{Example}

\theoremstyle{definition}
\newtheorem*{intro_definition}{Definition}
\newtheorem{definition}[theorem]{Definition}
\newtheorem{miniremark}[theorem]{}

\newcounter{counter1}

\newcommand{\Var}{\mathbf{V}}     % general varifolds
\newcommand{\RVar}{\mathbf{RV}}   % rectifiable varifolds

\newcommand{\Lp}[1]{\mathbf{L}_{#1}}

\newcommand{\Lploc}[1]{\mathbf{L}_{#1}^{\mathrm{loc}}}

\newcommand{\trunc}{\mathbf{T}}

\newcommand{\nat}{\mathscr{P}}

\newcommand{\rel}{\mathbf{R}}

\newcommand{\grass}[2]{\mathbf{G}(#1,#2)}

\newcommand{\pp}{\mathbf{p}}
\newcommand{\qq}{\mathbf{q}}

\newcommand{\oball}[2]{\mathbf{U}(#1,#2)}

\newcommand{\cball}[2]{\mathbf{B}(#1,#2)}

\newcommand{\density}{\boldsymbol{\Theta}}

\newcommand{\unitmeasure}[1]{\boldsymbol{\alpha}(#1)}

\newcommand{\besicovitch}[1]{\boldsymbol{\beta}(#1)}

\newcommand{\id}[1]{\mathbf{1}_{#1}}

\newcommand{\weakD}{\operatorname{\mathbf{D}}}

\newcommand{\derivative}[2]{{#1}\weakD{#2}}

\newcommand{\boundary}[2]{{#1}\,\partial{#2}}

\newcommand{\ud}{\,\mathrm{d}}

\DeclareMathOperator{\with}{:}

\newcommand{\classification}[3]{{#1} \cap \{ {#2} \with {#3} \}}

\newcommand{\project}[1]{#1_\natural}

\newcommand{\lIm}{[}
\newcommand{\rIm}{]}

\newcommand{\vdim}{{m}}
\newcommand{\adim}{{n}}

\newcommand{\mdistance}[2]{\boldsymbol{|} #1 \boldsymbol{|}_{#2}}

\newcommand{\printRoman}[1]{\setcounter{counter1}{#1}\Roman{counter1}}

\newcommand{\class}[1]{#1}

\newcommand{\tint}[2]{{\textstyle\int_{#1}^{#2}}}

\newcommand{\tsum}[2]{{\textstyle\sum_{#1}^{#2}}}

\newcommand{\measureball}[2]{{#1}\,{#2}}

\DeclareMathOperator{\without}{\sim}% complement of a set
\DeclareMathOperator{\restrict}{\llcorner}
\DeclareMathOperator{\card}{card}   % cardinality
\DeclareMathOperator{\Bdry}{Bdry}   % boundary
\DeclareMathOperator{\Clos}{Clos}   % closure
\DeclareMathOperator{\Int}{Int}     % interior
\DeclareMathOperator{\Tan}{Tan}     % tangent space
\DeclareMathOperator{\spt}{spt}     % support
\DeclareMathOperator{\im}{im}       % image
\DeclareMathOperator{\diam}{diam}   % diameter
\DeclareMathOperator{\Lip}{Lip}     % Lipschitz constant
\DeclareMathOperator{\grad}{grad}   % grad
\DeclareMathOperator{\dmn}{dmn}     % domain
\DeclareMathOperator{\dist}{dist}   % distance
\DeclareMathOperator{\Hom}{Hom}     % Vectorspace of homomorphisms
\DeclareMathOperator{\Der}{D}       % derivative
\DeclareMathOperator{\sign}{sign}   % sign of
\DeclareMathOperator{\ap}{ap}       % approximate
\DeclareMathOperator*{\aplim}{\mathrm{ap}\,\mathrm{lim}}
\newcommand{\Lpnorm}[3]{{#1}_{({#2})}({#3})}
\newcommand{\eqLpnorm}[3]{{(#1)}_{({#2})}({#3})}

\newcommand{\SWnorm}[3]{\mathbf{H}_{#1} ( {#2}, {#3} )}
\newcommand{\SWloc}[1]{\mathbf{H}_{#1}^{\mathrm{loc}}}
\newcommand{\SWSob}[1]{\mathbf{H}_{#1}}
\newcommand{\SWzero}[1]{\mathbf{H}_{#1}^{\diamond}}

\newenvironment{comment}
{

	\begin{footnotesize}
		\begin{quote}
}
{		\end{quote}
	\end{footnotesize}

}

\begin{document}

%%%%%%%%%%%%%%%%%%%%%%%%%%%%%%%%%%%%%%%%%%%%%%%%%%%%%%%%%%%%%%%%%%%%%%%%%%%%%%

\title{Sobolev functions on varifolds}
\author{Ulrich Menne}
\maketitle
\begin{abstract}
	This paper introduces first order Sobolev spaces on certain
	rectifiable varifolds.  These complete locally convex spaces are
	contained in the generally nonlinear class of generalised weakly
	differentiable functions and share key functional analytic properties
	with their Euclidean counterparts.

	Assuming the varifold to satisfy a uniform lower density bound and a
	dimensionally critical summability condition on its mean curvature,
	the following statements hold. Firstly, continuous and compact
	embeddings of Sobolev spaces into Lebesgue spaces and spaces of
	continuous functions are available. Secondly, the geodesic distance
	associated to the varifold is a continuous, not necessarily H{\"o}lder
	continuous Sobolev function with bounded derivative.  Thirdly, if the
	varifold additionally has bounded mean curvature and finite measure,
	the present Sobolev spaces are isomorphic to those previously
	available for finite Radon measures yielding many new results for
	those classes as well.

	Suitable versions of the embedding results obtained for Sobolev
	functions hold in the larger class of generalised weakly
	differentiable functions.
\end{abstract}
\paragraph{MSC-classes 2010}
	46E35 (Primary); 49Q15, 53C22 (Secondary).

    
%%%%%%%%%%%%%%%%%%%%%%%%%%%%%%%%%%%%%%%%%%%%%%%%%%%%%%%%%%%%%%%%%%%%%%%%%%%%%%

\addcontentsline{toc}{section}{\numberline{}Introduction}
\section*{Introduction}
\subsection*{Overview}
The main purpose of this paper is to present a concept of first order Sobolev
functions on nonsmooth ``surfaces'' in Euclidean space with
arbitrary dimension and codimension arising in variational problems involving
the area functional.  The model for such surfaces are certain rectifiable
varifolds, see the general hypothesis below. This class is sufficiently broad
to include area minimising rectifiable currents, perimeter minimising
``Caccioppoli sets'', or typical time slices of ``singular'' mean curvature
flow as well as many surfaces occurring in mathematical models for natural
sciences, see \cite[p.\,2]{snulmenn:tv.v2}.

The envisioned concept should satisfy the following two requirements.
\begin{enumerate}
	\item \label{item:vectorspace} Sobolev functions on varifolds should
	give rise to Banach spaces.
	\item \label{item:properties} Sobolev functions on varifolds should
	share as many embedding estimates and structural results as possible
	with their Euclidean counterparts.
\end{enumerate}
This is accomplished by the present paper which thus provides the basis for
the study of divergence form, second order elliptic partial differential
equations on varifolds in their natural setting.  The new concept is based on
generalised weakly differentiable functions on varifolds introduced by the
author in \cite{snulmenn:tv.v2} along with an array of properties in the
spirit of \eqref{item:properties}.  However, \emph{generalised weakly
differentiable functions do not form a linear space}, hence they violate the
requirement \eqref{item:vectorspace} which is necessary for the use of almost
any standard tool from functional analysis.  The Sobolev spaces on varifolds
introduced here provide a way to overcome this difficulty.  They satisfy
\eqref{item:vectorspace} and, as subsets of the nonlinear space of generalised
weakly differentiable functions, satisfy \eqref{item:properties} as well.

Sobolev functions also provide a new toolbox for the study of the delicate
local connectedness properties of varifolds satisfying suitable conditions on
their first variation.  Understanding these properties is a key challenge in
any regularity consideration.  Local connectedness is analytically measured by
the degree to which control of the gradient of a function entails control on
its oscillation.  Basic estimates in this respect were provided by the Sobolev
Poincar{\'e} inequalities obtained in \cite[\S\,10]{snulmenn:tv.v2} and the
oscillation estimate of \cite[\S\,13]{snulmenn:tv.v2}.

The present paper contains three main contributions to the study of local
connectedness of varifolds satisfying a uniform lower density bound and
dimensionally critical summability condition on their mean curvature, see the
hypotheses below.  Firstly, a Rellich type embedding theorem for Sobolev
functions is proven which is related to the local connectedness structure of
the varifold through subtle oscillation estimates in its proof.  In fact, a
Rellich type embedding theorem is established for generalised weakly
differentiable functions in a significantly more general setting.  Secondly,
it is proven that the geodesic distance -- even if the space is incomplete --
is an example of a continuous Sobolev function with bounded derivative.  In
\cite[\S\,14]{snulmenn:tv.v2} it had only been proven that the geodesic
distance is a real valued function.  Thirdly, an example is constructed
showing that the geodesic distance may fail to be locally H{\"o}lder
continuous with respect to any exponent.  In particular, the embedding of
Euclidean Sobolev functions with suitably summable derivative into H{\"o}lder
continuous functions does not extend to the varifold case.

A distinctive feature of the presently developed theory of first order Sobolev
spaces is the key role played by the first variation of the varifold. The
latter carries information of the extrinsic geometry of the varifold,
considered as generalised submanifold, namely its generalised mean curvature
and its ``boundary''. This is in line with potential use of the present theory
in the study of regularity properties of varifolds and it distinguishes the
present approach from the completely intrinsic viewpoint of metric measure
spaces. The latter perspective is described in the books of Heinonen, see
\cite{MR1800917}, Bj{\"o}rn and Bj{\"o}rn, see \cite{MR2867756}, and Heinonen,
Koskela, Shanmugalingam, and Tyson, see \cite{MR3363168}.

\subsection*{Hypotheses} The notation is mainly that of Federer
\cite{MR41:1976} and Allard \cite{MR0307015}, see Section \ref{sec:notation}.
Throughout the paper footnotes recall some parts of this notation or point to
related terminology whenever that appeared to be desirable.

Firstly, a list of hypotheses relevant for the present theory will be drawn
up.
\begin{citing} [General hypothesis]
        Suppose $\vdim$ and $\adim$ are positive integers, $\vdim \leq \adim$,
        $U$ is an open subset of $\rel^\adim$, $V$ is an $\vdim$ dimensional
        rectifiable varifold in $U$ whose first variation $\delta V$ is
        representable by integration, and $Y$ is a finite dimensional normed
	vectorspace.%
	\begin{footnote}
		{To recall the notation concerning varifolds from Allard
		\cite[3.1, 3.5, 4.2, 4.3]{MR0307015}, first recall the
		following items from \cite{MR41:1976} and \cite{MR1777737}.
		\begin{itemize}
			\item The inner product of $x$ and $y$ is denoted $x
			\bullet y$, see \cite[1.7.1]{MR41:1976}.
			\item If $X$ is a locally compact Hausdorff space then
			$\mathscr{K} (X)$ is the vectorspace of all continuous
			real valued functions on $X$ with compact support, see
			\cite[2.5.14]{MR41:1976}.
			\item If $U$ is an open subset of some finite
			dimensional normed vectorspace and $Z$ is a Banach
			space then $\mathscr{D}(U,Z)$ denotes the space of all
			functions $\theta : U \to Z$ of class $\infty$,
			i.e.~``smooth'' functions, with compact support, see
			\cite[4.1.1]{MR41:1976}.
			\item Whenever $P$ is an $\vdim$ dimensional plane in
			$\rel^\adim$, the orthogonal projection of
			$\rel^\adim$ onto $P$ will be denoted by
			$\project{P}$, see Almgren
			\cite[T.1\,(9)]{MR1777737}.
		\end{itemize}
		Whenever $U$ is an open subset of $\rel^\adim$ an $\vdim$
		dimensional varifold $V$ is a Radon measure over $U \times
		\grass \adim \vdim$, where $\grass \adim \vdim$ denotes the
		space of $\vdim$ dimensional subspaces of $\rel^\adim$.  An
		$\vdim$ dimensional varifold $V$ in $U$ is called rectifiable
		if and only if there exist sequences of compact subsets $C_i$
		of $\vdim$ dimensional submanifolds $M_i$ of $U$ of class $1$
		and $0 < \lambda_i < \infty$ such that
		\begin{equation*}
			V(k) = \sum_{i=1}^\infty \lambda_i \tint{C_i}{} k
			(x,\Tan(M_i,x)) \ud \mathscr{H}^\vdim \, x \quad
			\text{for $k \in \mathscr{K} ( U \times \grass \adim
			\vdim )$},
		\end{equation*}
		where $\Tan (M_i,x)$ denotes the tangent space of $M_i$ at
		$x$.  The first variation $\delta V : \mathscr{D} ( U,
		\rel^\adim ) \to \rel$ of $V$ is defined by
		\begin{equation*}
			( \delta V ) ( \theta ) = \tint {}{} \project P
			\bullet \Der \theta (x) \ud V \, (x,P) \quad \text{for
			$\theta \in \mathscr{D} (U,\rel^\adim )$}.
		\end{equation*}
		The total variation $\| \delta V \|$ is largest Borel regular
		measure over $U$ satisfying
		\begin{equation*}
			\| \delta V \| ( G ) = \sup \{ ( \delta V ) (\theta)
			\with \text{$\theta \in \mathscr{D} (U,\rel^\adim)$,
			$\spt \theta \subset G$ and $| \theta | \leq 1$} \}
		\end{equation*}
		whenever $G$ is an open subset of $U$. The first variation
		$\delta V$ is representable by integration if and only if $\|
		\delta V \|$ is a Radon measure. If $\| \delta V \|$ is
		absolutely continuous with respect to $\| V \|$ then
		the generalised mean curvature vector of $V$,
		$\mathbf{h}(V,\cdot)$, is $\| V \|$ almost characterised
		amongst functions in $\Lploc 1 ( \| V \|, \rel^\adim )$ by the
		condition
		\begin{equation*}
			( \delta V ) ( \theta ) = - \tint{}{} \mathbf{h} (V,x)
			\bullet \theta (x) \ud \| V \| \, x \quad \text{for
			$\theta \in \mathscr{D} (U,\rel^\adim)$}.
		\end{equation*}}
	\end{footnote}
\end{citing}
Aspects of the theory involving the isoperimetric inequality are most
conveniently developed under the following density hypothesis.
\begin{citing} [Density hypothesis]
        Suppose $\vdim$, $\adim$, $U$, and $V$ are as in the general
        hypothesis and satisfy%
	\begin{footnote}
		{If $\mu$ measures a metric space $X$, $a \in X$, and $\vdim$
		is a positive integer then
		\begin{equation*}
			\density^\vdim ( \mu, a ) = \lim_{r \to 0+}
			\frac{\measureball \mu { \cball ar}}{\unitmeasure
			\vdim r^\vdim}, \quad \text{where $\unitmeasure \vdim
			= \measureball{\mathscr{L}^\vdim}{\cball 01}$}
		\end{equation*}
		and $\cball ar$ is the closed ball with centre $a$ and radius
		$r$, see \cite[2.7.16, 2.8.1, 2.10.19]{MR41:1976}.}
	\end{footnote}
        \begin{equation*}
                \density^\vdim ( \| V \|, x ) \geq 1 \quad \text{for $\| V \|$
                almost all $x$}.
        \end{equation*}
\end{citing}
Several theorems will also make use of the following additional hypothesis.
\begin{citing} [Mean curvature hypothesis]
        Suppose $\vdim$, $\adim$, $U$, and $V$ are as in the general
        hypothesis and satisfies the following condition.

	If $\vdim > 1$ then $\| \delta V \|$ is absolutely continuous with
	respect to $\| V \|$, the function $\mathbf{h} (V,\cdot)$ belongs to
	$\Lploc \vdim (\| V \|, \rel^\adim )$,%
	\begin{footnote}
		{The spaces $\Lp p ( \mu, Y )$ and $\Lploc p ( \mu, Y)$
		contain functions rather than equivalence classes of
		functions.}
	\end{footnote}
	and $\psi$ is Radon measure over $U$ such that
	\begin{equation*}
		\psi (A) = \tint{A}{} |\mathbf{h}(V,x)|^\vdim \ud \| V \| \, x
		\quad \text{whenever $A$ is a Borel subset of $U$}.
	\end{equation*}
\end{citing}
The density hypothesis and the mean curvature hypothesis will be referred to
whenever they shall be in force.

\subsection*{Known results} As the present paper extends the author's paper
\cite{snulmenn:tv.v2}, it seems expedient to review those results of that
paper most relevant for the present development.

\subsubsection*{Axiomatic approach to Sobolev spaces} Given an open subset of
Euclidean space and a finite dimensional normed vectorspace $Y$, the class of
weakly differentiable $Y$ valued functions is clearly closed under addition
and composition with members of $\mathscr{D} ( Y, \rel )$ and any $Y$ valued
function constant on connected components belongs to that class. Replacing the
decomposition into connected components by the decomposition of a varifold in
the sense of \cite[6.9]{snulmenn:tv.v2}, one may formulate the following list
of desirable properties for a concept of weakly differentiable functions or
Sobolev functions on a varifold.

\begin{enumerate}[label=(\Roman*)]
        \item \label{item:addition} The class is closed under addition.
	\item \label{item:truncation} The class is closed under composition
	with members of $\mathscr{D}(Y,\rel)$.
	\item \label{item:decomposition} Each appropriately summable function
	which is constant on the components of \emph{some} decomposition of
	the varifold belongs to the class.
\end{enumerate}

Owing to the fact that decompositions of varifolds are nonunique, see
\cite[6.13]{snulmenn:tv.v2}, one can show that it is impossible to realise all
three properties in a single satisfactory concept, see
\cite[8.28]{snulmenn:tv.v2}. Accordingly, three concepts have been developed,
two in \cite{snulmenn:tv.v2} and one in the present paper, each missing
precisely one distinct one of the above three properties.

\subsubsection*{Integration by parts identity} The seemingly most natural way
to define a concept of weak differentiability is to employ the fact that the
first variation $\delta V$ is representable by integration to formulate an
integration by parts identity.

\begin{intro_definition} [see \protect{\cite[8.27]{snulmenn:tv.v2}}]
	Suppose $\vdim$, $\adim$, $U$, $V$, and $Y$ are as in the general
	hypothesis.

	Then the class $\mathbf{W} ( V, Y )$ is defined to consist of all $f
	\in \Lploc{1} ( \| V \| + \| \delta V \|, Y)$ such that for some $F
	\in \Lploc{1} ( \| V \|, \Hom ( \rel^\adim, Y ) )$ there holds
        \begin{equation*}
		( \delta V ) ( ( \alpha \circ f ) \theta ) = \alpha \big (
		\tint{}{} ( \project{P} \bullet \Der  \theta (x)) f (x) + F(x)
		(\theta(x)) \ud V \, (x,P) \big )
        \end{equation*}
        whenever $\theta \in \mathscr{D} (U, \rel^\adim )$ and $\alpha \in
        \Hom ( Y, \rel )$.
\end{intro_definition}

The function $F$ is $\| V \|$ almost unique and could act as weak derivative
of~$f$ with respect to $V$. One readily verifies that this class satisfies
\ref{item:addition} and  \ref{item:decomposition}. However, it fails to
satisfy \ref{item:truncation}, see \cite[8.27]{snulmenn:tv.v2}. Moreover, it
may happen that $f \in \mathbf{W}(V,\rel)$ has zero weak derivative but the
distributional $V$ boundaries $\boundary V E(y) : \mathscr{D} ( U, \rel^\adim
) \to \rel$ of the superlevel sets $ E(y) = \{ x \with f(x) > y \}$ satisfy
\begin{equation*}
	\mathscr{L}^1 \big ( \rel \cap \{ y \with \boundary V E(y) \neq 0 \}
	\big ) > 0,
\end{equation*}
see \cite[8.32]{snulmenn:tv.v2}. Consequently, no coarea formula
analogous to that for weakly differentiable functions in Euclidean space, see
\cite[4.5.9\,(13)]{MR41:1976}, may be formulated in this class.  These two
facts seem to pose a serious obstacle to the development of a satisfactory
theory for the class $\mathbf{W} (V,Y)$.

\subsubsection*{Generalised weakly differentiable functions} To overcome this
difficulty, one may modify the integration by parts identity by requiring it
to hold also for compositions with a class of nonlinear functions.
\begin{intro_definition} [see \protect{\cite[8.3]{snulmenn:tv.v2}}]
	Suppose $\vdim$, $\adim$, $U$, $V$, and $Y$ are as in the general
	hypothesis.

	Then a $Y$ valued $\| V \| + \| \delta V \|$ measurable%
	\begin{footnote}
		{If $\mu$ measures $X$ and $f$ maps a subset of $X$ into a
		topological space $Y$, then $f$ is $\mu$ measurable if and
		only if $\mu ( X \without \dmn f ) = 0$ and the preimage of
		every open subset of $Y$ under $f$ is $\mu$ measurable, see
		\cite[2.3.2]{MR41:1976}.}
	\end{footnote}
	function $f$ with domain contained in $U$ is called \emph{generalised
	$V$ weakly differentiable} if and only if for some $\| V \|$
	measurable $\Hom (\rel^\adim, Y )$ valued function $F$ the following
	two conditions hold:
	\begin{enumerate}
		\item If $K$ is a compact subset of $U$ and $0 \leq s <
		\infty$, then
		\begin{equation*}
			\tint{\classification{K}{x}{|f(x)|\leq s}}{} \|F\| \ud
			\| V \| < \infty.
		\end{equation*}
		\item If $\theta \in \mathscr{D} ( U, \rel^\adim )$, $\gamma
		\in \mathscr{E} ( Y, \rel)$ and $\spt \Der \gamma$ is compact,
		then%
		\begin{footnote}
			{If $U$ is an open subset of some finite dimensional
			normed vectorspace and $Z$ is Banach space then
			$\mathscr{E} (U,Z)$ denotes the vectorspace of
			functions $\theta : U \to Z$ of class $\infty$.}
		\end{footnote}
		\begin{gather*}
			\begin{aligned}
				& ( \delta V ) ( ( \gamma \circ f ) \theta ) \\
				& \quad = \tint{}{} \gamma(f(x)) \project{P}
				\bullet \Der  \theta (x) \ud V \, (x,P) +
				\tint{}{} ( \Der \gamma (f(x)) \circ F (x)) (
				\theta(x)) \ud \| V \| \, x.
			\end{aligned}
		\end{gather*}
	\end{enumerate}
	The set of all $Y$ valued generalised $V$ weakly differentiable
	functions will be denoted by $\trunc ( V, Y )$.
\end{intro_definition}
The function $F$ is $ \| V \|$ almost unique. Accordingly, the
\emph{generalised $V$ weak derivative of $f$}, denoted by $\derivative{V}{f}$,
may be defined to equal a particular such $F$ characterised by an approximate
continuity condition, see \cite[8.3]{snulmenn:tv.v2}.

This class has a favourable behaviour under truncation and composition as well
as decomposition of the underlying varifold, see \cite[8.12, 8.13, 8.15, 8.16,
8.18, 8.24]{snulmenn:tv.v2}. In particular, it satisfies properties
\ref{item:truncation} and \ref{item:decomposition}. In case $\vdim = \adim$,
$U = \rel^\adim$, $\| V \| = \mathscr{L}^\adim$, and $Y = \rel$ a function $f$
belongs to $\trunc (V,Y)$ if and only if the truncated functions $f_s : \dmn f
\to Y$ defined by
\begin{equation*}
	f_s(x)=f(x) \quad \text{if $|f(x)| \leq s$}, \qquad f_s (x) = (\sign
	f(x)) s \quad \text{if $|f(x)|>s$}
\end{equation*}
for $x \in \dmn f$ and $0 < s < \infty$, are weakly differentiable in the
classical sense, see \cite[8.19]{snulmenn:tv.v2}, and the subclass
\begin{equation*}
	\trunc (V,Y) \cap \Lploc 1 ( \| V \| + \| \delta V \|, Y) \cap \big \{
	f \with \derivative Vf \in \Lploc 1 ( \| V \|, \Hom ( \rel^\adim, Y )
	\big \}
\end{equation*}
equals the usual space of weakly differentiable functions, see
\cite[8.18]{snulmenn:tv.v2}. However, considering the varifold associated to
three lines in $\rel^2$ meeting at a common point at equal angles shows that
the indicated subclass need not to be closed with respect to addition, see
\cite[8.25]{snulmenn:tv.v2}.  In particular, it does not have property
\ref{item:addition}.  Of course, the class $\trunc (V,\rel)$ itself need not
to be closed under addition even in case of Lebesgue measure, see B{\'e}nilan,
Boccardo, Gallou{\"e}t, Gariepy, Pierre, and Vazquez \cite[p.\,245]{MR1354907}.
This drawback is partially compensated by the fact that the class $\trunc
(V,Y)$ is closed under addition of a locally Lipschitzian function, see
\cite[8.20\,(3)]{snulmenn:tv.v2}.

Whenever $G$ is a relatively open subset of $\Bdry U$, one may also realise
the concept of ``zero boundary values on $G$'' for nonnegative functions $f$
in $\trunc (V,\rel)$ by means of the class $\trunc_G (V)$, see
\cite[9.1]{snulmenn:tv.v2}. This class has good properties under composition
and convergence of the functions in measure with appropriate bounds on the
derivatives, see \cite[9.9, 9.13, 9.14]{snulmenn:tv.v2}. Of course, instead of
restricting to nonnegative functions, one could also consider the class
\begin{equation*}
	\trunc (V,Y) \cap \{f \with |f| \in \trunc_G (V) \}
\end{equation*}
but stability under compositions would fail in this case, see \cite[9.10,
9.11]{snulmenn:tv.v2}.

The more elaborate properties of $\trunc (V,Y)$ build on the isoperimetric
inequality which works most effectively under the density hypothesis. This
hypothesis allows for the formulation of various Sobolev Poincar{\'e} type
inequalities with and without boundary condition, see \cite[10.1, 10.7,
10.9]{snulmenn:tv.v2}. Furthermore, pointwise differentiability results both
of approximate and integral nature then hold for generalised $V$ weakly
differentiable functions, see \cite[11.2, 11.4]{snulmenn:tv.v2}.

Turning to some relevant results concerning the local connectedness structure
of varifolds, the mean curvature hypothesis becomes more relevant. Under the
density hypothesis and the mean curvature hypothesis, the connected components
of $\spt \| V \|$ are relatively open and any two points belonging to the same
connected component may be connected by a path of finite length whose image
lies in that component, see \cite[6.14\,(3), 14.2]{snulmenn:tv.v2}. At the
heart of the proof of the second part of this assertion lies an oscillation
estimate for an a~priori continuous generalised weakly differentiable function
whose derivative satisfies a $q$-th power summability hypothesis with $q >
\vdim$, see \cite[13.1]{snulmenn:tv.v2}. This estimate differs in two points
from the well known oscillation estimate for weakly differentiable functions
in Euclidean space. Firstly, the function needs to be continuous a~priori;
otherwise -- in view of property \ref{item:decomposition} -- a counterexample
is immediate from considering two crossing lines. Secondly, the estimate does
not yield H{\"o}lder continuity; in fact, an example showing that H{\"ol}der
continuity is not implied by those hypotheses will be constructed in the
present paper, see \hyperlink{Thm_C}{Theorem~C} and
\hyperlink{Thm_D}{Theorem~D} below.

\subsection*{Results of the present paper} \label{subsec:present-results}

The main contributions, apart from introducing the concept of Sobolev
functions on varifolds, are the five named theorems,
\hyperlink{Thm_A}{Theorem~A} to \hyperlink{Thm_E}{Theorem~E}, and its two
corollaries, \hyperlink{Cor_A}{Corollary~A} and
\hyperlink{Cor_B}{Corollary~B}, which will be described below.  In order to
complete the picture, further theorems are included which essentially follow
from combining the present theory with that of \cite{snulmenn:tv.v2}.

\subsubsection*{Approximation of Lipschitzian functions, see Section 3} The
Sobolev spaces on varifold will be defined by a completion procedure starting
from locally Lipschitzian functions.  As a consequence of the next result,
``smooth'' functions  are dense in these spaces for finite exponents, see
\hyperlink{Cor_A}{Corollary~A}.

{ \hypertarget{Thm_A}{}
\begin{citing} [Theorem~A, see \ref{corollary:approximation_lip} and \ref{remark:approximation_lip}]
	Suppose $\vdim$ and $\adim$ are positive integers, $\vdim \leq \adim$,
	$U$ is an open subset of $\rel^\adim$, $V$ is an $\vdim$ dimensional
	rectifiable varifold in $U$, $Y$ is a finite dimensional normed
	vectorspace, $K$ is a compact subset of $U$, and $f: U \to Y$ is a
	Lipschitzian function with $\spt f \subset \Int K$.

	Then there exists a sequence $f_i \in \mathscr{D} (U, Y)$ satisfying
	\begin{gather*}
		f_i(x) \to f(x) \quad \text{uniformly for $x \in \spt \| V \|$
		as $i \to \infty$}, \\
		\big \| ( \| V \|, \vdim ) \ap \Der (f_i-f) \big \| \to 0
		\quad \text{in $\| V \|$ measure as $i \to \infty$}, \\
		\spt f_i \subset K \quad \text{for each $i$}, \qquad
		\limsup_{i \to \infty} \Lip f_i \leq \Gamma \Lip f,
	\end{gather*}
	where $\Gamma$ is a positive, finite number depending only on $Y$.%
	\begin{footnote}
		{Suppose $\mu$ measures an open subset $U$ of a normed
		vectorspace $X$, $a \in U$, and $m$ is a positive integer.
		Then $\Tan^m ( \mu, a)$ denotes the closed cone of
		\emph{$(\mu,m)$ approximate tangent vectors} at $a$ consisting
		of all $u \in X$ such that
	        \begin{gather*}
			\density^{\ast m} ( \mu \restrict \mathbf{E}
			(a,u,\varepsilon), a ) > 0 \quad \text{for every
			$\varepsilon > 0$}, \\
			\text{where $\mathbf{E}(a,u,\varepsilon) = X \cap \{ x
			\with \text{$|r(x-a) -u| < \varepsilon$ for some $r >
			0$} \}$}.
	        \end{gather*}
		Moreover, if $f$ maps a subset of $X$ into another normed
		vectorspace $Y$, then $f$ is called \emph{$(\mu,m)$
		approximately differentiable} at $a$ if and only if there
		exist $b \in Y$ and a continuous linear map $L : X \to Y$ such
		that
	        \begin{equation*}
			\density^\vdim ( \mu \restrict X \without \{ x \with
			|f(x)-b-L (x-a)| \leq \varepsilon |x-a| \}, a ) = 0
			\quad \text{for every $\varepsilon > 0$}.
	        \end{equation*}
		In this case $L | \Tan^\vdim ( \mu, a )$ is unique and it is
		called the \emph{$(\mu,m)$ approximate differential} of $f$ at
		$a$, denoted $(\mu,m) \ap \Der f(a)$, see
		\cite[3.2.16]{MR41:1976}.}
	\end{footnote}
	Moreover, if $Y = \rel^l$, then one may take $\Gamma = 1$.
\end{citing} }

As there is no hypothesis on $\delta V$, the only available notion of
derivative is that of approximate derivative.  If $V$ satisfies the general
hypothesis, then $( \| V \|, \vdim ) \ap \Der (f_i-f)$ can be replaced by
$\derivative V{(f-f_i)}$, see \cite[8.7]{snulmenn:tv.v2}.

It is instructive to compare \hyperlink{Thm_A}{Theorem~A} with the familiar
fact that there exists a sequence of continuously differentiable functions
$g_i : U \to Y$ such that \begin{equation*} \| V \| ( U \without \{ x \with
\text{$f(x)=g_i(x)$ and $\ap \Der f(x) = \ap \Der g_i(x)$} \}) \to 0 \quad
\text{as $i \to \infty$}, \end{equation*} where ``$\ap$'' refers to $( \| V
\|, \vdim )$ approximate differentiation.  The difficulty to construct the
asserted functions $f_i$ from the functions $g_i$ is that agreement of the
approximate derivatives of $f$ and $g_i$ at $x$ only implies \begin{equation*}
\big \| \Der g_i (x) | \Tan^\vdim ( \| V \|, x ) \big \| \leq \Lip f
\end{equation*} but $\| \Der g_i(x)\|$ may be much larger than $\Lip f$.  This
is resolved by employing a special retraction onto continuously differentiable
submanifolds of $\rel^\adim$, see \ref{corollary:special-retraction}.

\subsubsection*{A Rellich type embedding theorem, see Section 4}

The Rellich type embedding theorem for Sobolev functions on varifolds, see
\hyperlink{Cor_B}{Corollary~B}, will be derived as a consequence of the
following significantly more general theorem for generalised weakly
differentiable functions on varifolds.

{ \hypertarget{Thm_B}{}
\begin{citing} [Theorem~B, see \ref{thm:rellich-in-measure}]
	Suppose $\vdim$, $\adim$, $U$, $V$, and $Y$ are as in the general
	hypothesis and the density hypothesis, and $f_i \in \trunc (V,Y)$ is a
	sequence satisfying
	\begin{gather*}
		\lim_{t \to \infty} \sup \big \{ \| V \| ( K \cap \{ x \with
		|f_i (x) | > t \} ) \with i = 1, 2, 3, \ldots \big \} = 0, \\
		\sup \big \{ \tint{K \cap \{ x \with |f_i(x)| \leq t \}}{} \|
		\derivative V{f_i} \| \ud \| V \| \with i = 1, 2, 3, \ldots
		\big \} < \infty \quad \text{for $0 \leq t < \infty$}
	\end{gather*}
	whenever $K$ is a compact subset of $U$.

	Then there exist a $\| V \|$ measurable $Y$ valued function $f$ and a
	subsequence of $f_i$ which, whenever $K$ is a compact subset of $U$,
	converges to $f$ in $\| V \| \restrict K$ measure.
\end{citing} }

As a consequence of \hyperlink{Thm_B}{Theorem~B} one obtains sequential
closedness results under weak convergence of both functions and derivative for
$\trunc (V,Y)$ and $\trunc_G (V)$, see \ref{corollary:closedness-tv} and
\ref{corollary:closedness_tg}.  Such results are of natural importance in
considering variational problems in these nonlinear spaces.

\subsubsection*{Definitions of Sobolev spaces, see Section 5} Next, the
definitions of the Sobolev spaces on varifolds and some basic properties that
are direct consequences of the theory of \cite{snulmenn:tv.v2} shall be
presented.  Firstly, the largest classes, that is the local Sobolev spaces,
will be defined.

\begin{intro_definition} [local Sobolev space, see \ref{def:strict_local_sobolev_space}
	and \ref{remark:strict_local_sobolev_space_inclusion}] Suppose
	$\vdim$, $\adim$, $U$, $V$, and $Y$ are as in the general hypothesis
	and $1 \leq q \leq \infty$.

	Then the \emph{local Sobolev space with respect to $V$ and exponent
	$q$}, denoted by $\SWloc{q} ( V , Y )$, is defined to be the class of
	all $f \in \trunc (V,Y)$ such that
	\begin{equation*}
		(f,\derivative Vf) \in \Clos \big \{ (g,\derivative Vg) \with
		\text{$g \in Y^U$ and $g$ is locally Lipschitzian} \big \},
	\end{equation*}
	where the closure is taken in $\Lploc q(\| V \|+\| \delta V \|, Y )
	\times \Lploc q ( \| V \|, \Hom ( \rel^\adim, Y ) )$.%
	\begin{footnote}
		{For any $Y$ and $U$, the expression $Y^U$ denotes the class
		of all functions $g : U \to Y$.}
	\end{footnote}
\end{intro_definition}
The letter ``$\mathbf{H}$'' is chosen to emphasise the fact that the space is
defined by a closure operation, see \cite[3.2]{MR2424078}, and the placement
of~$q$ as subscript is in line with the symbol $\Lp q ( \mu, Y )$ employed for
Lebesgue spaces, see \cite[2.4.12]{MR41:1976}.

It is important that $\SWloc q (V,Y)$ is a subset of the possibly nonlinear
space $\trunc (V,Y)$ but itself is a vectorspace; in fact, it is a complete
locally convex space when endowed with its natural topology resulting from its
inclusion into
\begin{equation*}
	\Lploc q(\| V \|+\| \delta V \|, Y ) \times \Lploc q ( \| V \|, \Hom (
	\rel^\adim, Y ) ),
\end{equation*}
see \ref{def:topology-SWSob} and \ref{remark:completeness_wqloc}.  Moreover,
it has good stability properties under composition, see
\ref{remark:strict_local_sobolev_space}\,\eqref{item:strict_local_sobolev_space:composition}\,\eqref{item:strict_local_sobolev_space:truncation}.
Therefore $\SWloc q (V,Y)$ has properties \ref{item:addition} and
\ref{item:truncation} but it does not have property \ref{item:decomposition}
as may be seen from considering two crossing lines, see
\ref{remark:two-crossing-lines}.

At first sight, one might consider to replace $\Lploc q ( \| V \| + \| \delta
V \|, Y )$ in the above definition by $\Lploc q ( \| V \|, Y )$. Indeed,
assuming the density hypothesis, the mean curvature hypothesis, and $\vdim >
1$, the same definition would have resulted, see \ref{thm:replacement-w1q-loc}
and \ref{remark:replacement-w1q-loc}. But in the general case it appears
natural (and indispensable) to require control with respect to the measure $\|
\delta V \|$ since both the integration by parts identities and the
isoperimetric inequalities involve control of the function with respect to $\|
\delta V \|$.

\begin{intro_definition} [Sobolev space, see \ref{def:sobolev-seminorm} and
\ref{def:strict_sobolev_space}]
	Suppose $\vdim$, $\adim$, $U$, $V$, and $Y$ are as in the general
	hypothesis and $1 \leq q \leq \infty$.

	Then define the \emph{Sobolev space with respect to $V$ and
	exponent $q$} by%
	\begin{footnote}
		{If $\mu$ measures $X$, $1 \leq q \leq \infty$, and $f$ is a
		$\mu$ measurable function with values in some Banach space
		$Y$, then one defines (see \cite[2.4.12]{MR41:1976})
	        \begin{gather*}
			\mu_{(q)}(f) = ( \tint{}{} |f|^q \ud \mu )^{1/q} \quad
			\text{in case $1 \leq q < \infty$}, \\
			\mu_{(\infty)}(f) = \inf \big \{ s \with \text{$s \geq
			0$, $\mu ( \{ x \with |f(x)| > s \} ) = 0$} \big \}.
	        \end{gather*}}
	\end{footnote}
	\begin{align*}
		& \SWSob{q} (V,Y ) = \classification{\SWloc{q} (V,Y)}{f}{
		\SWnorm{q}{V}{f} < \infty}, \\
		& \qquad \text{where $\SWnorm{q}{V}{f} = \eqLpnorm{\| V \|+\|
		\delta V \|}{q}{f} + \Lpnorm{\| V \|}{q}{ \derivative{V}{f} }$
		for $f \in \trunc (V,Y)$}.
	\end{align*}
\end{intro_definition}
The usage of the letter ``$\mathbf{H}$'' in the name of the seminorm $\SWnorm
qV\cdot | \SWloc q(V,Y)$ is modelled on the usage of the letter
``$\mathbf{F}$'' in the name of the seminorm $\mathbf{F}_K$ related to the
space of flat chains $\mathbf{F}_{m,K} (U )$, see \cite[4.1.12]{MR41:1976}.

The space $\SWSob q (V,Y)$ is $\SWnorm q V \cdot | \SWSob q (V,Y)$ complete,
see \ref{remark:strict_sobolev_space}.
\begin{intro_definition} [Sobolev space with ``zero boundary values'', see
\ref{def:strict-sobolev-space-zero-boundary-values}]
	Suppose $\vdim$, $\adim$, $U$, $V$, and $Y$ are as in the general
	hypothesis and $1 \leq q \leq \infty$.

	Then define $\SWzero{q} (V,Y )$ to be the $\SWnorm qV\cdot | \SWSob
	q(V,Y)$ closure of
	\begin{equation*}
		Y^U \cap \{ g \with \text{$\Lip g< \infty$, $\spt g$ is
		compact} \}
	\end{equation*}
	in $\SWSob q(V,Y)$.
\end{intro_definition}
The space $\SWzero q (V,Y)$ is $\SWnorm q V \cdot | \SWzero q (V,Y)$ complete,
see \ref{remark:zero_sobolev_space}.
If $U = \rel^\adim$ and $q < \infty$, then
$\SWzero q(V,Y) = \SWSob q (V,Y)$, see \ref{remark:zero_sobolev_euclid_space}.

The inclusions amongst the various local spaces are given by
\begin{align*}
	& Y^U \cap \{ g \with \text{$g$ locally Lipschitzian} \} \subset
	\SWloc q (V,Y) \\
	& \qquad \subset \trunc (V,Y) \cap \Lploc 1 ( \| V \| + \| \delta V
	\|, Y) \cap \big \{ f \with \derivative Vf \in \Lploc 1 ( \| V \|,
	\Hom ( \rel^\adim, Y ) \big \} \\
	& \qquad \subset \mathbf{W}(V,Y),
\end{align*}
see \ref{remark:strict_local_sobolev_space_inclusion} and
\cite[8.27]{snulmenn:tv.v2}. Finally, one may also consider the quotient of
$\SWloc q (V,Y)$ or $\SWloc q (V,Y)$ or $\SWzero q (V,Y)$ by
\begin{align*}
	& \SWloc q (V,Y) \cap \{ f \with \text{$f(x)=0$ for $\| V \|$ almost
	all $x$} \} \\
	& \qquad
	\begin{aligned}
		= \SWloc q (V,Y) \cap \big \{ f \with & \, \text{$f(x)=0$ for
		$\| V \| + \| \delta V \|$ almost all $x$}, \\
		& \, \text{$\derivative Vf(x)=0$ for $\| V \|$ almost all $x$}
		\big \},
	\end{aligned}
\end{align*}
see \ref{remark:eq_classes}, \ref{remark:swloc-quotient},
\ref{remark:quotient-sobolev-space}, and
\ref{remark:quotient-zero-sobolev-space}, which allows to conveniently apply
certain functional analytic results.

\subsubsection*{Basic theorems for Sobolev spaces, see Section 5} Having
\hyperlink{Thm_A}{Theorem~A} at one's disposal, the following density result
may be deduced analogously to the Euclidean case.

{ \hypertarget{Cor_A}{}
\begin{citing} [Corollary~A, see
\ref{remark:basic-top-swloc}\,\eqref{item:basic-top-swloc:q<oo_dense},
\ref{remark:density_sobolev_space}, and
\ref{remark:density-in-SWzero}]
	Suppose $\vdim$, $\adim$, $U$, $V$, and $Y$ are as in the general
	hypothesis, and $1 \leq q < \infty$.

	Then the following three statements hold.
	\begin{enumerate}
		\item The set $\mathscr{D}(U,Y)$ is dense in $\SWloc q (V,Y)$.
		\item The set $\SWSob q (V,Y) \cap \mathscr{E} (U,Y)$
		is $\SWnorm qV\cdot | \SWSob q (V,Y)$ dense in $\SWSob q
		(V,Y)$.
		\item The set $\mathscr{D}(U,Y)$ is $\SWnorm qV\cdot |
		\SWzero q (V,Y)$ dense in $\SWzero q (V,Y)$.
	\end{enumerate}
\end{citing} }

Evidently, the Sobolev space $\SWloc q (V,Y)$ is contained in $\trunc (V,Y)$.
The analogous statement involving ``zero boundary values'' is less obvious.

\begin{citing} [Theorem, see \ref{thm:zero_implies_zero}]
	Suppose $\vdim$, $\adim$, $U$, $V$, and $Y$ are as in the general
	hypothesis, $1 \leq q \leq \infty$, and $f \in \SWzero q (V,Y)$.

	Then $|f| \in \trunc_{\Bdry U} (V)$.
\end{citing}

Consequently, the Sobolev inequalities of \cite[10.1\,(2)]{snulmenn:tv.v2}
apply in the case of Sobolev functions as well, see
\ref{corollary:sob_poin_summary} and \ref{thm:sob_poin_summary}.

\subsubsection*{Geodesic distance, see Section 6} Apart of the envisioned use
of Sobolev functions for certain elliptic partial differential equations on
varifolds, Sobolev functions also occur naturally in the study of the geodesic
distance on the support of the weight measure of a varifold.

{ \hypertarget{Thm_C}{}
\begin{citing} [Theorem C, see \ref{thm:intrinsic_metric}] Suppose $\vdim$, $\adim$, $U$, and $V$ are as in the density
	hypothesis and the mean curvature hypothesis, $X = \spt \| V \|$, $X$
	is connected, $\varrho$~is the geodesic distance on $X$, see
	\ref{miniremark:intrinsic_metric}, and $W \in \Var_{2\vdim} ( U \times
	U )$ satisfies
   	\begin{equation*}
		W(k) = \tint{}{} k ((x_1,x_2),P_1\times P_2) \ud ( V \times V
		)\, ((x_1,P_1),(x_2,P_2))
	\end{equation*}
	whenever $k \in \mathscr{K} ( U \times U, \grass{ \rel^\adim \times
	\rel^\adim}{2 \vdim} )$.

	Then the following two statements hold.
	\begin{enumerate}
		\item The function $\varrho$ is continuous, a metric on $X$,
		and belongs to $\SWloc q (W,\rel)$ for $1 \leq q < \infty$
		with
		\begin{equation*}
			| \langle (u_1,u_2), \derivative W \varrho (x_1,x_2)
			\rangle | \leq | u_1 | + | u_2 | \quad \text{whenever
			$u_1,u_2 \in \rel^\adim$}
		\end{equation*}
		for $\| W \|$ almost all $(x_1,x_2)$.
		\item If $a \in X$, then $\varrho(a,\cdot) \in \SWloc q
		(V,\rel)$ for $1 \leq q < \infty$ and
		\begin{equation*}
			| \derivative V{(\varrho(a,\cdot))}(x)| = 1 \quad
			\text{for $\| V \|$ almost all $x$}.
		\end{equation*}
	\end{enumerate}
\end{citing} }

The previous result from \cite[14.2]{snulmenn:tv.v2} only showed that the
function $\varrho$ is real valued.  The sharpness of the preceding theorem is
illustrated by an example whose properties are summarised in the next theorem.

{ \hypertarget{Thm_D}{}
\begin{citing} [Theorem~D, see \ref{example:geodesic_distance}]
	There exist $\vdim$, $\adim$, $U$, and $V$ satisfying the density
	hypothesis and the mean curvature hypothesis and $a \in \spt \| V \|$
	such that the geodesic distance $\varrho$ on $\spt \| V \|$, see
	\ref{miniremark:intrinsic_metric}, has the following two properties.
	\begin{enumerate}
		\item The function $\varrho (a,\cdot)$ does not belong to
		$\SWloc \infty (V,\rel)$.
		\item The function $\varrho(a,\cdot)$ is not H{\"o}lder
		continuous with respect to any exponent.
	\end{enumerate}
\end{citing} }

\subsubsection*{Further theorems on Sobolev spaces, see Section 7}

All further results focus on the case when the density hypothesis and the mean
curvature hypothesis are satisfied.

{ \hypertarget{Cor_B}{}
\begin{citing} [Corollary~B, see
\ref{thm:rellich_embedding}\,\eqref{item:rellich_embedding:m>1} and \ref{remark:rellich_embedding}]
	Suppose $\vdim$, $\adim$, $U$, $V$, and $\psi$ are as in the density
	hypothesis and the mean curvature hypothesis, $\| V \| ( U ) <
	\infty$, $\Lambda = \Gamma_{\textup{\cite[10.1]{snulmenn:tv.v2}}} (
	\adim )$, $\psi ( U ) \leq \Lambda^{-1}$, $1 \leq q < \vdim$, $1 \leq
	\alpha < \vdim q/(\vdim-q)$, and $Y$ is a finite dimensional normed
	vectorspace.

	Then any sequence $f_i \in \SWzero q(V,Y)$ with
	\begin{equation*}
		\sup \big \{ \Lpnorm {\| V \|} q{\derivative V{f_i}} \with i =
		1, 2, 3, \ldots \big \} < \infty
	\end{equation*}
	admits a subsequence converging in $\Lp \alpha ( \| V \|, Y)$.
\end{citing} }

The smallness condition on $\psi (U)$ ensures that $\SWzero q(V,Y)$ does not
contain nontrivial functions with vanishing derivative as would be the case if
$V$ corresponded to a sphere for example. Moreover, a similar result holds for
$\vdim = 1$, see
\ref{thm:rellich_embedding}\,\eqref{item:rellich_embedding:m=1} and
\ref{remark:rellich_embedding}. Both results have an analogous formulation in
the space $\SWloc q (V,Y)$, see \ref{thm:rellich-local-embedding} and
\ref{remark:rellich-local-embedding}.

Next, an embedding theorem into the space $\mathscr{C} (\spt \| V \|, Y )$ of
continuous functions from $\spt \| V \|$ into $Y$ endowed with the topology of
locally uniform convergence, see
\ref{def:space-of-continuous-functions}, is stated.

\begin{citing} [Theorem, see \ref{thm:sob_continuous_emb} and
	\ref{remark:rellich-local-embedding}] Suppose $\vdim$, $\adim$, $U$,
	$V$, and $\psi$ are as in the density hypothesis and the mean
	curvature hypothesis, $1 < \vdim < q$, and $Y$ is a finite dimensional
	normed vectorspace.

	Then there exists a continuous linear map $L :\SWloc q (V,Y) \to
	\mathscr{C}( \spt \| V \|,Y)$ uniquely characterised by
	\begin{equation*}
		L(f)(x) = f(x) \quad \text{for $\| V \|$ almost all $x$}.
	\end{equation*}
	Moreover, if $f_i \in \SWloc q (V,Y)$ form a sequence satisfying
	\begin{equation*}
		\sup \big \{ \eqLpnorm{\| V \| \restrict K} q{f_i} +
		\eqLpnorm{\| V \| \restrict K}q {\derivative V{f_i}} \with i
		= 1, 2, 3, \ldots \big \} < \infty
	\end{equation*}
	whenever $K$ is a compact subset of $U$, then the sequence $L(f_i)$
	admits a subsequence converging in $\mathscr{C}( \spt \| V \|, Y )$.
\end{citing}

In view of \hyperlink{Thm_C}{Theorem~C} and \hyperlink{Thm_D}{Theorem~D}
concerning the geodesic distance to a point, it follows that, in contrast to
the Euclidean case, $L(f)$ need not to be H{\"o}lder continuous with respect
to any exponent, see \ref{remark:sob_continuous_emb}.

Finally, a subspace of $\SWloc q(V,Y)$ is considered which is defined by a
seminorm not involving $\| \delta V \|$ and hence more closely resembles the
Euclidean case.

{ \hypertarget{Thm_E}{}
\begin{citing} [Theorem~E, see
\ref{thm:replacement-w1q}\,\eqref{item:replacement-w1q:complete}\,\eqref{item:replacement-w1q:q<oo_dense}\,\eqref{item:replacement-w1q:equal}]
	Suppose $\vdim$, $\adim$, $U$, $V$, and $\psi$ are as in the density
	hypothesis and the mean curvature hypothesis, $1 \leq q \leq \infty$,
	$Y$ is a finite dimensional normed vectorspace,
	\begin{equation*}
		\sigma (f) = \Lpnorm{\| V \|} q f + \Lpnorm{\| V \|} q
		{\derivative Vf} \quad \text{for $f \in \SWloc q ( V,Y )$},
	\end{equation*}
	and $E = \SWloc q (V,Y) \cap \{ f \with \sigma (f) < \infty \}$.

	Then the following three statements hold.
	\begin{enumerate}
		\item \label{item:intro_complete} The vectorspace $E$ is
		$\sigma$ complete.
		\item \label{item:intro_dense} If $q < \infty$, then
		$\mathscr{E} (U, Y) \cap \{ f \with \sigma (f) < \infty \}$ is
		$\sigma$ dense in $E$, and if additionally $U = \rel^\adim$,
		then $\mathscr{D} (\rel^\adim,Y)$ is $\sigma$ dense in $E$.
		\item \label{item:intro_agree} If $U = \rel^\adim$ and $\psi (
		\rel^\adim ) < \infty$, then $E = \SWSob q (V,Y)$.
	\end{enumerate}
\end{citing} }

\eqref{item:intro_complete} is simple unless $\vdim = 1$ in which case the non
absolutely continuous part of $\| \delta V\|$ with respect to $\| V \|$
requires additional care, see \ref{thm:one_dim}.  \eqref{item:intro_dense} is
a corollary to \hyperlink{Thm_A}{Theorem~A}.  Finally,
\eqref{item:intro_agree} relies on an estimate of $\Lpnorm{\| \delta V \|} qf$
for generalised weakly differentiable functions $f$, see
\ref{thm:global_special_sob}.

In view of \hyperlink{Thm_E}{Theorem~E}, depending on the intended usage, both
\begin{equation*}
	\SWSob q (V,Y) \quad \text{and} \quad \SWloc q (V,Y) \cap \{ f \with
	\sigma (f) < \infty \}
\end{equation*}
could act as substitute for the Euclidean Sobolev space.

\subsubsection*{Comparison to other Sobolev spaces, Section 8} Finally, the
presently introduced notion of Sobolev space shall be compared to notions of
Sobolev space for finite Radon measures $\mu$ over $\rel^\adim$ defined in
Bouchitt{\'e}, Buttazzo and Fragal{\`a}, see \cite{MR1857850}; see also
Bouchitt{\'e}, Buttazzo and Seppecher in \cite{MR1424348}. To describe this
approach suppose $1 \leq q \leq \infty$, $1 \leq r \leq \infty$, and
$1/q+1/r=1$. Firstly, one defines vectorspaces $T_\mu^q (x)$ for $x \in
\rel^\adim$, acting as tangent space, by means of $r$-th power $\mu$ summable
vectorfields such that a suitable distributional divergence involving $\mu$ is
also $r$-th power summable, see \ref{miniremark:setup_bss}.%
\begin{footnote}
	{In \cite[p.\,403]{MR1857850} Bouchitt{\'e}, Buttazzo and Fragal{\`a}
	define $T_\mu^q$ to be an equivalence class of functions agreeing
	$\mu$ almost everywhere; see \ref{miniremark:setup_bss} for a
	canonical representative.}
\end{footnote}
Accordingly, the gradient $\nabla_\mu^q f$ of $f \in \mathscr{D} ( \rel^\adim,
\rel )$ is given by%
\begin{footnote}
	{In \cite[p.\,403]{MR1857850} Bouchitt{\'e}, Buttazzo and Fragal{\`a}
	suppress the dependency on $q$, and $\nabla_\mu$ is considered as a
	linear map of a subset of $L_q ( \mu, \rel )$ into $( L_q ( \mu,
	\rel))^\adim$, see \ref{example:quotient-lp-spaces}.}
\end{footnote}
\begin{equation*}
	\nabla_\mu^q f(x) = \project{T_\mu^q (x)} ( \grad f(x) ) \quad
	\text{for $\mu$ almost all $x$}.
\end{equation*}
Then one obtains both the strong Sobolev space $H_\mu^{1,q} ( \rel^\adim )$
and the associated weak derivative by taking a suitable closure of the
afore-mentioned gradient operator $\nabla_\mu^q$.

Examples due to Di~Marino and Speight \cite[Theorem~1]{MR3411142} imply that
the vectorspace $T_\mu^q(x)$ and hence also the weak derivative depend on $q$,
see \ref{miniremark:tangent_planes} and \ref{remark:di_marino_speight}.  This
is in some sense analogous to the dependency of the $(\mu,m)$ approximate
derivative on the dimension $m$.  It appears to be largely unknown what is the
weakest condition on the first variation, for instance amongst those
considered in \ref{miniremark:situation}, to ensure that the two concepts
agree for an $\vdim$ dimensional rectifiable varifold. By results of
Fragal{\`a} and Mantegazza in \cite{MR1686704}, one is at least assured that
they agree provided $\vdim$, $\adim$, $U$, and $V$ are as in the mean
curvature hypothesis, $\| V \| ( \rel^\adim ) < \infty$, $\| \delta V \|$ is
absolutely continuous with respect to $\| V \|$, and $\mathbf{h} (V,\cdot) \in
\Lp \infty ( \| V \|, \rel^\adim )$, see \ref{miniremark:setup_bss_continued}.
In case the tangent spaces agree, taking $\sigma$ as in the preceding theorem,
one may isometrically identify the strong Sobolev space $H_{\| V \|}^{1,q} (
\rel^\adim)$ with the quotient space
\begin{equation*}
	\big ( \SWloc q ( V, Y ) \cap \{ f \with \sigma (f) < \infty \} \big )
	\Big / \big ( \SWloc q (V,Y) \cap \{ f \with \sigma (f) = 0 \} \big )
\end{equation*}
provided $\vdim$, $\adim$, $U$, and $V$ are as in the density hypothesis and
the mean curvature hypothesis, $U = \rel^\adim$, $1 \leq q < \infty$, and $\|
V \| ( \rel^\adim ) < \infty$, see \ref{miniremark:prelim-comparison}.  Under
the same hypotheses, one may similarly identify the weak Sobolev space $W_{\|
V \|}^{1,q} ( \rel^\adim )$ introduced by Bouchitt{\'e}, Buttazzo, and
Fragal{\`a} in \cite[p.\,403]{MR1857850} with a quotient space based on
$\mathbf{W} (V,Y)$, see \ref{miniremark:comparison-weak-sobolev-spaces}.

Summarising, the approach initiated by Bouchitt{\'e}, Buttazzo and Seppecher
in \cite{MR1424348} allows to treat geometric objects consisting of pieces of
different dimensions.  However, it seems not to be tailored for the study of
varifolds as is exemplified by the behaviour of the different notions of
tangent planes.  Moreover, apart from a general coarea formula, see
Bellettini, Bouchitt{\'e} and Fragal{\`a} \cite[\S\,4]{MR1736243}, very few
structural results and no embedding estimates appear to have been known for
those spaces even if the Radon measure is the weight of a suitable varifold.
Now, in the cases where the above-mentioned isometries to the spaces developed
here are valid, much of the theory of the present paper and its predecessor,
\cite{snulmenn:tv.v2}, applies to Sobolev spaces in the sense of
Bouchitt{\'e}, Buttazzo and Seppecher as well.

\subsection*{Possible lines of further study} \subsubsection*{Second order
elliptic partial differential equations in divergence form} The primary
motivation for this paper was to provide a natural framework for the study of
divergence form, second order elliptic partial differential equations.  More
concretely, the author's motivation stems from two results announced already
in \cite{snulmenn.mfo1230}.  Firstly, a local maximum estimate for
subsolutions generalising those of Allard \cite[7.5\,(6)]{MR0307015} or
Michael and Simon \cite[3.4]{MR0344978} could be used to deduce a strong
second order differentiability of the support of \emph{integral}%
\begin{footnote}
	{An $\vdim$ dimensional rectifiable varifold $V$ in an open subset of
	$\rel^\adim$ is integral if and only if $\density^\vdim ( \| V \|, x
	)$ is an integer for $\| V \|$ almost all $x$, see Allard
	\cite[3.5\,(1c)]{MR0307015}.}
\end{footnote}
varifolds satisfying the mean curvature hypothesis from the author's second
order rectifiability result \cite[4.8]{snulmenn.c2}, see
\cite[Corollary~2\,(1)]{snulmenn.mfo1230}.  Secondly, a suitable version of a
weak Harnack estimate could be employed to prove an area formula for a
suitably defined substitute of the Gauss map of at least two dimensional such
varifolds in codimension one, see \cite[Theorem~3]{snulmenn.mfo1230}.  For
these two specific applications, it would suffice to formulate the estimates
for Lipschitzian subsolutions respectively Lipschitzian solutions.  However,
as these estimates are of independent significance they shall be formulated in
their -- yet to be determined -- natural generality using the presently
introduced Sobolev spaces.  Finally, the author hopes that dispensing with ad
hoc formulations in favour of using Sobolev spaces will also facilitate the
exchange of ideas between varifold theory and other areas of geometric
analysis.

\subsubsection*{Minimisation of integral functionals}

For integral functionals based on the presently introduced Sobolev spaces
certain minimisation problems possess a solution. A simple example is given by
the Rayleigh quotient where one may check the conditions of the abstract
framework of Arnlind, Bj{\"o}rn and Bj{\"o}rn, see \cite[5.3]{MR3462618}, in
the situation of \ref{thm:rellich_embedding} if $1 < q < \infty$ using
\ref{remark:quotient-zero-sobolev-space} and \ref{remark:rellich_embedding}. A
more comprehensive study of the problem would include investigation of lower
semicontinuity of integral functionals defined on the presently introduced
Sobolev spaces in the spirit of quasiconvexity. In the context of Sobolev
spaces over compactly supported Radon measures in $\rel^\adim$ this topic was
considered by Fragal{\`a}, see \cite{MR1957092}.

In view of the closedness results \ref{corollary:closedness-tv} and
\ref{corollary:closedness_tg}, similar questions might be considered in the
framework of $\trunc (V,Y)$ and $\trunc_G (V)$ spaces.

\subsection*{Logical prerequisites} The present paper is a continuation of the
author's paper \cite{snulmenn:tv.v2}. Concerning newer results on varifolds,
additionally only some auxiliary results from
\cite[\S\,1]{snulmenn.isoperimetric} and Kolasi{\'n}ski and the author
\cite[\S\,3]{kol-men:decay.v2} are employed.  Concerning Sobolev spaces over
finite Radon measures, certain items from Fragal{\`a} and Mantegazza
\cite{MR1686704} and Bouchitt{\'e}, Buttazzo, and Fragal{\`a} \cite{MR1857850}
are used.

Additionally, a number of classical results are employed. For those items, as
a service to the reader, detailed references to Whitney \cite{MR0087148},
Dunford and Schwartz \cite{MR0117523}, Federer \cite{MR41:1976}, Allard
\cite{MR0307015}, Kelley \cite{MR0370454}, Castaing and Valadier
\cite{MR0467310}, Bourbaki \cite{MR910295,MR979294,MR979295}, do~Carmo
\cite{MR1138207}, and Adams and Fournier \cite{MR2424078} are given.
\begin{comment}
	Throughout the paper comments in small font are included.  These are
	not part of the logical line of arguments but are rather offered to
	the reader as a guide through the formal presentation of the material.
\end{comment}

\subsection*{Acknowledgements} The author would like to thank
Prof.~Dr.~Richard Schoen for suggesting to investigate the validity of
Rellich's theorem in the present context, Dr.~Simone Di~Marino for
conversations concerning connections of the present theory to the one of
metric measure spaces, and Dr.~S{\l}awomir Kolasi{\'n}ski for his advice on
exposition.  Moreover, the author would like to thank the University of
Warwick where a part of this paper was written for its hospitality.  Finally,
the author would like to thank the referee for his or her very careful reading
of the manuscript.

\section{Notation} \label{sec:notation} The notation of
\cite[\S\,1]{snulmenn:tv.v2} will be employed which follows with some additions
and modifications Federer \cite{MR41:1976} and Allard \cite{MR0307015}.

\paragraph{Modifications} If $X$ is a metric space metrised by $\varrho$, $A
\subset X$, and $x \in X$, then the \emph{distance of $x$ to $A$} is defined
by $\dist (x,A) = \inf \{ \varrho (x,a) \with a \in A \}$.%
\begin{footnote}
	{Notice that $\inf \varnothing = \infty$ and $\sup \varnothing = -
	\infty$, see \cite[2.1.1]{MR41:1976}.}
\end{footnote}
If $X$ is a metric space and $M$ is the class of Borel regular measures $\psi$
over $X$ such that $\psi ( U_i ) < \infty$ for $i \in \nat$ for some sequence
of open sets $U_1, U_2, U_3, \ldots$ covering $X$, measures $\psi_\phi \in M$
will be defined by
\begin{equation*} \label{page:psi_phi}
	\psi_\phi (A) = \inf \{ \psi (B) \with \text{$B$ is a Borel set and
	$\phi ( A \without B ) = 0$} \} \quad \text{for $A \subset X$}
\end{equation*}
whenever $\phi, \psi \in M$. This extends \cite[2.9.1, 2.9.2,
2.9.7]{MR41:1976} to certain measures failing to be finite on bounded sets.

\paragraph{Definitions in the text} The notion of \emph{pseudometric} is
introduced in \ref{def:pseudometric}. The local Lebesgue space $\Lploc p (
\mu, Y )$ and the space of continuous functions $\mathscr{C} (X,Y)$ are
defined in \ref{def:local-lebesgue-space} and
\ref{def:space-of-continuous-functions} respectively.
The local Sobolev space
$\SWloc q ( V, Y)$ and its topology is defined in
\ref{def:strict_local_sobolev_space} and \ref{def:topology-SWSob}. The
quantity $\SWnorm q V f$ for certain functions $f$ is defined in
\ref{def:sobolev-seminorm}. Finally, the Sobolev space $\SWSob q (V,Y)$ and
its subspace $\SWzero q (V,Y)$ are defined in \ref{def:strict_sobolev_space}
and \ref{def:strict-sobolev-space-zero-boundary-values}.

\section{Locally convex spaces}
The purpose of this section is to summarise properties of locally convex
spaces, including definitions of some particular spaces, for convenient
reference.
\begin{comment}
	Firstly, some properties of Lebesgue spaces are given.
\end{comment}
\begin{miniremark} \label{miniremark:prop_lp_spaces}
	Suppose $1 \leq p < \infty$, $\mu$ is a Radon measure over an open
	subset of $U$ of $\rel^\adim$, and $Y$ is a separable Banach space.
	Then $\mathscr{D} (U,Y)$ is $\mu_{(p)}$ dense in $\Lp p ( \mu, Y )$
	and $\Lp p ( \mu, Y )$ is $\mu_{(p)}$ separable; in fact, whenever $G$
	is an open subset of $U$ and $K$ is a compact subset of $G$, there
	exists $\zeta \in \mathscr{D} (U, \rel)$ such that $0 \leq \zeta (x)
	\leq 1$ for $x \in U$, $\zeta (x) = 1$ for $x \in K$, and $\spt \zeta
	\subset G$, hence \cite[2.2.5, 2.4.12]{MR41:1976} implies the
	denseness and the separability follows from \cite[2.2, 2.15,
	2.24]{snulmenn:tv.v2}.
\end{miniremark}
\begin{comment}
	Next, the meaning of the term pseudometric is specified.
\end{comment}
\begin{definition} [see \protect{\cite[\printRoman 2, \S\,1.2,
	def.\,3]{MR979294}, \cite[\printRoman{9}, \S\,1.1, def.\,1; \printRoman
	9, \S\,1.2, def.\,2]{MR979295}}] \label{def:pseudometric}

	Suppose $X$ is a set.

	Then $\varrho : X \times X \to \{ t \with 0 \leq t \leq \infty \}$
	will be called a \emph{pseudometric on $X$} if and only if the
	following three conditions are satisfied.
	\begin{enumerate}
		\item If $x \in X$, then $\varrho (x,x) = 0$.
		\item If $a,x \in X$, then $\varrho (a,x)=\varrho(x,a)$.
		\item If $a,x,\chi \in X$, then $\varrho (a,\chi) \leq \varrho
		(a,x) + \varrho (x,\chi)$.
	\end{enumerate}
	The sets $X \cap \{ x \with \varrho (a,x) < r \}$ corresponding to $a
	\in X$ and $0 < r < \infty$ form a base of a topology on $X$, called
	the \emph{topology induced by $\varrho$}.
\end{definition}
\begin{remark}
	Notice that $\infty$ may occur amongst the values of
	$\varrho$.\footnote{This is in contrast with the definition of
	``pseudo-metric'' in \cite[p.\,119]{MR0370454}.}
\end{remark}
\begin{comment}
	The next two items provide basic properties of locally convex spaces
	whose topology is defined by a set of real valued seminorms. All of
	these properties are readily verified and most of them can also be
	found in Bourbaki \cite{MR910295}.
\end{comment}
\begin{miniremark} \label{miniremark:lcs-induced-by-seminorms}
	Suppose $E$ is a vectorspace endowed with the topology induced by a
	nonempty family $\Sigma$ of real valued seminorms on $E$, see
	\cite[\printRoman 2, p.\,3]{MR910295}.\footnote{A real valued seminorm
	is precisely a ``semi-norm'' in the sense of \cite[\printRoman 2,
	p.\,1, def.\,1]{MR910295}.} Then $E$ is a locally convex space and the
	family of sets
	\begin{equation*}
		E \cap \{ x \with \sigma (x-a) < r \}
	\end{equation*}
	corresponding to $a \in E$, $0 < r < \infty$, and $\sigma \in \Sigma$
	form a subbase of the topology of $E$, see \cite[\printRoman 2, p.\,24,
	cor.]{MR910295}. A subset $B$ of $E$ is bounded if and only if
	\begin{equation*}
		\sup \sigma \lIm B \rIm < \infty \quad \text{whenever $\sigma
		\in \Sigma$},
	\end{equation*}
	see \cite[\printRoman 3, p.\,2]{MR910295}. If for $\sigma_1, \sigma_2
	\in \Sigma$ there exists $\sigma_3 \in \Sigma$ with $\sup \{ \sigma_1,
	\sigma_2 \} \leq \sigma_3$, then the above-mentioned family forms a
	base of the topology of $E$. If $\sigma_1, \sigma_2, \sigma_3, \ldots$
	form an enumeration of $\Sigma$, then the topology of $E$ is induced
	by the translation invariant real valued pseudometric with value
	\begin{equation*}
		\sum_{i=1}^\infty 2^{-i} \inf \{ 1, \sigma_i (x-a) \} \quad
		\text{at $(a,x) \in E \times E$}
	\end{equation*}
	which is a metric if and only if $E$ is Hausdorff.
\end{miniremark}
\begin{miniremark} \label{miniremark:lcs-quotient}
	Suppose $E$ and $\Sigma$ are as in
	\ref{miniremark:lcs-induced-by-seminorms} and $V$ is the closure of
	$\{ 0 \}$ in $E$. Then $V = \bigcap \{ \sigma^{-1} \lIm \{ 0 \} \rIm
	\with \sigma \in \Sigma \}$ is a vector subspace and $E$ is Hausdorff
	if and only if $V = \{ 0 \}$. Moreover, denoting the canonical
	projection of $E$ onto $E/V$ by $\pi$ and endowing $E/V$ with the
	topology induced by the family $\{ \sigma \circ \pi^{-1} \with \sigma
	\in \Sigma \}$ of real valued seminorms on $E/V$, one obtains the
	quotient locally convex space, see \cite[\printRoman 2, p.\,5;
	\printRoman 2, p.~29, Example~\printRoman{1}]{MR910295}.%
	\begin{footnote}
		{Whenever $f$ and $g$ are relations the inverse and
		composition satisfy
		\begin{equation*}
			f^{-1} = \{ (y,x) \with (x,y) \in f \}, \quad g \circ
			f = \{ (x,z) \with \text{$(x,y) \in f$ and $(y,z) \in
			g$ for some $y$} \}.
		\end{equation*}}
	\end{footnote}
	Clearly, $E/V$ is Hausdorff and, if $E$ is complete, so is $E/V$.
\end{miniremark}
\begin{comment}
	Often, it is more convenient to consider functions -- as contained
	in $\Lp p ( \mu, Y )$ -- instead of equivalence classes of functions.
	However, to access certain functional analytic results, the quotient
	space $L_p ( \mu, Y )$ will be introduced as well.
\end{comment}
\begin{example} \label{example:quotient-lp-spaces}
	Occasionally, the Banach spaces
	\begin{equation*}
		L_q ( \mu, Y ) = \Lp q ( \mu, Y ) \Big / \big ( \Lp q ( \mu, Y
		) \cap \{ f \with \mu_{(q)} (f) = 0 \} \big )
	\end{equation*}
	corresponding to $1 \leq q \leq \infty$, measures $\mu$, and Banach
	spaces $Y$ will be employed. If $1 < q < \infty$ and $\dim Y <
	\infty$, then $L_q ( \mu, Y )$ is reflexive, see
	\cite[2.5.7\,(i)]{MR41:1976}; in fact, a basis of $Y$ induces an
	isomorphism $L_q ( \mu, Y ) \simeq L_q ( \mu, \rel )^{\dim Y}$.
\end{example}
\begin{comment}
	Next, the local Lebesgue spaces for Radon measures over locally
	compact Hausdorff spaces $X$ are introduced; in fact, $X$ will always
	be an open subset of some Euclidean space in the later sections.
\end{comment}
\begin{definition} \label{def:local-lebesgue-space}
	Suppose $1 \leq p \leq \infty$, $\mu$ is a Radon measure over a
	locally compact Hausdorff space $X$, and $Y$ is a Banach space.

	Then $\Lploc p ( \mu, Y )$ is defined to be the vectorspace consisting
	of all functions~$f$ mapping a subset of $X$ into $Y$ such that $f \in
	\Lp p ( \mu \restrict K, Y )$ whenever $K$ is a compact subset of $X$.
	Moreover, $\Lploc p ( \mu, Y )$ is endowed with the topology induced
	by the family of seminorms mapping $f \in \Lploc p ( \mu, Y )$ onto
	$\eqLpnorm{\mu \restrict K} p f$ corresponding to all compact subsets
	$K$ of $X$, see \ref{miniremark:lcs-induced-by-seminorms}. Let $\Lploc
	p ( \mu ) =  \Lploc p ( \mu, \rel )$.
\end{definition}
\begin{remark}
	This definition is in accordance with \cite[p.\,16]{snulmenn:tv.v2}.
\end{remark}
\begin{remark} \label{remark:lploc-countable}
	If $K(i)$ is a sequence of compact subsets of $X$ with $K(i) \subset
	\Int K(i+1)$ for $i \in \nat$ and $X = \bigcup_{i=1}^\infty K(i)$,
	then the topology on $\Lploc p (\mu,Y)$ is induced by the seminorms
	$(\mu \restrict K(i))_{(p)}$ corresponding to $i \in \nat$ by
	\ref{miniremark:lcs-induced-by-seminorms}.
\end{remark}
\begin{remark} \label{remark:lploc-complete}
	The topological vector space $\Lploc p ( \mu, Y )$ is a complete
	locally convex space and
	\begin{enumerate}
		\item $\mu ( T \without \dmn f ) = 0$ whenever $\mu (T) <
		\infty$,
		\item $f^{-1} \lIm B \rIm$ is $\mu$ measurable whenever $B$ is
		a Borel subset of $Y$,
	\end{enumerate}
	whenever $f \in \Lploc p ( \mu, Y )$; this is evident if $X$ is
	countably $\mu$ measurable%
\begin{footnote}
	{A set is called countably $\mu$ measurable if and only if it equals
	the union of a countable family of $\mu$ measurable sets with finite
	$\mu$ measure, see \cite[2.3.4]{MR41:1976}.}
\end{footnote}
	and may be verified using the family $G$
	constructed in \cite[2.5.10]{MR41:1976} in the general case.
\end{remark}
\begin{remark} \label{remark:lploc-dense-sep}
	If $p < \infty$, $X$ is an open subset of $\rel^\adim$, and $Y$ is
	separable, then $\mathscr{D} (U,Y)$ is dense in $\Lploc p ( \mu, Y )$
	and $\Lploc p ( \mu, Y )$ is separable; in fact, the inclusion map of
	$\Lp p ( \mu, Y )$ topologised by $\mu_{(p)}$ into $\Lploc p ( \mu, Y
	)$ is continuous with dense image by
	\ref{miniremark:lcs-induced-by-seminorms}, hence
	\ref{miniremark:prop_lp_spaces} implies the conclusion.
\end{remark}
\begin{remark} \label{remark:lploc-quotient}
	The quotient locally convex space $Q = \Lploc q (\mu,Y) \big / V$,
	where
	\begin{equation*}
		V = \Lploc q(\mu,Y) \cap \big \{ f
		\with \text{$\eqLpnorm {\mu \restrict K}qf = 0$ whenever $K$
		is a compact subset of $U$} \big \},
	\end{equation*}
	see \ref{miniremark:lcs-quotient}, is Hausdorff and complete by
	\ref{remark:lploc-complete}. Under the conditions of
	\ref{remark:lploc-countable}, the topology of $Q$ is induced by a
	translation invariant metric by
	\ref{miniremark:lcs-induced-by-seminorms} and $Q$ is an ``$F$-space''
	in the terminology of \cite[\printRoman{2}.1.10]{MR0117523}.
\end{remark}
\begin{comment}
	Finally, the locally convex space of continuous functions defined on
	some locally compact Hausdorff space with values in some Banach space
	is introduced.
\end{comment}
\begin{definition} \label{def:space-of-continuous-functions}
	Suppose $X$ is a locally compact Hausdorff space and $Y$ is a Banach
	space.

	Then $\mathscr{C}(X,Y)$ denotes the vectorspace of all continuous
	functions mapping $X$ into $Y$. Its topology is induced by the
	seminorms $\nu_K$ defined by
	\begin{equation*}
		\nu_K(f) = \sup ( \{ 0 \} \cup \{ |f(x)| \with x \in K \} )
		\quad \text{for $f \in \mathscr{C}(X,Y)$}
	\end{equation*}
	corresponding to all compact subsets $K$ of $X$, see
	\ref{miniremark:lcs-induced-by-seminorms}.%
	\begin{footnote}
		{The topological space $\mathscr{C}(X,Y)$ is denoted
		$\mathscr{C}_c (X;Y)$ in \cite[\printRoman{10}, \S\,1.6,
		p.\,280]{MR979295}, where the letter ``$c$'' indicates the
		topology of compact convergence.}
	\end{footnote}
	Let $\mathscr{C} (X) = \mathscr{C} (X,\rel)$.
\end{definition}
\begin{remark} \label{remark:continuous-functions-complete-lcs}
	The topological vector space $\mathscr{C}(X,Y)$ is a Hausdorff
	complete locally convex space. If $K(i)$ is a sequence of compact
	subsets of $X$ with $K(i) \subset \Int K(i+1)$ for $i \in \nat$ and $X
	= \bigcup_{i=1}^\infty K(i)$, then the topology on $\mathscr{C} (X,Y)$
	is induced by the seminorms $\nu_{K(i)}$ corresponding to $i \in \nat$
	by \ref{miniremark:lcs-induced-by-seminorms}.
\end{remark}
\begin{remark} \label{remark:kx_dense_in_cx}
	The inclusion map of $\mathscr{K} (X)$ into $\mathscr{C} (X)$
	is continuous. Moreover, $\mathscr{K} (X)$ is dense in
	$\mathscr{C}(X)$ since for every compact subset $K$ of $X$ there
	exists $\zeta \in \mathscr{K} (X)$ with $\zeta(x)=1$ for $x \in K$ by
	\cite[5.17, 5.18]{MR0370454}.
\end{remark}

\section{Locally Lipschitzian functions} In this section an approximation
result for Lipschitzian functions by functions of class $1$ over rectifiable
varifolds is proven, see \ref{thm:approximation_lip}. Its main additional
feature is that agreement outside a set of small weight measure may be
achieved while essentially maintaining the Lipschitz constant of the original
function. This rests on the observation that every submanifold of class~$1$ of
$\rel^\adim$ may be expressed as the image of a retraction of class~$1$ whose
differential at each point of the submanifold equals the orthogonal projection
of $\rel^\adim$ onto the tangent space at that point, see
\ref{corollary:special-retraction}.

The approximation result of this section will be used to prove various density
results of function of class $\infty$ in Sobolev spaces with exponent $q <
\infty$, see
\ref{remark:basic-top-swloc}\,\eqref{item:basic-top-swloc:q<oo_dense},
\ref{remark:density_sobolev_space}, and \ref{remark:density-in-SWzero}.
\begin{comment}
	The following theorem is a direct consequence of the definition of
	submanifolds and Whitney's extension theorem, see
	\cite[3.1.14]{MR41:1976}.
\end{comment}
% NO PAGE BREAK if possible
\begin{theorem} \label{thm:submanifold_extension}
	Suppose $\vdim, \adim \in \nat$, $\vdim \leq \adim$, $M$ is an $\vdim$
	dimensional submanifold of class $1$ of $\rel^\adim$, $Y$ is a normed
	vectorspace, and $f : M \to Y$ is of class $1$ relative to $M$.%
	\begin{footnote}
		{If $A \subset \rel^\adim$ and $g : A \to Y$, then $g$ is of
		class~$1$ relative to $A$ if and only if there exist an open
		subset $U$ of $\rel^\adim$ and $h : U \to Y$ of class~$1$ such
		that $A \subset U$ and $h|A=g$, see \cite[3.1.22]{MR41:1976}.}
	\end{footnote}

	Then the following two statements hold:
	\begin{enumerate}
		\item \label{item:submanifold_extension:uni_diff} If $\varrho
		(C,\delta)$ denotes the supremum of the set consisting of $0$
		and all numbers
		\begin{equation*}
			| f(x)-f(a)-\langle \project{\Tan(M,a)}(x-a), \Der f
			(a) \rangle |/|x-a|
		\end{equation*}
		corresponding to $\{ x,a \} \subset C$ with $0 < |x-a| \leq
		\delta$ whenever $C \subset M$ and $\delta > 0$, then $\varrho
		( C, \delta ) \to 0$ as $\delta \to 0+$ whenever $C$ is a
		compact subset of $M$.
		\item \label{item:submanifold_extension:extension} There exist
		an open subset $U$ of $\rel^\adim$ with $M \subset U$ and a
		function $g : U \to Y$ of class $1$ with $g|M = f$ and
		\begin{equation*}
			\Der g(a) = \Der f(a) \circ \project{\Tan(M,a)} \quad
			\text{for $a \in M$}.
		\end{equation*}
	\end{enumerate}
\end{theorem}
\begin{proof}
    \eqref{item:submanifold_extension:uni_diff} is readily verified by use of
    \cite[3.1.19\,(1), 3.1.11]{MR41:1976}.

    Define $P_a : \rel^\adim \to Y$ by $P_a(x) = f (a) + \langle
    \project{\Tan(M,a)} (x-a), \Der  f(a) \rangle$ for $a \in M$ and $x \in
    \rel^\adim$. Noting \eqref{item:submanifold_extension:uni_diff} and
    \begin{equation*}
	\Der P_a(x)-\Der P_x(x) = \Der  f(a) \circ \project{\Tan(M,a)} - \Der
	f(x) \circ \project{\Tan(M,x)}
    \end{equation*}
    for $a,x \in M$, one applies \cite[3.1.14]{MR41:1976} to construct for
    each closed subset $A$ of $\rel^\adim$ with $A \subset M$, a function $g_A
    : \rel^\adim \to Y$ of class $1$ with $g_A|A = f |A$ and
    \begin{equation*}
	\Der g_A(a) = \Der f(a) \circ \project{\Tan(M,a)} \quad \text{for $a
	\in A$}.
    \end{equation*}
    Therefore $g$ is constructable by use of a partition of unity.
\end{proof}
\begin{comment}
	The existence of a retraction with the desired additional property now
	follows from the known existence result of retractions, see
	\cite[p.\,121]{MR0087148}.
	%\cite[3.1.20]{MR41:1976}.
\end{comment}
\begin{corollary} \label{corollary:special-retraction}
    Suppose $\vdim, \adim \in \nat$, $\vdim \leq \adim$, and $M$ is an $\vdim$
    dimensional submanifold of class $1$ of $\rel^\adim$.

    Then there exists a function $r$ of class $1$ retracting some open subset
    of $\rel^\adim$ onto $M$ and satisfying
    \begin{equation*}
	\Der r(a) = \project{\Tan(M,a)} \quad \text{whenever $a \in M$}.
    \end{equation*}
\end{corollary}
\begin{proof}
    Obtaining from \cite[p.\,121]{MR0087148} a map $h$ of class $1$ retracting
    some open subset of $\rel^\adim$ onto $M$ and from
    \ref{thm:submanifold_extension}\,\eqref{item:submanifold_extension:extension}
    with $Y = \rel^\adim$ and $f = \id{M}$ a function $g$, one may take $r = h
    \circ g$.
\end{proof}
\begin{comment}
	Three more preparatory lemmata for the approximation result are
	needed.  The first one is fairly elementary.
\end{comment}
\begin{lemma} \label{lemma:app-nonnegative}
    Suppose $U$ is an open subset of $\rel^\adim$, $\mu$ is a Radon measure
    over~$U$, $h : U \to \rel$ is of class $1$, $A = \{ x \with h(x) \geq 0
    \}$, and $\varepsilon > 0$.

    Then there exists a nonnegative function $g : U \to \rel$ of class $1$
    such that
    \begin{equation*}
	\mu ( A \without \{ x \with h(x) = g (x) \} ) \leq \varepsilon.
    \end{equation*}
\end{lemma}
\begin{proof}
    Applying \cite[3.1.13]{MR41:1976} with $\Phi = \{ U \}$, one obtains a
    sequence $\zeta_i \in \mathscr{D} (U,\rel)$ forming a partition of unity
    on $U$ associated to $\{ U \}$. Abbreviating $K_i = \spt \zeta_i$, choose
    $\delta_i > 0$ and nonnegative functions $f_i : \rel \to \rel$ of class
    $1$ with
    \begin{gather*}
	\mu ( K_i \cap \{x \with 0 < h (x) < \delta_i \} ) \leq 2^{-i}
	\varepsilon, \\
	f_i (t) = \sup \{ t, 0 \} \quad \text{if either $t \leq 0$ or $t \geq
	\delta_i$}
    \end{gather*}
    whenever $i \in\nat$. Since
    \begin{equation*}
	A \without \left \{ x \with h(x) = \sum_{i=1}^\infty \zeta_i (x) (f_i
	\circ h ) (x) \right \} \subset \bigcup_{i=1}^\infty K_i \cap \{ x
	\with 0 < h(x) < \delta_i \},
    \end{equation*}
    one may take $g = \sum_{i=1}^\infty \zeta_i (f_i \circ h)$.
\end{proof}
\begin{comment}
	The second one is slightly more elaborate and relies, among other
	things, on Kirsz\-braun's extension theorem, see
	\cite[2.10.43]{MR41:1976}, and a partition of unity, see
	\cite[3.1.13]{MR41:1976}.
\end{comment}
\begin{lemma} \label{lemma:c1_extension}
    Suppose $l, \adim \in \nat$, $U$ is an open subset of $\rel^\adim$, $A
    \subset U$, $f : U \to \rel^l$ is of class~$1$, and $\varepsilon > 0$.

    Then there exist an open subset $X$ of $U$ and a function $g : \rel^\adim
    \to \rel^l$ of class~$1$ such that $A \subset X$, $f|X = g|X$, and
    \begin{equation*}
	\Lip g \leq \varepsilon + \sup \{ \Lip ( f|A ), \sup \| \Der f \| \lIm A
	\rIm \}.
    \end{equation*}
    Moreover, if $l = 1$ and $f \geq 0$ then one may require $g \geq 0$.
\end{lemma}
\begin{proof}
    Assume $\kappa = \sup \{ \Lip (f|A), \sup \| \Der  f \| \lIm A \rIm \} <
    \infty$ and that $A$ is relatively closed in $U$. Firstly, it will be
    shown that there exists an open subset $G$ of $U$ with
    \begin{equation*}
	A \subset G, \quad \Lip (f | G ) \leq \varepsilon/2 + \kappa.
    \end{equation*}
    
    Define $\eta = 2^{-4} \varepsilon (\varepsilon+\kappa )^{-1}$,
    note $\eta \leq 1/2$, choose $\delta : A \to \{ r \with r > 0 \}$ such
    that
    \begin{equation*}
	\oball{a}{\delta(a)} \subset U \quad \text{and} \quad \Lip ( f |
	\oball {a}{\delta(a)} ) \leq \varepsilon/2 + \kappa \quad
	\text{whenever $a \in A$},
    \end{equation*}
    and let $G = \bigcup \{ \oball {a}{\eta \delta (a)} \with a \in A \}$.
    Suppose $a,x \in A$, $\alpha, \chi \in \rel^\adim$, $| a-\alpha | < \eta
    \delta (a)$ and $|x-\chi| < \eta \delta (x)$. In case $\delta ( a) +
    \delta (x) \leq 4 |\chi-\alpha|$, one estimates
    \begin{gather*}
	|a-\alpha|+|x-\chi| < 4 \eta |\chi-\alpha|, \quad |x-a| \leq (1+4\eta)
	|\chi-\alpha|, \\
	\begin{aligned}
	    |f(\chi)-f(\alpha)| & \leq ( \varepsilon/2 + \kappa ) (
	    |a-\alpha|+|x-\chi| ) + \kappa |x-a| \\
	    & \leq ( 8 \eta ( \varepsilon + \kappa ) + \kappa ) |\chi-\alpha|
	    = ( \varepsilon/2 + \kappa) |\chi-\alpha|
	\end{aligned}
    \end{gather*}
    and, in case $\delta (a) + \delta (x) > 4 |\chi-\alpha|$ and $\delta (a)
    \geq \delta (x)$ one estimates
    \begin{equation*}
	|\chi-a| \leq | \chi- \alpha| + | \alpha-a | < ( 1/2+\eta) \delta (a)
	\leq \delta (a), \quad | \alpha-a | < \delta (a),
    \end{equation*}
    hence always $| f ( \chi ) - f(\alpha) | \leq ( \varepsilon/2 + \kappa ) |
    \chi-\alpha |$.

    Next, choose $g_0 : \rel^\adim \to \rel^l$ with $g_0 | G = f | G$ and
    $\Lip g_0 = \Lip ( f | G)$ such that $g_0 \geq 0$ if $l=1$ and $f \geq 0$,
    see \cite[2.10.43, 4.1.16]{MR41:1976}. Using \cite[3.1.13]{MR41:1976} with
    $\Phi$ replaced by $\{ G, U \without A \}$, one constructs nonnegative
    functions $\phi_0 \in \mathscr{E} ( U , \rel)$ and $\phi_i \in \mathscr{D}
    (U,\rel)$ for $i \in \nat$ such that
    \begin{gather*}
	\card ( \nat \cap \{ i \with K \cap \spt \phi_i \neq \varnothing \} )
	< \infty \quad \text{whenever $K$ is compact subset of $U$}, \\
        A \subset \Int \{ x \with \phi_0 (x) = 1 \},
        \qquad \spt \phi_0 \subset G, \qquad \spt \phi_i \subset
        U \without A \quad \text{for $i \in \nat$}, \\
        \tsum{j=0}{\infty} \phi_j (x) = 1 \quad \text{for $x \in U$}.
    \end{gather*}
    Employing convolution, one obtains functions
    $g_i : \rel^\adim \to \rel^l$ of class $1$ satisfying
    \begin{gather*}
	\Lip g_i \leq \Lip g_0, \quad ( \Lip \phi_i ) \sup \im | g_i-g_0 |
	\leq 2^{-i-1} \varepsilon, \\
        \text{if $l=1$ and $f \geq 0$ then $g_i \geq 0$}
    \end{gather*}
    for $i \in \nat$. Let $g = \sum_{j=0}^\infty \phi_j g_j$ and observe
    that $g$ is of class $1$. Also for $x,\chi \in U$
    \begin{gather*}
        g(x)-g(\chi) = \tsum{j=0}{\infty} \big ( \phi_j (x) (
        g_j(x)-g_j(\chi)) + ( \phi_j (x) - \phi_j (\chi) ) ( g_j
        (\chi) - g_0(\chi)) \big ), \\
        \Lip g \leq \varepsilon/2 + \Lip g_0 = \varepsilon/2 + \Lip
        (f|G) \leq \varepsilon + \kappa.
    \end{gather*}
    Therefore one may take $X = \Int \{ x \with \phi_0 (x) = 1 \}$.
\end{proof}
\begin{comment}
	The third one concerns basic properties of rectifiable varifolds.
\end{comment}
\begin{lemma} \label{lemma:basic_rect_var}
	Suppose $\vdim, \adim \in \nat$, $\vdim \leq \adim$, $U$ is an open
	subset of $\rel^\adim$, $V \in \RVar_\vdim ( U )$, and $\varepsilon >
	0$.

	Then then following two statements hold.
	\begin{enumerate}
		\item \label{item:basic_rect_var:var} There exists an $\vdim$
		dimensional submanifold $M$ of class~$1$ of~$\rel^\adim$ with
		$\| V \| ( U \without M ) \leq \varepsilon$.
		\item \label{item:basic_rect_var:fct} If $Y$ is a finite
		dimensional normed vectorspace, $f$ is a $Y$~valued $\| V \|$
		measurable function and $A$ is set of points at which $f$ is
		$( \| V \|, \vdim )$ approximately differentiable, then there
		exists $g : U \to Y$ of class~$1$ such that
		\begin{equation*}
			\| V \| ( A \without \{ x \with f (x) = g (x) \} )
			\leq \varepsilon.
		\end{equation*}
	\end{enumerate}
\end{lemma}
\begin{proof}
	The problem may reduced to the case $\| V \| ( U ) < \infty$.
	Concerning \eqref{item:basic_rect_var:var}, it is then sufficient to
	construct, using \cite[2.2.5]{MR41:1976} and
	\cite[3.6\,(1)]{kol-men:decay.v2}, a countable disjointed family $F$
	such that each member of $F$ is a compact subset of $U$ contained in
	an $\vdim$ dimensional submanifold of class~$1$ of $\rel^\adim$ and
	$\| V \| ( U \without \bigcup F ) = 0$.  Concerning
	\eqref{item:basic_rect_var:fct}, one similarly uses
	\cite[11.1\,(2)]{snulmenn:tv.v2} to construct a countable disjointed
	family $G$ such that each member $C$ of $G$ is a compact subset of $U$
	such that $f|C$ is of class~$1$ relative to $C$ and $\| V \| ( A
	\without \bigcup G ) = 0$.
\end{proof}
\begin{comment}
	The approximation result for locally Lipschitzian functions now
	readily follows.
\end{comment}
\begin{theorem} \label{thm:approximation_lip}
	Suppose $l, \vdim, \adim \in \nat$, $\vdim \leq \adim$, $U$ is an open
	subset of $\rel^\adim$, $V \in \RVar_\vdim ( U )$, $C$ is a relatively
	closed subset of $U$, $f : U \to \rel^l$ is locally Lipschitzian,
	$\spt f \subset \Int C$, and $\varepsilon > 0$.%
	\begin{footnote}
		{The symbol $\RVar_\vdim ( U )$ denotes the set of $\vdim$
		dimensional rectifiable varifolds in $U$, see Allard
		\cite[3.5]{MR0307015}.}
	\end{footnote}

	Then there exists $g : U \to \rel^l$ of class $1$ satisfying
	\begin{equation*}
		\spt g \subset C, \quad \Lip g \leq \varepsilon + \Lip f,
		\quad \| V \| ( U \without \{ x \with f(x) = g(x) \} ) \leq
		\varepsilon.
	\end{equation*}
	Moreover, if $l=1$ and $f \geq 0$ then one may require $g \geq 0$.
\end{theorem}
\begin{proof}
    Let $X = \Int C$.  Noting \cite[11.1\,(3)]{snulmenn:tv.v2}, one obtains a
    function $h : X \to \rel^l$ of class $1$ and an $\vdim$ dimensional
    submanifold $M$ of class $1$ of $\rel^\adim$ such that
    \begin{equation*}
	M \subset X, \quad \| V \| ( X \without ( M \cap \{ x \with f(x) = h
	(x) \} ) ) < \varepsilon.
    \end{equation*}
    from \ref{lemma:basic_rect_var}.  In view of \ref{lemma:app-nonnegative},
    one may require $h \geq 0$ if $l = 1$ and $f \geq 0$. Since
    \begin{equation*}
	\Der  (f|M) (x) = \Der (h|M) (x) \quad \text{for $\mathscr{H}^\vdim$
	almost all $x \in M$ with $f(x) = h(x)$},
    \end{equation*}
    by \cite[2.8.18, 2.9.11, 3.1.5, 3.1.22]{MR41:1976}, the set
    \begin{equation*}
	B = M \cap \{ x \with \text{$f(x) = h(x)$ and $\Der (f|M)(x) = \Der
	(h|M)(x)$} \}
    \end{equation*}
    satisfies $\| V \| ( X \without B ) < \varepsilon$. From
    \ref{corollary:special-retraction} one obtains a function $r$ of class $1$
    retracting some open subset $G$ of $X$ onto $M$ such that
    \begin{equation*}
    	\Der r(a) = \project{\Tan(M,a)} \quad \text{whenever $a \in M$},
    \end{equation*}
    hence $\sup \| \Der  ( h \circ r ) \| \lIm B \rIm \leq \Lip f$. Applying
    \ref{lemma:c1_extension} with $U$, $A$, and $f$ replaced by $(U \without C
    ) \cup G$, $( U \without C) \cup B$, and $(( U \without C ) \times \{ 0 \}
    ) \cup ( h \circ r )$, one obtains a function $g : U \to \rel^l$ of class
    $1$ such that
    \begin{equation*}
	g|U \without X = 0, \quad g|B = f|B, \quad \Lip g \leq \varepsilon +
	\Lip f
    \end{equation*}
    and $g \geq 0$ if $l = 1$ and $f \geq 0$.
\end{proof}
\begin{comment}
	For functions with compact support the result takes the following
	form.
\end{comment}
\begin{corollary} \label{corollary:approximation_lip}
    Suppose $l, \vdim, \adim \in \nat$, $\vdim \leq \adim$, $U$ is an open
    subset of $\rel^\adim$, $V \in \RVar_\vdim ( U )$, $K$ is a compact subset
    of $U$, and $f: U \to \rel^l$ is a Lipschitzian function with $\spt f
    \subset \Int K$.

    Then there exists a sequence $f_i \in \mathscr{D} (U, \rel^l)$ satisfying
    \begin{gather*}
	f_i(x) \to f(x) \quad \text{uniformly for $x \in \spt \| V \|$ as $i
	\to \infty$}, \\
	\big \| ( \| V \|, \vdim ) \ap \Der (f_i-f) \big \| \to 0 \quad
	\text{in $\| V \|$ measure as $i \to \infty$}, \\
	\spt f_i \subset K \quad \text{for $i \in \nat$}, \qquad \limsup_{i
	\to \infty} \Lip f_i \leq \Lip f.
    \end{gather*}
    Moreover, if $l=1$ and $f \geq 0$ one may require $f_i \geq 0$ for $i \in
    \nat$.
\end{corollary}
\begin{proof}
    The problem may be reduced firstly to the construction of functions $f_i$
    of class $1$ by means of convolution and secondly to establishing
    convergence of $f_i$ to $f$ in $\| V \|$ measure as $i \to \infty$ by
    \cite[2.10.21]{MR41:1976}.  In view of \cite[11.1\,(4)]{snulmenn:tv.v2}, the
    conclusion now follows from \ref{thm:approximation_lip}.
\end{proof}
\begin{remark} \label{remark:approximation_lip}
    In the preceding statement $\rel^l$ may be replaced by a finite
    dimensional normed vectorspace $Y$ provided ``$\Lip f$'' is replaced by
    ``$\Gamma \Lip f$'' in the conclusion, where $\Gamma$ is a positive,
    finite number depending only on $Y$.
\end{remark}
\section{Rellich type embeddings} \label{sec:rellich}
In the present section a Rellich type compactness result for generalised
weakly differentiable functions will be established in
\ref{thm:rellich-in-measure}.  As a consequence one obtains a sequential
closedness result under weak convergence for the nonlinear space of
generalised weakly differentiable functions with and without ``boundary
conditions'', see \ref{corollary:closedness-tv} and
\ref{corollary:closedness_tg}.

The key is to quantify the approximability on large sets by Lipschitzian
functions obtained in \cite[11.1, 11.2]{snulmenn:tv.v2} by means of maximal
function techniques.  This procedure is complicated by the fact that Sobolev
Poincar{\'e} inequalities in their usual form with one median are only
available near almost every point centred at that point on all scales below a
certain threshold which depends on the point and the varifold considered in a
rather nonuniform way. An instructive example of qualitative nature is given
by Brakke in \cite[6.1]{MR485012}, two quantified forms of which were
presented by Kolasi{\'n}ski and the author in \cite[10.3,
10.8]{kol-men:decay.v2}.
\begin{comment}
	Firstly, a simple closedness result assuming convergence locally in
	measure of the functions and local weak convergence of the generalised
	derivative is recorded.

	Several proofs in this section employ the duality of Lebesgue spaces,
	see \cite[2.5.7]{MR41:1976}, and properties of weak convergence in
	general, see \cite[\printRoman 2.3.27]{MR0117523}, and in $L_1 ( \mu,
	Y )$ in particular, see \ref{example:quotient-lp-spaces} and
	\cite[\printRoman 4.8.9--\printRoman 4.8.12]{MR0117523}.
\end{comment}
\begin{lemma} \label{lemma:weak_closedness_tv}
	Suppose $\vdim, \adim \in \nat$, $\vdim \leq \adim$, $U$ is an open
	subset of $\rel^\adim$, $V \in \RVar_\vdim (U)$, $\| \delta V \|$ is a
	Radon measure, $Y$ is a finite dimensional normed vectorspace,%
	\begin{footnote}
		{Whenever $\mu$ is a measure and $Y$ is a separable Banach
		space, $\mathbf{A} ( \mu, Y )$ equals the vectorspace of $\mu$
		measurable functions with values in $Y$, see
		\cite[2.3.8]{MR41:1976}.}
	\end{footnote}
	\begin{equation*}
		f \in \mathbf{A} ( \| V \| + \| \delta V \|, Y ), \quad F \in
		\Lploc 1 ( \| V \|, \Hom ( \rel^\adim, Y ) ),
	\end{equation*}
	and $f_i \in \trunc (V,Y)$ is a sequence with $\derivative V{f_i} \in
	\Lploc 1 ( \| V \|, \Hom ( \rel^\adim, Y ) )$ satisfying
	\begin{gather*}
		f_i \to f \quad \text{in $( \| V \| + \| \delta V \| )
		\restrict K$ measure as $i \to \infty$}, \\
		\lim_{i \to \infty} \tint K{} \langle \derivative V{f_i}, G
		\rangle \ud \| V \| = \tint K{} \langle F,G \rangle \ud \| V
		\| \quad \text{for $G \in \Lp \infty \big ( \| V \|, \Hom
		(\rel^\adim, Y )^\ast \big )$}
	\end{gather*}
	whenever $K$ is a compact subset of $U$.

	Then $f \in \trunc (V,Y)$ and
	\begin{equation*}
		F(x) = \derivative Vf(x) \quad \text{for $\| V \|$ almost all
		$x$}.
	\end{equation*}
\end{lemma}
\begin{proof}
	Recall \cite[2.5.7\,(ii)]{MR41:1976} and
	\cite[\printRoman{2}.3.27]{MR0117523}. Since
	\begin{equation*}
		\lim_{i \to \infty} \tint{}{} \langle \theta (x), \Der  \gamma
		( f_i (x) ) \circ \derivative V{f_i} (x) \rangle \ud \| V \|
		\, x = \tint{}{} \langle \theta (x), \Der  \gamma ( f(x))
		\circ F(x) \rangle \ud \| V \| \, x
	\end{equation*}
	whenever $\theta \in \mathscr{D} (U,\rel^\adim )$, $\gamma \in
	\mathscr{E} (Y,\rel)$ and $\spt \Der  \gamma$ is compact by
	\cite[\printRoman{4}.8.10, \printRoman{4}.8.11]{MR0117523}, the
	conclusion is readily verified by means of \cite[8.3]{snulmenn:tv.v2}.
\end{proof}
\begin{remark} \label{remark:weak_closedness_tv}
	If $\| \delta V \|$ is absolutely continuous with respect to $\| V
	\|$, one may replace ``$\| V \| + \| \delta V \|$'' by ``$\| V \|$''
	in the preceding lemma; in fact, \cite[2.4.11, 2.8.18, 2.9.2,
	2.9.7]{MR41:1976} implies that $f$ is $\| \delta V \|$ measurable and
	that $f_i$ converges to $f$ in $\| \delta V \| \restrict K$ measure as
	$i \to \infty$ whenever $K$ is a compact subset of $U$.
\end{remark}
\begin{miniremark} \label{miniremark:compactness-pseudometric-spaces}
    The following proposition is an elementary fact about pseudometric spaces.
    \emph{If $Z$ is the space of a complete pseudometric $\varrho$, $f_i$ is a
    sequence in $Z$ and for every $\varepsilon > 0$ there exists a sequence
    $g_i$ in $Z$ with $\varrho (f_i,g_i) \leq \varepsilon$ for $i \in \nat$
    such that each subsequence of $g_i$ admits a convergent subsequence, then
    $f_i$ possesses a convergent subsequence.}
\end{miniremark}
\begin{comment}
	Next, an elementary but useful criterion for sequential compactness
	for local convergence in measure, ultimately based on the
	Ascoli theorem for Lipschitzian functions, see
	\cite[2.10.21]{MR41:1976}, is proven.
\end{comment}
\begin{lemma} \label{lemma:compactness-in-measure}
	Suppose $U$ is an open subset of $\rel^\adim$, $\mu$ is a Radon
	measure over $U$, $Y$ is a finite dimensional normed vectorspace,
	$f_i$ is a sequence in $\mathbf{A} ( \mu, Y )$ such that whenever $X$
	is a $\mu$ measurable set with $\mu (X)<\infty$ and $\varepsilon > 0$
	there exists $\kappa < \infty$ such that for each $i \in \nat$ there
	exists a subset $A$ of $X$ with $\mu ( X \without A) \leq \varepsilon$
	and $\sup | f_i | \lIm A \rIm + \Lip (f_i|A) \leq \kappa$.

	Then there exist $f \in \mathbf{A} (\mu,Y)$ and a subsequence of $f_i$
	which, whenever $K$ is a compact subset of $U$, converges to $f$ in
	$\mu \restrict K$ measure.
\end{lemma}
\begin{proof}
	One may assume $\mu (U) < \infty$ and $Y = \rel$. Suppose $\varepsilon
	> 0$. Taking $\kappa$ as in the hypotheses for $X = U$, one constructs
	functions $g_i : U \to \rel$ with $\sup \im | g_i | + \Lip g_i \leq
	\kappa$ and $\mu ( U \without \{ x \with f_i (x) = g_i (x) \}) \leq
	\varepsilon$ for $i \in \nat$ by \cite[2.10.44, 4.1.16]{MR41:1976}, in
	particular $\mdistance{f_i-g_i}{\mu}{} \leq \varepsilon$.%
	\begin{footnote}
		{Whenever $\mu$ is a measure and $Y$ is a separable Banach
		space
		\begin{equation*}
			\mdistance f \mu = \inf \big \{ t \with \mu ( \{ x
			\with |f(x)| > t \} ) \leq t \big \} \quad \text{for
			$f \in \mathbf{A} ( \mu, Y)$},
		\end{equation*}
		see \cite[2.3.8]{MR41:1976}.}
	\end{footnote}
	Therefore, in view of \cite[2.3.8, 2.3.10, 2.10.21]{MR41:1976}, one
	may apply \ref{miniremark:compactness-pseudometric-spaces} with $Z$
	and $\varrho (f,g)$ replaced by $\mathbf{A} (\mu,\rel)$ and
	$\mdistance{f-g}{\mu}$ to obtain the conclusion.
\end{proof}
\begin{comment}
	For the use in the present and the next section, a set of hypotheses
	is collected.
\end{comment}
\begin{miniremark} \label{miniremark:minimum_conditions}
	Suppose $\vdim, \adim \in \nat$, $\vdim \leq \adim$, $U$ is an
	open subset of $\rel^\adim$, $V \in \Var_\vdim ( U )$, $\| \delta V
	\|$ is a Radon measure, and $\density^\vdim ( \| V \|, x ) \geq 1$ for
	$\| V \|$ almost all~$x$. In particular, $V$ is rectifiable by Allard
	\cite[5.5\,(1)]{MR0307015}.
\end{miniremark}
\begin{comment}
	The key estimate for the Rellich type embedding result will now be
	formulated. It is based on the Sobolev Poincar{\'e} type inequalities
	obtained in \cite[10.1]{snulmenn:tv.v2}.
\end{comment}
\begin{lemma} \label{lemma:average}
	Suppose $\vdim$, $\adim$, $U$, and $V$ are as in
	\ref{miniremark:minimum_conditions}, $\adim \leq M < \infty$, $a \in
	\rel^\adim$, $0 < r < \infty$, $1 < \lambda \leq 2$, $U = \oball
	a{\lambda r}$, $1 \leq Q \leq M$, $0 \leq \kappa < \infty$, $f \in
	\trunc (V)$,
	\begin{gather*}
		\beta = \infty \quad \text{if $\vdim = 1$}, \qquad \beta =
		\vdim/(\vdim-1) \quad \text{if $\vdim > 1$}, \\
		\Lambda = 2 \Gamma_{\textup{\cite[10.1]{snulmenn:tv.v2}}} ( M)
		\big ( 1 + 2^{\vdim+3} M^{1/\beta} \big ), \\
		C = \{ (x,\cball xs) \with \text{$x \in \oball a{\lambda r}$
		and $0 < s + |x-a| < \lambda r$} \}, \\
		\text{$f$ is $(\| V \|, C )$ approximately continuous at $a$},
	\end{gather*}
	and, for $0 < s \leq r$,
	\begin{gather*}
		\measureball{\| V \|}{\cball as} \geq (1/2) \unitmeasure \vdim
		s^\vdim, \quad \measureball{\| V \|}{\oball a{\lambda s}} \leq
		2 ( Q-M^{-1} ) \unitmeasure \vdim s^\vdim, \\
		\| V \| ( \oball a {\lambda s} \cap \{ x \with \density^\vdim
		( \| V \|, x ) < Q \} ) \leq
		\Gamma_{\textup{\cite[10.1]{snulmenn:tv.v2}}} (M)^{-1}
		s^\vdim, \\
		\tint{\oball a{\lambda s}}{} | \derivative Vf | \ud \| V \| +
		\eqLpnorm{\| V \| \restrict \oball a{\lambda s}}{\infty}{f}
		\measureball{\| \delta V \|}{\oball a{\lambda s}} \leq \kappa
		s^\vdim;
	\end{gather*}
	here $0 \cdot \infty = \infty \cdot 0 = 0$.

	Then there holds
	\begin{equation*}
		\eqLpnorm{\| V \| \restrict \cball a{(\lambda-1)s}}{\beta}{f (
		\cdot)- f(a)} \leq \Lambda \kappa s^\vdim \quad \text{for $0 <
		s \leq r$}.
	\end{equation*}
\end{lemma}
\begin{proof}
    Abbreviate $\Delta = \Gamma_{\textup{\cite[10.1]{snulmenn:tv.v2}}} ( M )$.
    Choose $y(s) \in \rel$ such that
    \begin{gather*}
	\| V \| ( \oball a {\lambda s} \cap \{ x \with f(x) < y(s) \} ) \leq
	(1/2) \measureball{\| V \|}{\oball a{\lambda s}}, \\
	\| V \| ( \oball a {\lambda s} \cap \{ x \with f(x) > y(s) \} ) \leq
	(1/2) \measureball{\| V \|}{\oball a{\lambda s}}
    \end{gather*}
    for $0 < s \leq r$, in particular
    \begin{equation*}
	|y(s)| \leq \eqLpnorm{\| V \| \restrict \oball a{\lambda
	s}}{\infty}{f} \quad \text{and} \quad f(a) = \lim_{s \to 0+} y(s).
    \end{equation*}
    Define $g_{s} (x) = f(x) - y(s)$ whenever $0 < s \leq r$ and $x \in \dmn
    f$, hence
    \begin{equation*}
	\eqLpnorm{\| \delta V \| \restrict \oball a{\lambda s}}{\infty}{g_s}
	\leq 2 \eqLpnorm{\| V \| \restrict \oball a{\lambda s}}{\infty}{f}
    \end{equation*}
    by \cite[8.33]{snulmenn:tv.v2}. Recalling \cite[8.12, 8.13\,(4),
    9.2]{snulmenn:tv.v2}, one applies \cite[10.1\,(1a)]{snulmenn:tv.v2} with
    $p$, $U$, $G$, $f$, and $r$ replaced by $1$, $\oball a{\lambda s}$,
    $\varnothing$, $g_{s}^+| \oball a{\lambda s}$ respectively $g_{s}^-|
    \oball a{\lambda s}$, and $s$ to infer
    \begin{align*}
	\eqLpnorm{\| V \| \restrict \cball a{(\lambda-1) s} }{\beta}{g_{s}}
	& \leq \Delta \big ( \tint{\oball{a}{\lambda s}}{} | \derivative Vf
	| \ud \| V \| + \tint{\oball a{\lambda s}}{} | g_{s}| \ud \| \delta V
	\| \big ) \\
	& \leq 2 \Delta \kappa s^\vdim
    \end{align*}
    for $0 < s \leq r$. Finally, one estimates
    \begin{gather*}
	\begin{aligned}
	    & | y(s)-y(s/2)| \cdot \| V \| ( \cball a{(\lambda-1)s/2}
	    )^{1/\beta} \\
	    & \leq \eqLpnorm{\| V \| \restrict \cball a {(\lambda-1)s/2}
	    }{\beta}{g_{s/2}} + \eqLpnorm{\| V \| \restrict \cball
	    {a}{(\lambda-1)s}}{\beta}{g_{s}} \leq 4 \Delta \kappa s^\vdim,
	\end{aligned} \\
	| y(s)-f(a)| \leq 2^{\vdim+3} \Delta \unitmeasure{\vdim}^{-1/\beta} (
	\lambda-1 )^{1-\vdim} \kappa s, \\
	\eqLpnorm{\| V \| \restrict \cball{a}{(\lambda-1)s}
	}{\beta}{f(\cdot)-f(a) } \leq \Lambda \kappa s^\vdim
    \end{gather*}
    for $0 < s \leq r$.
\end{proof}
\begin{remark}
	The preceding lemma is extracted from the proof of
	\cite[11.2]{snulmenn:tv.v2}.
\end{remark}
\begin{comment}
	Proving the following Rellich type embedding result now mainly amounts
	to applying the preceding two lemmata, defining the necessary
	parameters in the appropriate order in this process, and using
	Egoroff's theorem, see \cite[2.3.7]{MR41:1976}, to construct large
	sets on which certain conditions are satisfied uniformly.
\end{comment}
\begin{theorem} \label{thm:rellich-in-measure}
    Suppose $\vdim$, $\adim$, $U$, and $V$ are as in
    \ref{miniremark:minimum_conditions}, $Y$ is a finite dimensional normed
    vectorspace, and $f_i \in \trunc (V,Y)$ is a sequence satisfying
    \begin{gather*}
	\lim_{t \to \infty} \sup \big \{ \| V \| ( K \cap \{ x \with |f_i (x)
	| > t \} ) \with i \in \nat \big \} = 0, \\
	\sup \big \{ \tint{K \cap \{ x \with |f_i(x)| \leq t \}}{} \|
	\derivative V{f_i} \| \ud \| V \| \with i \in \nat \big \} < \infty
	\quad \text{for $0 \leq t < \infty$}
    \end{gather*}
    whenever $K$ is a compact subset of $U$.

    Then there exist $f \in \mathbf{A} ( \| V \|, Y )$ and a subsequence of
    $f_i$ which, whenever $K$ is a compact subset of $U$, converges to $f$ in
    $\| V \| \restrict K$ measure.
\end{theorem}
\begin{proof}
    The proof will be conducted by verifying the hypotheses of
    \ref{lemma:compactness-in-measure} with $\mu$ replaced by $\| V \|$. For
    this purpose suppose $X$ is a $\| V \|$ measurable set with $\| V \| (X) <
    \infty$. Defining $C = \{ (a,\cball ar) \with \cball ar \subset U \}$, one
    may assume that for some $M$ with $\sup \{ 4, \adim \} \leq M < \infty$
    there holds
    \begin{gather*}
	1 \leq \density^\vdim ( \| V \|, x ) \leq M, \\
	\text{$\density^\vdim ( \| V \|, \cdot )$ and $f_i$ are $(\| V \|, C)$
	approximately continuous at $x$}
    \end{gather*}
    whenever $x \in X$ and $i \in \nat$ by \cite[2.8.18, 2.9.13]{MR41:1976}
    and that $X$ is compact.

    Choose a compact subset $K$ of $U$ with $X \subset \Int K$ and let
    \begin{equation*}
    	\delta = \inf \{ \dist (x, \rel^\adim \without K) \with x \in X \}.
    \end{equation*}
    Define $\lambda = (1.1)^{1/\vdim}$, hence $1 < \lambda \leq 2$, let
    \begin{equation*}
    	Q(x) = \sup \{ 1, (5/6) \density^\vdim ( \| V \|, x ) \} \quad
	\text{whenever $x \in X$}
    \end{equation*}
    and notice that
    \begin{equation*}
	\density^\vdim ( \| V \|, x ) < 2 \lambda^{-\vdim} ( Q(x)-M^{-1} )
	\quad \text{for $x \in X$}.
    \end{equation*}
    Abbreviate $\Delta_1 = \Gamma_{\textup{\cite[10.1]{snulmenn:tv.v2}}}
    (M)^{-1}$.
    
    Suppose $\varepsilon > 0$.
    
    In order to define $\kappa$, first observe that one may construct, by
    means of \cite[2.3.7, 2.6.2]{MR41:1976}, a $\| V \|$ measurable subset
    $X'$ of $X$ with $\| V \| ( X \without X' ) \leq \varepsilon/3$ and $0 < r
    \leq \delta/2$ satisfying
    \begin{gather*}
	\measureball{\| V \|}{\cball xs} \geq (1/2) \unitmeasure \vdim
	s^\vdim, \quad \measureball{\| V \|}{\oball x{\lambda s}} \leq 2 (
	Q(x)-M^{-1} ) \unitmeasure \vdim s^\vdim, \\
	\| V \| ( \oball x {\lambda s} \cap \{ \chi \with \density^\vdim ( \|
	V \|, \chi ) < Q(x) \} ) \leq \Delta_1 s^\vdim
    \end{gather*}
    for $x \in X'$ and $0 < s \leq r$. Choose $0 \leq \Delta_2 < \infty$ such
    that
    \begin{equation*}
	\| V \| ( X \cap \{ x \with |f_i(x)| > \Delta_2 \} ) \leq
	\varepsilon/3 \quad \text{for $i \in \nat$}
    \end{equation*}
    and $g \in \mathscr{D} (Y,Y)$ with $g(y)=y$ whenever $y \in \cball
    0{\Delta_2}$ and $\sup \im |g| \leq 2 \Delta_2$. Define $h_i = g \circ
    f_i$ and notice that \cite[8.12]{snulmenn:tv.v2} implies that $h_i \in
    \trunc (V,Y)$ and
    \begin{equation*}
	\Delta_3 = \sup \big \{ \tint K{} \| \derivative V{h_i} \| \ud \| V \|
	+ \| \delta V \| (K) \with i \in \nat \big \} < \infty.
    \end{equation*}
    Define $\beta$ and $\Lambda$ to be related to $\vdim$ and $M$ as in
    \ref{lemma:average}, and let
    \begin{align*}
	\Delta_4 & = 8 M ( 1+2\Delta_2) \Delta_3 \unitmeasure{\vdim}
	\besicovitch \adim \varepsilon^{-1} \Lambda, \\
	\kappa & = \sup \big \{ 2\Delta_2, 2^{\vdim+3} ( \lambda-1)^{-\vdim}
	\unitmeasure{\vdim}^{-1/\beta} \Delta_4, 8 \Delta_2 (\lambda-1)^{-1}
	r^{-1} \big \}.
    \end{align*}

    Suppose $i \in \nat$.

    Defining $F : X \to \overline{\rel}$ by
    \begin{equation*}
	F (x) = \sup \left \{ \frac{\tint{\cball xs}{} \| \derivative V h_i
	\| \ud \| V \| + \measureball{\| \delta V \|}{\cball
	xs}}{\measureball{\| V \|}{\cball xs}} \with 0 < s \leq \delta \right
	\}
    \end{equation*}
    for $x \in X$, let
    \begin{equation*}
	A = X' \cap \big \{ x \with \text{$|f_i(x)| \leq \Delta_2$ and $F(x)
	\leq 3 \Delta_3 \besicovitch \adim \varepsilon^{-1}$} \big \}
    \end{equation*}
    and observe $\| V \| ( X \without A ) \leq \varepsilon$. Noting $f_i|A =
    h_i|A$ and $\Delta_2 \leq \kappa/2$, it is sufficient (by
    \cite[\printRoman{2}.3.15]{MR0117523}) to prove that $\Lip ( \alpha
    \circ h_i|A ) \leq \kappa/2$ whenever $\alpha \in \Hom ( Y, \rel)$ and $\|
    \alpha \| \leq 1$. Noting \cite[8.18]{snulmenn:tv.v2}, one applies
    \ref{lemma:average} with $Q$, $\kappa$, and $f$ replaced by $Q(a)$, $2M
    (1+2\Delta_2) \unitmeasure \vdim F(a)$, and $\alpha \circ h_i| \oball
    a{\lambda s}$ to infer
    \begin{equation*}
	\eqLpnorm{\| V \| \restrict \cball{a}{(\lambda-1)s} }{\beta}{( \alpha
	\circ h_i ) ( \cdot ) - (\alpha \circ h_i)(a) )} \leq \Delta_4 s^\vdim
    \end{equation*}
    for $a \in A$ and $0 < s \leq r$. Therefore, if $a, x \in A$ and $|x-a|
    \leq ( \lambda-1)r/2$, then taking $s = 2 |x-a|/(\lambda-1)$ yields
    \begin{align*}
	& |(\alpha \circ h_i)(x) -( \alpha \circ h_i)(a)| \\
	& \qquad \leq 2^{\vdim+1} (\lambda-1)^{-\vdim}
	\Delta_4 \| V \| ( \cball a{2|x-a|} \cap
	\cball x{2|x-a|} )^{-1/\beta} |x-a|^\vdim \\
	& \qquad \leq 2^{\vdim+2} ( \lambda-1)^{-\vdim}
	\unitmeasure{\vdim}^{-1/\beta} \Delta_4 |x-a| \leq ( \kappa/2 ) |x-a|.
    \end{align*}
    Finally, notice that $|(\alpha \circ h_i) (x)- (\alpha \circ h_i)(a)| \leq
    2 \Delta_2 \leq (\kappa/2) |x-a|$ whenever $a, x \in A$ and $|x-a| >
    (\lambda-1)r/2$.
\end{proof}
\begin{comment}
	Combining the theorem with basic properties of weak convergence,
	the first corollary concerning the closedness of the space of
	generalised weakly differentiable functions is readily verified.
\end{comment}
\begin{corollary} \label{corollary:closedness-tv}
	Suppose $\vdim$, $\adim$, $U$, and $V$ are as in
	\ref{miniremark:minimum_conditions}, $Y$ is a finite dimensional
	normed vectorspace,
	\begin{equation*}
		f \in \mathbf{A} ( \| \delta V \|, Y ) \cap \Lploc{1} ( \| V
		\|, Y ), \quad F \in \Lploc{1} ( \| V \|, \Hom ( \rel^\adim,
		Y) ),
	\end{equation*}
	and $f_i \in \trunc (V,Y) \cap \Lploc{1} ( \| V \|, Y)$ is a sequence
	satisfying
	\begin{gather*}
		\derivative V{f_i} \in \Lploc{1} ( \| V \|, \Hom ( \rel^\adim,
		Y) ) \quad \text{for $i \in \nat$}, \\
		\lim_{i \to \infty} \tint{K}{} \langle f_i, g \rangle \ud \| V
		\| = \tint{K}{} \langle f, g \rangle \ud \| V \| \quad
		\text{for $g \in \Lp{\infty} ( \| V \|, Y^\ast )$}, \\
		\lim_{i \to \infty} \tint{K}{} \langle \derivative V{f_i}, G
		\rangle \ud \| V \| = \tint{K}{} \langle F, G \rangle \ud \| V
		\| \quad \text{for $G \in \Lp{\infty} \big ( \| V \|, \Hom (
		\rel^\adim, Y)^\ast \big )$}, \\
		f_i \to f \quad \text{in $( \| \delta V \| - \| \delta V
		\|_{\| V \|}) \restrict K$ measure as $i \to \infty$}
	\end{gather*}
	whenever $K$ is a compact subset of $U$, see page
	\pageref{page:psi_phi}.

	Then $f \in \trunc (V,Y)$ and
	\begin{gather*}
		F(x) = \derivative Vf(x) \quad \text{for $\| V \|$ almost all
		$x$}, \\
		\lim_{i \to \infty} \eqLpnorm{\| V \| \restrict K}{1}{f_i-f} =
		0 \quad \text{whenever $K$ is a compact subset of $U$}.
	\end{gather*}
\end{corollary}
\begin{proof}
    Recall \cite[2.5.7\,(ii)]{MR41:1976} and
    \cite[\printRoman{2}.3.27]{MR0117523}. Applying
    \ref{thm:rellich-in-measure} and \cite[\printRoman{4}.8.12]{MR0117523}
    yields
    \begin{equation*}
	\lim_{i \to \infty} \eqLpnorm{\| V \| \restrict K}{1}{f_i-f} = 0 \quad
	\text{whenever $K$ is a compact subset of $U$},
    \end{equation*}
    hence \cite[2.4.11, 2.8.18, 2.9.7]{MR41:1976} implies
    \begin{equation*}
    	f_i \to f \quad \text{in $( \| V \| + \| \delta V \| ) \restrict K$
	measure as $i \to \infty$}
    \end{equation*}
    whenever $K$ is a compact subset of $U$ and the conclusion follows from
    \ref{lemma:weak_closedness_tv}.
\end{proof}
\begin{comment}
	Taking the more basic closedness result obtained in
	\cite[9.13]{snulmenn:tv.v2} into account, the second corollary
	involving a ``boundary condition'' follows similarly.
\end{comment}
\begin{corollary} \label{corollary:closedness_tg}
	Suppose $\vdim$, $\adim$, $U$, and $V$ are as in
	\ref{miniremark:minimum_conditions}, $G$ is a relatively open subset
	of $\Bdry U$, $B = ( \Bdry U ) \without G$,
	\begin{gather*}
		0 \leq f \in \mathbf{A} ( \| \delta V \|, \rel ) \cap
		\mathbf{A} ( \| V \|, \rel ), \quad F \in \mathbf{A} ( \| V
		\|, \Hom ( \rel^\adim, \rel ) ), \\
		\| \delta V \| ( U \cap K ) < \infty, \quad \tint{K}{} f +
		|F| \ud \| V \| < \infty
	\end{gather*}
	whenever $K$ is a compact subset of $\rel^\adim \without B$, and $f_i
	\in \trunc_G (V)$ is a sequence satisfying
	\begin{gather*}
		\tint{U \cap K}{} f_i + | \derivative V{f_i} | \ud \| V \| <
		\infty \quad \text{for $i \in \nat$}, \\
		\lim_{i \to \infty} \tint{U \cap K}{} f_i g \ud \| V \| =
		\tint{U \cap K}{} f g \ud \| V \| \quad \text{for $g \in
		\Lp{\infty} ( \| V \| )$}, \\
		\lim_{i \to \infty} \tint{U \cap K}{} \langle \theta,
		\derivative V{f_i} \rangle \ud \| V \| = \tint{U \cap K}{}
		\langle \theta,F \rangle \ud \| V \| \quad \text{for $\theta
		\in \Lp{\infty} ( \| V \|, \rel^\adim )$}, \\
		f_i \to f \quad \text{in $( \| \delta V \| - \| \delta V
		\|_{\| V \|}) \restrict U \cap K$ measure as $i \to \infty$}
	\end{gather*}
	whenever $K$ is a compact subset of $\rel^\adim \without B$, see page
	\pageref{page:psi_phi}.

	Then $f \in \trunc_G (V)$ and
	\begin{gather*}
		F(x) = \derivative Vf(x) \quad \text{for $\| V \|$ almost all
		$x$}, \\
		\lim_{i \to \infty} \eqLpnorm{\| V \| \restrict U \cap
		K}{1}{f_i-f} = 0
	\end{gather*}
	whenever $K$ is a compact subset of $\rel^\adim \without B$.
\end{corollary}
\begin{proof}
	Recall \cite[2.5.7\,(ii)]{MR41:1976} and
	\cite[\printRoman{2}.3.27]{MR0117523}.

	Applying \ref{corollary:closedness-tv} with $Y = \rel$ yields $f \in
	\trunc (V)$ with $F$ and $\derivative Vf$ being $\| V \|$ almost equal
	and, in combination with \cite[\printRoman{4}.8.9,
	\printRoman{4}.8.10]{MR0117523}, also
	\begin{equation*}
		\lim_{i \to \infty} \eqLpnorm{\| V \| \restrict U \cap
		K}{1}{f_i-f} = 0
	\end{equation*}
	whenever $K$ is a compact subset of $\rel^\adim \without B$. By
	\cite[2.4.11, 2.8.18, 2.9.7]{MR41:1976} this implies
	\begin{equation*}
		f_i \to f \quad \text{in $( \| V \| + \| \delta V \| )
		\restrict U \cap K$ measure as $i \to \infty$}
	\end{equation*}
	whenever $K$ is a compact subset of $\rel^\adim \without B$, in
	particular for such $K$
	\begin{equation*}
		\limsup_{i \to \infty} \| V \| ( K \cap A \cap \{ x \with f_i
		(x)>y \}) \leq \| V \| ( K \cap A \cap \{ x \with f(x) > b \}
		)
	\end{equation*}
	whenever $A$ is $\|V \|$ measurable and $0 < b < y < \infty$. Noting
	\cite[\printRoman{4}.8.10, \printRoman{4}.8.11]{MR0117523}, the
	conclusion now follows from \cite[9.13]{snulmenn:tv.v2}.
\end{proof}
\begin{remark}
	If $\| \delta V \|$ is absolutely continuous with respect to $\| V
	\|$, then the hypotheses of $\| \delta V \|$ measurability of $f$ and
	convergence in $( \| \delta V \| - \| \delta V \|_{\| V \|} )
	\restrict K$ measure in \ref{corollary:closedness-tv} and
	\ref{corollary:closedness_tg} are evidently redundant by \cite[2.8.18,
	2.9.2, 2.9.7]{MR41:1976}.
\end{remark}
\section{Sobolev spaces}
In this section, mainly definitions and basic properties of Sobolev spaces
with respect to certain rectifiable varifolds for functions with values in a
finite dimensional normed space are provided, see
\ref{def:strict_local_sobolev_space}--\ref{remark:quotient-zero-sobolev-space}.
In the proof of the deeper properties of these spaces their link to the spaces
of generalised weakly differentiable functions will be used heavily. The
relation to generalised weakly differentiable functions is immediate in case
of local Sobolev spaces, see
\ref{remark:strict_local_sobolev_space_inclusion}, and takes the form of a
theorem for Sobolev functions with ``zero boundary values'', see
\ref{thm:zero_implies_zero}.  A first example of the utility of this
link is provided by the Sobolev inequality for Sobolev functions with ``zero
boundary values'' in \ref{corollary:sob_poin_summary}.
\begin{comment}
	Firstly, the local Sobolev space is defined as a vectorspace; its
	topology will be defined only after some basic properties are
	established.
\end{comment}
\begin{definition} \label{def:strict_local_sobolev_space}
    Suppose $\vdim, \adim \in \nat$, $\vdim \leq \adim$, $U$ is an open subset
    of $\rel^\adim$, $Y$ is a finite dimensional normed vectorspace, $V \in
    \RVar_\vdim (U)$, $\| \delta V \|$ is a Radon measure, and $1 \leq q \leq
    \infty$.

    Then the \emph{local Sobolev space with respect to $V$ and exponent $q$},
    denoted by $\SWloc{q} ( V , Y )$, is defined to be the vectorspace
    consisting of all $f \in \Lploc{q} ( \| V \|+\| \delta V\|, Y )$ such that
    there exists $F \in \Lploc{q} ( \| V \|, \Hom ( \rel^\adim, Y ))$ with the
    following property. If $K$ is a compact subset of $U$ and $\varepsilon >
    0$ then
    \begin{equation*}
	\eqLpnorm{(\| V \| + \| \delta V \| )\restrict K}{q}{f-g} +
	\eqLpnorm{\| V \| \restrict K}{q}{F-\derivative{V}{g} } \leq
	\varepsilon
    \end{equation*}
    for some locally Lipschitzian function $g:U \to Y$. Abbreviate $\SWloc{q}
    (V,\rel) = \SWloc{q}(V)$.
\end{definition}
\begin{remark} \label{remark:strict_local_sobolev_space_inclusion}
    Notice that \cite[8.7]{snulmenn:tv.v2} and \ref{lemma:weak_closedness_tv}
    imply
    \begin{equation*}
	Y^U \cap \{ f \with \text{$f$ is locally Lipschitzian} \} \subset
	\SWloc{q} (V,Y) \subset \trunc (V,Y)
    \end{equation*}
    and $F$ is $\| V \|$ almost equal to $\derivative Vf$.
\end{remark}
\begin{remark} \label{remark:sobolev-simpler-def}
	In some cases the definition may be reformulated.
	\begin{enumerate}
		\item \label{item:sobolev-simpler-def:q=oo} If $\| \delta V
		\|$ is absolutely continuous with respect to $\| V \|$ and
		$q = \infty$, then ``$\| V \| + \| \delta V \|$'' may be
		replaced by ``$\| V \|$''.
		\item \label{item:sobolev-simpler-def:cpt} One may require $g$
		to have compact support.
		\item \label{item:sobolev-simpler-def:q<oo} If $q < \infty$,
		then one may require $g \in \mathscr{D} (U,Y)$. If
		additionally $Y = \rel$ and $f \geq 0$, then one may in turn
		also require $g \geq 0$.
		\item \label{item:sobolev-simpler-def:q=oo_sign} If $q =
		\infty$, $Y = \rel$, and $f \geq 0$ then one may require $g
		\geq 0$.
		\item \label{item:sobolev-simpler-def:family} The family of
		all compacts subsets of $U$ may be replaced by a family of
		compact subsets of $U$ whose interiors cover $U$.
	\end{enumerate}
	\eqref{item:sobolev-simpler-def:q=oo},
	\eqref{item:sobolev-simpler-def:cpt}, and
	\eqref{item:sobolev-simpler-def:q=oo_sign} are evident.
	\eqref{item:sobolev-simpler-def:q<oo} follows from
	\eqref{item:sobolev-simpler-def:cpt},
	\ref{corollary:approximation_lip}, \ref{remark:approximation_lip}, and
	\cite[8.7]{snulmenn:tv.v2}. \eqref{item:sobolev-simpler-def:family}
	may be verified by means of a partition of unity.
\end{remark}
\begin{remark} \label{remark:continuity_w_infty}
	If $f \in \SWloc{\infty} (V,Y)$ then there exists a continuous
	function $g : \spt \| V \| \to Y$ such that $f(x)=g(x)$ for $\| V \| +
	\| \delta V \|$ almost all $x$. However, modifying
	\cite[1.2\,(v)]{snulmenn.isoperimetric} shows that $g$ may fail to be
	locally Lipschitzian.
\end{remark}
\begin{remark} \label{remark:eq_classes}
	If $f \in \trunc (V,Y)$ and $f(x) = 0$ for $\| V \|$ almost all $x$,
	then $f(x) = 0$ for $\| \delta V \|$ almost all $x$ by
	\cite[8.33]{snulmenn:tv.v2}, hence $\derivative Vf(x) = 0$ for $\| V
	\|$ almost all $x$; in particular, this applies to $f \in \SWloc q
	(V,Y)$ by \ref{remark:strict_local_sobolev_space_inclusion}.
\end{remark}
\begin{remark} \label{remark:strict_local_sobolev_space}
    The following four basic statements hold.
    \begin{enumerate}
	\item \label{item:strict_local_sobolev_space:product} If $1 \leq r \leq
	\infty$, $1 \leq s \leq \infty$, $1/r+1/s=1/q$, $f \in \SWloc{r} ( V,
	Y)$, and $g \in \SWloc{s} (V) $, then $gf \in \SWloc{q} (V,Y)$ and
	\begin{equation*}
	    \derivative{V}{(gf)} (x) = \derivative{V}{g}(x)\,f(x) + g(x)
	    \derivative{V}{f} (x) \quad \text{for $\| V \|$ almost all $x$}.
	\end{equation*}
	\item \label{item:strict_local_sobolev_space:composition} If $f \in
	\SWloc{q} ( V, Y )$, $Z$ is a finite dimensional normed vectorspace,
	and $g : Y \to Z$ is of class $\class{1}$ with $\Lip g < \infty$, then
	$g \circ f \in \SWloc{q} ( V, Z )$ and
	\begin{equation*}
	    \derivative{V}{( g \circ f )} (x) = \Der  g ( f(x) ) \circ
	    \derivative{V}{f} (x) \quad \text{for $\| V \|$ almost all $x$}.
	\end{equation*}
	\item \label{item:strict_local_sobolev_space:truncation} If $f \in
	\SWloc{q} (V)$ and $q < \infty$, then $\{ f^+, f^-, |f| \} \subset
	\SWloc{q} (V)$.
	\item \label{item:strict_local_sobolev_space:addition} If $f \in
	\SWloc 1 (V,Y)$, $g \in \trunc (V,Y)$, and $\derivative Vg \in \Lploc
	1 ( \| V \| , \Hom ( \rel^\adim, Y ) )$, then $f+g \in \trunc (V,Y)$
	and
	\begin{equation*}
		\derivative V{(f+g)} (x) = \derivative Vf(x) + \derivative
		Vg(x) \quad \text{for $\| V \|$ almost all $x$}.
	\end{equation*}
    \end{enumerate}
    \eqref{item:strict_local_sobolev_space:product} and
    \eqref{item:strict_local_sobolev_space:composition} are direct
    consequences of the definition and \cite[8.7]{snulmenn:tv.v2}.
    \eqref{item:strict_local_sobolev_space:truncation} follows from
    \eqref{item:strict_local_sobolev_space:composition} and the approximation
    technique employed in \cite[8.13\,(4)]{snulmenn:tv.v2}.
    \eqref{item:strict_local_sobolev_space:addition} follows from
    \cite[8.20\,(3)]{snulmenn:tv.v2} in conjunction with
    \ref{lemma:weak_closedness_tv}.
\end{remark}
\begin{comment}
	Now, the locally convex topologies on the local Sobolev spaces can
	be defined without referring to approximating functions.
\end{comment}
\begin{definition} \label{def:topology-SWSob}
	Suppose $\vdim, \adim \in \nat$, $\vdim \leq \adim$, $U$ is an open
	subset of $\rel^\adim$, $Y$ is a finite dimensional normed
	vectorspace, $V \in \RVar_\vdim (U)$, $\| \delta V \|$ is a Radon
	measure, and $1 \leq q \leq \infty$.

	Then $\SWloc q ( V,Y )$ is endowed with the topology induced by the
	family of seminorms mapping $f \in \SWloc q ( V,Y )$ onto
	\begin{equation*}
		\eqLpnorm{( \| V \| + \| \delta V \|) \restrict K} q f +
		\eqLpnorm{\| V \| \restrict K} q{\derivative Vf}
	\end{equation*}
	corresponding to all compact subsets $K$ of $U$, see
	\ref{miniremark:lcs-induced-by-seminorms}.
\end{definition}
\begin{remark} \label{remark:completeness_wqloc}
	Clearly, $\SWloc q (V,Y)$ is a locally convex space and whenever
	$K(i)$ is a sequence of compact subsets of $U$ with $K(i) \subset \Int
	K(i+1)$ for $i \in \nat$ and $U = \bigcup_{i=1}^\infty K(i)$ the
	topology of $\SWloc q ( V,Y)$ is induced by the seminorms
	corresponding to $K(i)$ for $i \in \nat$, see
	\ref{miniremark:lcs-induced-by-seminorms}, hence
	the topology of $\SWloc q ( V,Y)$ is induced by a real valued
	translation invariant pseudometric. Moreover, $\SWloc q (V,Y)$ is
	complete since (see \ref{remark:strict_local_sobolev_space_inclusion}
	and
	\ref{remark:sobolev-simpler-def}\,\eqref{item:sobolev-simpler-def:cpt})
	\begin{align*}
		& \{ (f,F) \with \text{$f \in \SWloc q (V,Y)$ and $F(x) =
		\derivative Vf(x)$ for $\| V \|$ almost all $x$} \} \\
		& \quad = \Clos \big \{ (g,\derivative Vg) \with \text{$g \in
		Y^U$, $\Lip g < \infty$, and $\spt g$ is compact} \big \},
	\end{align*}
	where the closure is taken in $\Lploc q ( \| V \| + \| \delta V \|, Y
	) \times \Lploc q ( \| V \|, \Hom ( \rel^\adim, Y ) )$, is complete by
	\ref{remark:lploc-complete} and \cite[6.22, 6.25]{MR0370454}.
\end{remark}
\begin{remark} \label{remark:basic-top-swloc}
	The following three basic statements will be verified.
	\begin{enumerate}
		\item \label{item:basic-top-swloc:lip} The subspace $Y^U
		\cap \{ g \with \text{$\Lip g < \infty$, $\spt g$ is compact}
		\}$ is dense in $\SWloc q (V,Y)$.
		\item \label{item:basic-top-swloc:q<oo_dense} If $q <
		\infty$, then $\mathscr{D} (U,Y)$ is dense in $\SWloc q
		(V,Y)$ and $\mathscr{D} (U, \rel) \cap \{ g \with g \geq 0 \}$
		is dense in $\SWloc q (V) \cap \{ f \with f \geq 0 \}$.
		\item \label{item:basic-top-swloc:q<oo_separable} If $q <
		\infty$, then $\SWloc q ( V,Y )$ is separable.
	\end{enumerate}
	\eqref{item:basic-top-swloc:lip} is a consequence of
	\ref{remark:sobolev-simpler-def}\,\eqref{item:sobolev-simpler-def:cpt}.
	\ref{remark:sobolev-simpler-def}\,\eqref{item:sobolev-simpler-def:q<oo}
	implies \eqref{item:basic-top-swloc:q<oo_dense}.
	\ref{miniremark:lcs-induced-by-seminorms},
	\ref{remark:lploc-dense-sep}, and \ref{remark:completeness_wqloc}
	yield \eqref{item:basic-top-swloc:q<oo_separable}.
\end{remark}
\begin{remark}
	If $q = \infty$ then the topology of $\SWloc q (V,Y)$ is induced by
	the family of seminorms mapping $f \in \SWloc q (V,Y)$ onto
	\begin{equation*}
		\eqLpnorm{\| V \| \restrict K} qf + \eqLpnorm{\| V \|
		\restrict K} q {\derivative Vf}
	\end{equation*}
	corresponding to all compact subsets $K$ of $U$ by
	\ref{remark:continuity_w_infty}.
\end{remark}
\begin{comment}
	Next, in order to conveniently formulate the quotient local Sobolev
	space and to prepare for the definition of the Sobolev space, the
	following quantity is defined.
\end{comment}
\begin{definition} \label{def:sobolev-seminorm}
	Suppose $\vdim, \adim \in \nat$, $\vdim \leq \adim$, $U$ is an open
	subset of $\rel^\adim$, $Y$ is a finite dimensional normed
	vectorspace, $V \in \RVar_\vdim (U)$, $\| \delta V \|$ is a Radon
	measure, and $1 \leq q \leq \infty$.

	Then define
	\begin{equation*}
		\SWnorm{q}{V}{f} = \eqLpnorm{\| V \|+\|
		\delta V \|}{q}{f} + \Lpnorm{\| V \|}{q}{ \derivative{V}{f} }
		\quad \text{for $f \in \trunc (V,Y)$}.
	\end{equation*}
\end{definition}
\begin{remark}
	The function $\SWnorm qV\cdot | E$ is a seminorm whenever $E$
	is a vectorspace contained in $\trunc (V,Y)$. However, the function
	$\SWnorm qV\cdot$ may not be a seminorm as its domain may fail to be a
	vectorspace, see \cite[8.25]{snulmenn:tv.v2}.
\end{remark}
\begin{remark} \label{remark:swloc-quotient}
	The quotient locally convex space
	\begin{equation*}
		Q = \SWloc q(V,Y) \Big / \big ( \SWloc q(V,Y) \cap \{ f \with
		\SWnorm qVf = 0 \} \big ),
	\end{equation*}
	see \ref{miniremark:lcs-quotient}, is Hausdorff and complete (by
	\ref{remark:completeness_wqloc}) and the topology of $Q$ is induced by
	a translation invariant metric by
	\ref{miniremark:lcs-induced-by-seminorms}. In particular, $Q$ is an
	``$F$-space'' in the terminology of
	\cite[\printRoman{2}.1.10]{MR0117523}.
\end{remark}
\begin{comment}
	The definition of Sobolev space is now obvious.
\end{comment}
\begin{definition} \label{def:strict_sobolev_space}
	Suppose $\vdim, \adim \in \nat$, $\vdim \leq \adim$, $U$ is an open
	subset of $\rel^\adim$, $Y$ is a finite dimensional normed
	vectorspace, $V \in \RVar_\vdim (U)$, $\| \delta V \|$ is a Radon
	measure, and $1 \leq q \leq \infty$.

	Then define the \emph{Sobolev space with respect to $V$ and
	exponent $q$} by
	\begin{equation*}
		\SWSob{q} (V,Y ) = \classification{\SWloc{q} (V,Y)}{f}{
		\SWnorm{q}{V}{f} < \infty}.
	\end{equation*}
	Abbreviate $\SWSob{q} (V,\rel) = \SWSob{q} (V)$.
\end{definition}
\begin{remark} \label{remark:strict_sobolev_space}
	Notice that \emph{$\SWSob{q} ( V, Y )$ is a
	$\SWnorm{q}{V}{\cdot}|\SWSob q (V,Y)$ complete topological vector
	space} by \ref{remark:completeness_wqloc}; in particular the set
	\begin{equation*}
		\{ (f,F) \with \text{$f \in \SWSob q(V,Y)$ and $F(x) =
		\derivative Vf (x)$ for $\| V \|$ almost all $x$} \}
	\end{equation*}
	is closed in $\Lp q ( \| V \| + \| \delta V \|, Y ) \times \Lp q ( \|
	V \|, \Hom ( \rel^\adim, Y ) )$.
\end{remark}
\begin{remark} \label{remark:density_sobolev_space}
    \emph{The vector subspaces
    \begin{gather*}
        \text{$\SWSob{q} ( V, Y ) \cap \mathscr{E} (U, Y)$ if
        $q < \infty$}, \\
        \text{and $\classification{\SWSob{\infty} (V,Y) }{g}{\text{$g$ is
        locally Lipschitzian} }$ if $q = \infty$}
    \end{gather*}
    are $\SWnorm{q}{V}{\cdot}|\SWSob q (V,Y)$ dense in $\SWSob{q}
    (V,Y)$}; in fact, suppose $\varepsilon > 0$ and $f \in \SWSob{q} ( V, Y
    )$, choose a sequence $\zeta_i$ forming a partition of unity on $U$
    associated with $\{ U \}$ as in \cite[3.1.13]{MR41:1976}, abbreviate
    $\kappa_i = \sup \im | \Der  \zeta_i |$ and $K_i = \spt \zeta_i$, select
    $g_i \in \mathscr{E} (U,Y)$ if $q < \infty$ by
    \ref{remark:sobolev-simpler-def}\,\eqref{item:sobolev-simpler-def:q<oo},
    respectively locally Lipschitzian functions $g_i : U \to Y$ if $q =
    \infty$, with
    \begin{equation*}
	( 1 + \kappa_i ) \eqLpnorm{( \| V \| + \| \delta V \|) \restrict K_i
	}{q}{f-g_i} + \eqLpnorm{\| V \| \restrict K_i
	}{q}{\derivative{V}{(f-g_i)}} \leq 2^{-i} \varepsilon
    \end{equation*}
    for $i \in \nat$, and define $g = \sum_{i=1}^{\infty} \zeta_i g_i$, hence
    one verifies $\SWnorm{q}{V}{f-g} \leq \varepsilon$ by means of
    \ref{remark:strict_local_sobolev_space}\,\eqref{item:strict_local_sobolev_space:product}.
    Observe that \emph{if $q < \infty$ then $\mathscr{E} ( U, Y )$ may be
    replaced by $\classification{ \mathscr{E} ( U, Y ) }{g}{ \text{$\spt g$ is
    bounded} }$ in the preceding statement}. Finally, notice that in case $Y =
    \rel$ similar results for the corresponding cones of nonnegative functions
    may be formulated.
\end{remark}
\begin{comment}
	The next remark employs the fact that closed subspaces of reflexive
	Banach spaces are reflexive, see \cite[\printRoman 2.3.23]{MR0117523}.
\end{comment}
\begin{remark} \label{remark:quotient-sobolev-space}
	In view of \ref{miniremark:lcs-quotient}, the quotient space
	\begin{equation*}
		Q = \SWSob q (V,Y ) \Big / \big ( \SWSob q (V,Y) \cap \{ f
		\with \SWnorm qVf = 0 \} \big )
	\end{equation*}
	is a Banach space normed by $\SWnorm qV\cdot \circ \pi^{-1}$, where $
	\pi : \SWSob q (V,Y) \to Q$ denotes the canonical projection. If $1 <
	q < \infty$, then $Q$ is reflexive by
	\ref{example:quotient-lp-spaces}, \ref{remark:strict_sobolev_space},
	and \cite[\printRoman 2.3.23]{MR0117523}.
\end{remark}
\begin{comment}
	Also, the definition of the subspace of Sobolev functions with ``zero
	boundary values'' now follows the usual pattern.
\end{comment}
\begin{definition} \label{def:strict-sobolev-space-zero-boundary-values}
    Suppose $\vdim, \adim \in \nat$, $\vdim \leq \adim$, $U$ is an open subset
    of $\rel^\adim$, $Y$ is a finite dimensional normed vectorspace, $V \in
    \RVar_\vdim (U)$, $\| \delta V \|$ is a Radon measure, and $1 \leq q \leq
    \infty$.

    Then define $\SWzero{q} (V,Y )$ to be the $\SWnorm qV\cdot | \SWSob
    q(V,Y)$ closure of
    \begin{equation*}
    	Y^U \cap \{ g \with \text{$\Lip g< \infty$, $\spt g$ is compact} \}
    \end{equation*}
    in $\SWSob q(V,Y)$. Abbreviate $\SWzero{q} (V,\rel) = \SWzero{q} (V)$.
\end{definition}
\begin{remark} \label{remark:zero_sobolev_space}
	Notice that $\SWzero q(V,Y)$ is $\SWnorm qV\cdot | \SWzero q(V,Y)$
	complete by \ref{remark:strict_sobolev_space}; in particular the set
	\begin{equation*}
		\{ (f,F) \with \text{$f \in \SWzero q (V,Y)$ and $F(x) =
		\derivative Vf(x)$ for $\| V \|$ almost all $x$} \}
	\end{equation*}
	is closed in $\Lp q ( \| V \| + \| \delta V \|, Y ) \times \Lp q ( \|
	V \|, \Hom ( \rel^\adim, Y ) )$.
\end{remark}
\begin{remark}
	If $K$ is a compact subset of $U$, $f \in \SWSob q (V,Y)$, and $f(x) =
	0$ for $\| V \| + \| \delta V \|$ almost all $x \in U \without K$,
	then $f \in \SWzero q (V,Y)$.
\end{remark}
\begin{remark} \label{remark:basic-zero-sobolev}
	Similarly to \ref{remark:strict_local_sobolev_space}, one obtains the
	following three basic properties.
	\begin{enumerate}
		\item \label{item:zero_strict_sobolev_space:product} If $1
		\leq r \leq \infty$, $1 \leq s \leq \infty$, $1/r+1/s=1/q$, $f
		\in \SWSob{r} ( V, Y)$, and $g \in \SWzero{s} (V) $, then $gf
		\in \SWzero{q} (V,Y)$.
		\item \label{item:zero_strict_sobolev_space:composition} If $f
		\in \SWzero{q} ( V, Y )$, $Z$ is a finite dimensional normed
		vectorspace, and $g : Y \to Z$ is of class $\class{1}$ with
		$\Lip g < \infty$ and $g(0)=0$, then $g \circ f \in \SWzero{q}
		( V, Z )$.
		\item \label{item:zero_strict_sobolev_space:truncation} If $f
		\in \SWzero{q} (V)$ and $q < \infty$, then $\{ f^+, f^-, |f|
		\} \subset \SWzero{q} (V)$.
	\end{enumerate}
\end{remark}
\begin{remark} \label{remark:density-in-SWzero}
	If $q < \infty$, then $\mathscr{D} (U,Y)$ is $\SWnorm qV\cdot| \SWzero
	q(V,Y)$ dense in $\SWzero q (V,Y)$ and $\mathscr{D} (U,\rel) \cap \{ g
	\with g \geq 0 \}$ is $\SWnorm qV\cdot| \SWzero q(V,\rel)$ dense in
	$\SWzero q (V) \cap \{ f \with f \geq 0 \}$ by
	\ref{corollary:approximation_lip}, \ref{remark:approximation_lip}, and
	\cite[8.7]{snulmenn:tv.v2}.
\end{remark}
\begin{remark} \label{remark:zero_sobolev_euclid_space}
	If $U = \rel^\adim$ and $q < \infty$, then $\SWzero q (V,Y) = \SWSob q
	(V,Y)$ by \ref{remark:density_sobolev_space}.
\end{remark}
\begin{remark} \label{remark:zero_extension}
	If $f : \spt \| V \| \to Y$ is continuous and $f \in \SWzero \infty
	(V,Y)$, then
	\begin{equation*}
		\text{$\{ x \with |f(x)| \geq t \}$ is compact whenever $0 < t
		< \infty$};
	\end{equation*}
	in fact, this is trivial if $f$ has compact support and the asserted
	condition is closed under uniform convergence.
\end{remark}
\begin{remark} \label{remark:zero_finite_measure}
	If $f \in \SWzero q (V,Y)$, then
	\begin{equation*}
		( \| V \| + \| \delta V \| ) ( \{ x \with |f(x)| \geq t \} ) <
		\infty \quad \text{whenever $0 < t < \infty$}
	\end{equation*}
	by \ref{remark:continuity_w_infty} and \ref{remark:zero_extension} if
	$q = \infty$ and trivially else.
\end{remark}
\begin{remark} \label{remark:quotient-zero-sobolev-space}
	In view of \ref{miniremark:lcs-quotient}, the quotient space
	\begin{equation*}
		Q = \SWzero q (V,Y ) \Big / \big ( \SWzero q (V,Y) \cap \{ f
		\with \SWnorm qVf = 0 \} \big )
	\end{equation*}
	is a Banach space normed by $\SWnorm qV\cdot \circ \pi^{-1}$, where $
	\pi : \SWzero q (V,Y) \to Q$ denotes the canonical projection. If $1 <
	q < \infty$, then $Q$ is reflexive by
	\ref{example:quotient-lp-spaces}, \ref{remark:zero_sobolev_space},
	and \cite[\printRoman 2.3.23]{MR0117523}.
\end{remark}
\begin{comment}
	Next, the link between the two realisations, for Sobolev functions and
	generalised weakly differentiable functions, of the concept of ``zero
	boundary values'' will be established. The proof uses an
	approximation procedure and relies on basic properties of generalised
	weakly differentiable functions and the corresponding concept of
	``zero boundary values''.
\end{comment}
\begin{theorem} \label{thm:zero_implies_zero}
	Suppose $\vdim, \adim \in \nat$, $\vdim \leq \adim$, $U$ is an open
	subset of $\rel^\adim$, $V \in \RVar_\vdim ( U )$, $\| \delta V \|$ is
	a Radon measure, $1 \leq q \leq \infty$, $Y$ is a finite dimensional
	normed vectorspace, and $f \in \SWzero q (V,Y)$.

	Then $|f| \in \trunc_{\Bdry U} (V)$.
\end{theorem}
\begin{proof}
	Firstly, it will be proven that \emph{if $g : Y \to \rel$ is a
	nonnegative, proper Lipschitzian function of class $1$ and $g(0)=0$,
	then $g \circ f \in \trunc_{\Bdry U} (V)$ and
	\begin{equation*}
		\derivative V{(g \circ f)} (x) = \Der g(f(x)) \circ
		\derivative Vf (x) \quad \text{for $\| V \|$ almost all $x$}.
	\end{equation*}}
	By \cite[8.12]{snulmenn:tv.v2}, $g \circ f \in \trunc (V)$ and the
	asserted formula holds. Observe that
	\begin{equation*}
		( \| V \| + \| \delta V \| ) ( \{ x \with g(f(x)) \geq z \} )
		< \infty \quad \text{for $0 < z < \infty$}
	\end{equation*}
	by \ref{remark:zero_finite_measure}. Choose Lipschitzian functions
	$f_i : U \to Y$ with compact support and
	\begin{equation*}
		\SWnorm q V {f-f_i} \to 0 \quad \text{as $i \to \infty$}.
	\end{equation*}
	Noting $g \circ f_i \in \trunc_{\Bdry U} (V)$ by
	\cite[9.2, 9.4]{snulmenn:tv.v2} and, if $q = 1$, then
	\begin{equation*}
		\tint{}{} | \derivative V{(g \circ f)} | \ud \| V \|
		< \infty, \quad \lim_{i \to \infty} \Lpnorm{\| V \|} 1
		{\derivative{V}{(g \circ f)} - \derivative V{(g \circ f_i)}} =
		0.
	\end{equation*}
	Now, applying \cite[9.13, 9.14]{snulmenn:tv.v2} with $f$ and $f_i$
	replaced by $g \circ f$ and $g \circ f_i$ yields the assertion.

	Secondly, there exist functions $g_i : Y \to \rel$ of class $1$ with
	\begin{equation*}
		g_i \geq 0, \quad \Lip g_i \leq 1, \quad \delta_i = \sup \big
		\{ \big |g_i(y)-|y| \big | \with y \in Y \big \} < \infty
	\end{equation*}
	for $i \in \nat$ and $\delta_i \to 0$ as $i \to \infty$. In
	particular, the maps $g_i$ are proper. Therefore one may require
	additionally $g_i(0)=0$ for $i \in \nat$. Notice that $|f| \in \trunc
	(V)$ by \ref{remark:strict_local_sobolev_space_inclusion} and
	\cite[8.16]{snulmenn:tv.v2} and
	\begin{equation*}
		\tint{A \cap \{ x \with g_i(f(x)) \geq
		z \}}{} | \derivative V{(g_i \circ f)} | \ud \| V \|
		\leq \tint{A \cap \{ x \with |f(x)| \geq c \}}{} \|
		\derivative Vf \| \ud \| V \| < \infty
	\end{equation*}
	whenever $0 < c < z < \infty$, $\delta_i \leq z-c$, and $A$ is $\| V
	\|$ measurable by \ref{remark:zero_finite_measure} and H{\"o}lder's
	inequality. In view \ref{remark:zero_finite_measure} and
	\cite[2.4.11]{MR41:1976}, the assertion of the preceding paragraph
	allows to apply \cite[9.13]{snulmenn:tv.v2} with $G$, $f$, and $f_i$
	replaced by $\Bdry U$, $|f|$, and $g_i \circ f$ to obtain the
	conclusion.
\end{proof}
\begin{comment}
	The Sobolev inequality now follows immediately.
\end{comment}
\begin{corollary} \label{corollary:sob_poin_summary}
	Suppose $\vdim$, $\adim$, $U$, and $V$ are as in
	\ref{miniremark:minimum_conditions}, $1 \leq q \leq \infty$, $Y$ is a
	finite dimensional normed vectorspace, $f \in \SWzero q ( V,Y )$, and
	\begin{equation*}
		\beta = \infty \quad \text{if $\vdim = 1$}, \qquad \beta =
		\vdim/(\vdim-1) \quad \text{if $\vdim > 1$}.
	\end{equation*}

	Then there holds
	\begin{equation*}
		\Lpnorm{\| V \|}{\beta}{f} \leq
		\Gamma_{\textup{\cite[10.1]{snulmenn:tv.v2}}} ( \adim ) \big (
		\Lpnorm{\| V \|}{1}{ \derivative{V}{f} } + \| \delta V \|(f)
		\big ).
	\end{equation*}
\end{corollary}
\begin{proof}
	In view of \ref{remark:zero_finite_measure} and
	\ref{thm:zero_implies_zero}, the conclusion is a consequence of
	\cite[8.16, 10.1\,(2a)]{snulmenn:tv.v2}.
\end{proof}
\section{Geodesic distance}
In this section and in the following section, varifolds satisfying a
dimensionally critical summability condition on the mean curvature and a
lower bound on their densities as described in \ref{miniremark:situation} with
$p = \vdim$ will be investigated. Here, the properties of the geodesic
distance in the support of the weight measure of such varifolds are
studied.  Since connected components of this support are relatively open
by \cite[6.14]{snulmenn:tv.v2}, one may assume for this purpose that the
support of the weight measure is connected, see \ref{remark:connectedness}.
Moreover, it is known from \cite[14.2]{snulmenn:tv.v2} that in this case the
geodesic distance between any two points is finite.

In the present section, it is established that -- under the previously
described hypotheses -- the geodesic distance is a continuous function and
gives rise to a local Sobolev function with bounded generalised weak
derivative, see \ref{thm:intrinsic_metric}. Moreover, an example is
constructed that shows that this function need not be locally H{\"o}lder
continuous with respect to any exponent, see \ref{example:geodesic_distance}.
This example also yields that the embedding result of local Sobolev functions
into continuous functions which will be obtained in
\ref{thm:sob_continuous_emb} is sharp, see \ref{remark:sob_continuous_emb}.

The proof of the properties of the geodesic distance consists of a refinement
of the techniques used in \cite[14.2]{snulmenn:tv.v2}. In particular, it
relies as well on the oscillation estimates for continuous generalised weakly
differentiable functions obtained in \cite[13.1]{snulmenn:tv.v2}.
\begin{comment}
	Firstly, the condition on the varifolds will be formulated in which,
	usually, the case $p = \vdim$ will be considered.
\end{comment}
\begin{miniremark} \label{miniremark:situation}
	Suppose $\vdim, \adim \in \nat$, $\vdim \leq \adim$, $1 \leq p \leq
	\infty$, $U$ is an open subset of $\rel^\adim$, $V \in \Var_\vdim
	(U)$, $\| \delta V \|$ is a Radon measure, $\density^\vdim ( \| V \|,
	x ) \geq 1$ for $\| V \|$ almost all $x$. If $p > 1$, then suppose
	additionally that $\mathbf{h} ( V, \cdot ) \in \Lploc{p} ( \| V \|,
	\rel^\adim )$ and
	\begin{equation*}
		(\delta V)(\theta) = - \tint{}{} \mathbf{h} (V,x) \bullet
		\theta (x) \ud \| V \| \, x \quad \text{for $\theta \in
		\mathscr{D} (U,\rel^\adim)$}.
	\end{equation*}
	In particular, $V$ is rectifiable by Allard
	\cite[5.5\,(1)]{MR0307015}. If $p = 1$ let $\psi = \| \delta V \|$. If
	$1 < p < \infty$ define a Radon measure $\psi$ over $U$ by
	\begin{equation*}
		\psi (A) = \tint A\ast | \mathbf{h} (V,x) |^p \ud \| V \| \, x
		\quad \text{for $A \subset U$}.
	\end{equation*}
\end{miniremark}
\begin{comment}
	Secondly, an observation concerning the differential of a function
	relative to a set will be made.
\end{comment}
\begin{miniremark} \label{miniremark:relative_differential}
	Suppose $X$ and $Y$ are a normed vectorspaces, $a \in A \subset X$,
	and $f : A \to Y$ is differentiable relative to $A$ at $a$.%
	\begin{footnote}
		{Suppose $X$ and $Y$ are normed vectorspaces, $A \subset X$,
		$a \in \Clos A$, and $f : A \to Y$. Then the \emph{tangent
		cone} of $A$ at $a$, denoted $\Tan(A,a)$, is the set of all $u
		\in X$ such that for every $\varepsilon > 0$ there exist $x
		\in A$ and $0 < r \in \rel$ with
		\begin{equation*}
			|x-a| < \varepsilon \quad \text{and} \quad | r(x-a)-u|
			< \varepsilon,
		\end{equation*}
		see \cite[3.1.21]{MR41:1976}. Moreover, $f$ is called
		\emph{differentiable relative to $A$ at $a$} if and only if
		there exist a neighbourhood $U$ of $a$ in $X$ and function $g
		: U \to Y$ such that
		\begin{equation*}
			g | A \cap U = f | A \cap U, \quad \text{$g$ is
			differentiable at $a$}.
		\end{equation*}
		In this case $\Der g(a)|\Tan(A,a)$ is determined by $f$ and
		$a$ and denoted $\Der f(a)$, see \cite[3.1.22]{MR41:1976}.}
	\end{footnote}
	Then
	\begin{equation*}
		\sup \{ | \Der f(a)(u) | \with \text{$u \in \Tan (A,a)$ and
		$|u|=1$} \} \leq \limsup_{x \to a} \frac{|f(x)-f(a)|}{|x-a|}
	\end{equation*}
	and equality holds if $\dim X < \infty$; in fact, by \cite[3.1.21,
	3.1.22]{MR41:1976} it is sufficient to note that if $x_i$ is a
	sequence in $A \without \{ a \}$, $u \in X$, $x_i \to a$ as $i \to
	\infty$, and $|x_i-a|^{-1}(x_i-a) \to u$ as $i \to \infty$ then
	\begin{equation*}
		|\Der f(a)(u)| = \lim_{i \to \infty}
		\frac{|f(x_i)-f(a)|}{|x_i-a|}.
	\end{equation*}
\end{miniremark}
\begin{comment}
	Notice that if $U$ is a proper subset of $\rel^\adim$ in
	\ref{miniremark:situation} then $\spt \| V \|$, which by definition is
	a subset of $U$, may be incomplete. This makes the study of the
	geodesic distance on $\spt \| V \|$ more delicate. Therefore,
	initially, the geodesic distance on the closure of $\spt \| V \|$ in
	$\rel^\adim$ will investigated before treating the general case by
	means of an exhaustion procedure.

	Some well known facts concerning geodesic distances are summarised
	below.
\end{comment}
\begin{miniremark} \label{miniremark:special_intrinsic_metric}
	Suppose $Y$ is a boundedly compact metric space metrised by $\tau$ and
	$X$ is a dense subset of $Y$. Whenever $0 < \delta \leq \infty$ one
	may define a pseudometric $\sigma_\delta : X \times X \to \overline
	\rel$ by letting $\sigma_\delta (a,x)$ denote the infimum of the set
	of numbers
	\begin{equation*}
		\sum_{i=1}^{j} \tau ( x_i,x_{i-1} )
	\end{equation*}
	corresponding to all finite sequences $x_0, x_1, \ldots, x_j$ in $X$
	with $x_0 = a$, $x_j = x$ and $\tau(x_i,x_{i-1}) \leq \delta$ for $i =
	1, \ldots, j$ and $j \in \nat$. Clearly, $\sigma_\infty = \tau | X
	\times X$ and $\sigma_\delta \geq \sigma_\varepsilon$ whenever $0 <
	\delta \leq \varepsilon$. Defining $\sigma : X \times X \to \overline
	\rel$ by
	\begin{equation*}
		\sigma (a,x) = \lim_{\delta \to 0+} \sigma_\delta (a,x) \quad
		\text{for $a,x \in X$},
	\end{equation*}
	one obtains a pseudometric over $X$ such that $\sigma (a,x)$ equals
	the infimum of the set of numbers $\mathbf{V}_{\inf I}^{\sup I} g$
	corresponding to continuous maps $g : \rel \to Y$ such that $g( \inf I
	) = a$ and $g ( \sup I ) = x$ for some nonempty compact subinterval
	$I$ of $\rel$, where the length of $g$ from $\inf I$ to $\sup I$ is
	computed with respect to $\tau$; in fact, if $\sigma (a,x) < \infty$,
	then there exists $g$ mapping $\rel$ into $Y$ satisfying
	\begin{equation*}
		g ( 0 ) = a, \quad g ( \sigma(a,x) ) = x, \quad \Lip g \leq 1.
	\end{equation*}
	These classical facts may be verified by means of \cite[2.5.16,
	2.10.21]{MR41:1976}.
\end{miniremark}
\begin{comment}
	For the auxiliary pseudometric $\sigma$ the desired result now follows
	by approximating by the pseudometrics $\sigma_\delta$ and passing to
	the limit with the help of the oscillation estimate
	\cite[13.1]{snulmenn:tv.v2} and Ascoli's theorem.
\end{comment}
\begin{lemma} \label{lemma:special_intrinsic_metric}
    Suppose $\vdim$, $\adim$, $p$, $U$, and $V$ are as in
    \ref{miniremark:situation}, $p = \vdim$, $X = \spt \| V \|$, $X$ is
    connected, $\sigma$ is associated to $X$ as in
    \ref{miniremark:special_intrinsic_metric} with $Y = \Clos X$, and $W \in
    \Var_{2\vdim} ( U \times U )$ satisfies
    \begin{equation*}
	W(k) = \tint{}{} k ((x_1,x_2),P_1\times P_2) \ud ( V \times V )\,
	((x_1,P_1),(x_2,P_2))
    \end{equation*}
    whenever $k \in \mathscr{K} ( U \times U, \grass{ \rel^\adim \times
    \rel^\adim}{2 \vdim} )$.

    Then the following two statements hold.
    \begin{enumerate}
	\item The function $\sigma$ is continuous, a metric on $X$, and
	belongs to $\SWloc q (W)$ for $1 \leq q < \infty$ with
	\begin{equation*}
	    | \langle (u_1,u_2), \derivative W \sigma (x_1,x_2) \rangle |
	    \leq | u_1 | + | u_2 | \quad \text{whenever $u_1,u_2 \in
	    \rel^\adim$}
	\end{equation*}
	for $\| W \|$ almost all $(x_1,x_2)$.
	\item \label{item:distance-sob:point} If $a \in X$, then
	$\sigma(a,\cdot) \in \SWloc q (V)$ for $1 \leq q < \infty$ and
	\begin{equation*}
	    | \derivative V{(\sigma(a,\cdot))}(x)| = 1 \quad \text{for $\| V \|$
	    almost all $x$}.
	\end{equation*}
    \end{enumerate}
\end{lemma}
\begin{proof}
	Define a norm $\nu$ over $\rel^\adim \times \rel^\adim$ by $\nu
	(x_1,x_2) = | x_1 | + |x_2|$ for $x_1,x_2 \in \rel^\adim$. Quantities
	derived with $\nu$ replacing the norm associated to the inner product
	on $\rel^\adim \times \rel^\adim$ will be distinguished by the
	subscript $\nu$. Notice that $W$ satisfies the conditions of
	\ref{miniremark:situation} with $\vdim$, $\adim$, $p$, and $U$
	replaced by $2 \vdim$, $2 \adim$, $\vdim$, and $U \times U$ and
	\begin{equation*}
		\| \project P \|_\nu \leq 1 \quad \text{for $W$ almost all
		$(z,P)$}
	\end{equation*}
	by \cite[3.7\,(1)\,(2)\,(4)\,(5)\,(6)]{kol-men:decay.v2}.  Let
	$\sigma_\delta$ for $0 < \delta \leq \infty$ be defined as in
	\ref{miniremark:special_intrinsic_metric}.

	Notice that
	\begin{equation*}
		\sigma_\delta(a,x) \leq \sigma_\delta (\alpha,\chi) + \nu (
		(a,x) - ( \alpha, \chi ) )
	\end{equation*}
	whenever $a,x,\alpha,\chi \in X$ and $\sup \{ |a-\alpha|,|x-\chi| \}
	\leq \delta$. Since $X$ is connected and $\sigma_\delta (a,a) = 0$ for
	$a \in X$, it follows that $\sigma_\delta$ is a locally Lipschitzian
	real valued function with
	\begin{equation*}
		\Lip_\nu ( \sigma_\delta | A ) \leq 1 \quad \text{whenever $A
		\subset X \times X$ and $\diam A \leq \delta$},
	\end{equation*}
	in particular $\sigma_\delta \in \SWloc \infty (W)$ and $\sigma_\delta
	(a,\cdot) \in \SWloc \infty (V)$ with
	\begin{gather*}
		\| \derivative W\sigma_\delta (z) \|_\nu = \big \| ( \| W \|,
		2 \vdim ) \ap \Der \sigma_\delta (z) \big \|_\nu \leq 1 \quad
		\text{for $\| W\|$ almost all $z$}, \\
		| \derivative V{(\sigma_\delta(a,\cdot))} (x) | \leq 1 \quad
		\text{for $\| V \|$ almost all $x$}
	\end{gather*}
	for $a \in X$ by \cite[8.7]{snulmenn:tv.v2} in conjunction with
	\cite[3.2.16]{MR41:1976} and \ref{miniremark:relative_differential}.
	Since $\{ \sigma_\delta (a,\cdot) |K \with a \in X, \delta > 0 \}$ is
	an equicontinuous family of functions whenever $K$ is a compact subset
	of $X$ by \ref{remark:strict_local_sobolev_space_inclusion} and
	\cite[4.8\,(1), 13.1]{snulmenn:tv.v2} and $\sigma$ is real valued by
	\cite[14.2]{snulmenn:tv.v2}, one obtains that
	\begin{equation*}
		\sigma_\delta (a,\cdot) \uparrow \sigma(a,\cdot) \quad
		\text{locally uniformly as $\delta \to 0+$ for $a \in X$}.
	\end{equation*}
	by the Ascoli theorem, see \cite[7.14, 7.18]{MR0370454}. Therefore
	$\sigma(a,\cdot)$ is continuous for $a \in X$, hence $\sigma$ is
	continuous as $\sigma$ is a metric and
	\begin{equation*}
		\sigma_\delta \uparrow \sigma \quad \text{locally uniformly
		as $\delta \to 0+$}
	\end{equation*}
	by Dini's theorem, see \cite[Problem~7.E]{MR0370454}. Consequently,
	\ref{lemma:weak_closedness_tv} and \cite[8.14]{snulmenn:tv.v2} yield
	$\sigma \in \trunc (W)$ and $\sigma(a,\cdot) \in \trunc (V)$ with
	\begin{gather*}
		\| \derivative W \sigma (z) \|_\nu \leq 1 \quad \text{for $\|
		W \|$ almost all $z$}, \\
		| \derivative V {(\sigma(a,\cdot))} (x) | \leq 1 \quad
		\text{for $\| V \|$ almost all $x$}
	\end{gather*}
	for $a \in X$. From \ref{miniremark:prop_lp_spaces},
	\ref{lemma:weak_closedness_tv}, \cite[2.5.7\,(ii)]{MR41:1976}, and
	Alaoglu's theorem, see \cite[\printRoman{5}.4.2,
	\printRoman{5}.5.1]{MR0117523}, one obtains
	\begin{gather*}
		\tint{}{} \langle \theta, \derivative W{\sigma_\delta}
		\rangle \ud \| W \| \to \tint{}{} \langle \theta, \derivative
		W\sigma \rangle \ud \| W \| \quad \text{for $\theta \in \Lp 1
		( \| W \|, \rel^\adim \times \rel^\adim )$}, \\
		\tint{}{} \langle \theta, \derivative
		V{(\sigma_\delta(a,\cdot))} \rangle \ud \| V \| \to \tint{}{}
		\langle \theta, \derivative V{(\sigma(a,\cdot))} \rangle \ud
		\| V \| \quad \text{for $\theta \in \Lp 1 ( \| V \|,
		\rel^\adim )$}
	\end{gather*}
	as $\delta \to 0+$ for $a \in X$, so that Mazur's lemma, see
	\cite[\printRoman{5}.3.14]{MR0117523}, and
	\cite[2.5.7\,(i)]{MR41:1976} in fact yield $\sigma \in \SWloc q (W)$
	and $\sigma(a,\cdot) \in \SWloc q (V)$ for $a \in X$ and $1 \leq q <
	\infty$.

	Suppose $a \in X$. Then, by \cite[11.2, 11.4\,(4)]{snulmenn:tv.v2},
	$\| V \|$ almost all $x$ satisfy $x \in X$, $x \neq a$, and
	$\sigma(a,\cdot)$ is differentiable relative to $X$ at $x$ and
	$|\Der(\sigma(a,\cdot))(x)| = | \derivative V(\sigma(a,\cdot))(x)|$.
	Consider such $x$, abbreviate $b = \sigma (a,x)$, and choose $g$ as in
	\ref{miniremark:special_intrinsic_metric}. Then $\Upsilon = g^{-1}
	\lIm X \rIm$ is neighbourhood of $b$ and one observes that
	\begin{gather*}
		\sigma(a,g(\upsilon)) = \upsilon \quad \text{whenever
		$\upsilon \in \Upsilon$ and $0 \leq \upsilon \leq b$}, \\
		1 \leq \big ( \limsup_{\chi \to x} | \sigma
		(a,\chi)-\sigma(a,x)|/|\chi-x| \big ) \big ( \limsup_{\upsilon
		\to b-} |g(\upsilon)-g(b)|/|\upsilon-b| \big ).
	\end{gather*}
	Consequently, one infers $1 \leq | \Der (\sigma(a,\cdot))(x)|$ by
	\ref{miniremark:relative_differential}.
\end{proof}
\begin{comment}
	Next, the exhaustion procedure is prepared to treat the general case.
\end{comment}
\begin{miniremark} \label{miniremark:lcs_connected}
	Suppose $X$ is a connected, locally connected, locally compact,
	separable metric space. Then there exists a sequence of connected,
	open subsets $A_i$ of $X$ with compact closure and $\Clos A_i \subset
	A_{i+1}$ for $i \in \nat$ and $X = \bigcup_{i=1}^\infty A_i$; in fact,
	if $\Phi$ is a nonempty countable base of the topology of $X$
	consisting of nonempty connected open subsets of $X$ with compact
	closure, then one observes that there exists an enumeration $B_1, B_2,
	B_3, \ldots$ of $\Phi$ such that $B_{i+1} \cap \bigcup_{j=1}^i B_j
	\neq \varnothing$ for $i \in \nat$, hence one may inductively select a
	strictly increasing sequence of positive integers $j(i)$ such that
	\begin{equation*}
		A_i = \bigcup_{k=1}^{j(i)} B_k \quad \text{satisfies} \quad
		\Clos A_i \subset \bigcup_{k=1}^{j(i+1)} B_k \quad \text{for
		$i \in \nat$}.
	\end{equation*}
\end{miniremark}
\begin{miniremark} \label{miniremark:intrinsic_metric}
	Suppose $X$ is a metric space and $\varrho$ is the pseudometric on $X$
	defined by letting $\varrho (a,x)$ for $a,x \in X$ denote the infimum
	of the set of numbers
	\begin{equation*}
		\mathbf{V}_{\inf I}^{\sup I} g
	\end{equation*}
	corresponding to continuous maps $g : \rel \to X$ and nonempty compact
	subintervals $I$ of $\rel$ with $g ( \inf I ) = a$ and $g ( \sup I ) =
	x$. If $A_i$ form a sequence of open subsets of $X$ with compact
	closure and $\Clos A_i \subset A_{i+1}$ for $i \in \nat$ and $X =
	\bigcup_{i=1}^\infty A_i$, and $\varrho_i$ are the pseudometrics on
	$A_i$ such that $\varrho_i (a,x)$ for $a,x \in X$ equals the infimum
	of the set of numbers $\mathbf{V}_{\inf I}^{\sup I} g$ corresponding
	to continuous maps $g : \rel \to \Clos A_i$ and nonempty compact
	subintervals $I$ of $\rel$ such that $g ( \inf I ) = a$ and $g ( \sup
	I ) = x$, then $\varrho_i$ equals the metric constructed in
	\ref{miniremark:special_intrinsic_metric} under the name ``$\sigma$''
	with $X$ and $Y$ replaced by $A_i$ and $\Clos A_i$, $\varrho_{i+1}
	|A_i \times A_i \leq \varrho_i$ for $i \in \nat$ and
	\begin{equation*}
		\varrho_i (a,x) \to \varrho (a,x) \quad \text{as $i \to
		\infty$ for $a,x \in X$}.
	\end{equation*}
	Evidently, if $X$ is a dense subset of some boundedly compact metric
	space $Y$, then the pseudometric $\sigma$ constructed in
	\ref{miniremark:special_intrinsic_metric} satisfies $\sigma \leq
	\varrho$.
\end{miniremark}
\begin{comment}
	If $X$ is incomplete the infimum occurring in the definition of
	$\varrho (a,x)$ need not to be attained even if $\varrho (a,x) <
	\infty$. The following lemma serves as a substitute.
\end{comment}
\begin{lemma} \label{lemma:ex_mini_path}
	Suppose $Y$ is a boundedly compact metric space, $X \subset Y$,
	$\varrho$ is associated to $X$ as in
	\ref{miniremark:intrinsic_metric}, $\varrho$ is continuous, $a,x \in
	X$, and $b = \varrho(a,x) < \infty$.

	Then there exists a map $g : \rel \to Y$ satisfying
	\begin{gather*}
		g (0) = a, \quad g(b) = x, \quad \Lip g \leq 1, \\
		\varrho (a,g( \upsilon)) = \upsilon \quad \text{whenever $0
		\leq \upsilon \leq b$ and $g( \upsilon ) \in X$}.
	\end{gather*}
\end{lemma}
\begin{proof}
	For $i \in \nat$ choose continuous maps $g_i : \rel \to X$ and $0 \leq
	b_i < \infty$ such that $g_i(0) = a$, $g_i( b_i ) = x$, and
	$\mathbf{V}_0^{b_i} g_i \to b$ as $i \to \infty$. In view of
	\cite[2.5.16]{MR41:1976} one may require additionally $\Lip g_i \leq
	1$ and $b_i = \mathbf{V}_0^{b_i} g_i$, hence $\mathbf{V}_y^\upsilon
	g_i = \upsilon-y$ for $0 \leq y \leq \upsilon \leq b_i$. Possibly
	passing to a subsequence, one constructs a map $g : \rel \to Y$ as the
	locally uniform limit of $g_i$ as $i \to \infty$ with
	\begin{equation*}
		g(0) = a, \quad g( b ) = x, \quad \Lip g \leq 1,
	\end{equation*}
	see \cite[2.10.21]{MR41:1976}. If $0 \leq \upsilon \leq b$ and
	$g(\upsilon) \in X$, then
	\begin{gather*}
		\varrho (a,g(\upsilon)) = \lim_{i \to \infty} \varrho
		(a,g_i(\upsilon)) \leq \liminf_{i \to \infty}
		\mathbf{V}_0^\upsilon g_i = \upsilon, \\
		\varrho (g(\upsilon),x) = \lim_{i \to \infty} \varrho
		(g_i(\upsilon),g_i(b)) \leq \liminf_{i \to \infty}
		\mathbf{V}_\upsilon^b g_i = b-\upsilon,
	\end{gather*}
	hence $\varrho (a,g(\upsilon)) = \upsilon$.
\end{proof}
\begin{comment}
	Now, the general case may be treated using the same pattern of proof
	as in \ref{lemma:special_intrinsic_metric}.
\end{comment}
\begin{theorem} \label{thm:intrinsic_metric}
	Suppose $\vdim$, $\adim$, $p$, $U$, and $V$ are as in
	\ref{miniremark:situation}, $p = \vdim$, $X = \spt \| V \|$, $X$ is
	connected, $\varrho$ is associated to $X$ as in
	\ref{miniremark:intrinsic_metric}, and $W \in \Var_{2\vdim} ( U \times
	U )$ satisfies
   	\begin{equation*}
		W(k) = \tint{}{} k ((x_1,x_2),P_1\times P_2) \ud ( V \times V
		)\, ((x_1,P_1),(x_2,P_2))
	\end{equation*}
	whenever $k \in \mathscr{K} ( U \times U, \grass{ \rel^\adim \times
	\rel^\adim}{2 \vdim} )$.

	Then the following two statements hold.
	\begin{enumerate}
		\item \label{item:instrinsic_metric:product} The function
		$\varrho$ is continuous, a metric on $X$, and belongs to
		$\SWloc q (W)$ for $1 \leq q < \infty$ with
		\begin{equation*}
			| \langle (u_1,u_2), \derivative W \varrho (x_1,x_2)
			\rangle | \leq | u_1 | + | u_2 | \quad \text{whenever
			$u_1,u_2 \in \rel^\adim$}
		\end{equation*}
		for $\| W \|$ almost all $(x_1,x_2)$.
		\item \label{item:intrinsic_metric:point} If $a \in X$, then
		$\varrho(a,\cdot) \in \SWloc q (V)$ for $1 \leq q < \infty$
		and
		\begin{equation*}
			| \derivative V{(\varrho(a,\cdot))}(x)| = 1 \quad
			\text{for $\| V \|$ almost all $x$}.
		\end{equation*}
	\end{enumerate}
\end{theorem}
\begin{proof}
	Firstly, define a norm $\nu$ over $\rel^\adim \times \rel^\adim$ by
	$\nu (x_1,x_2) = | x_1 | + |x_2|$ for $x_1,x_2 \in \rel^\adim$.
	Quantities derived with $\nu$ replacing the norm associated to the
	inner product on $\rel^\adim \times \rel^\adim$ will be distinguished
	by the subscript $\nu$.

	Secondly, observe that \cite[6.14\,(3)]{snulmenn:tv.v2} implies that
	$X$ is locally connected, hence one may choose subsets $A_i$ of $X$ as
	in \ref{miniremark:lcs_connected} and apply
	\ref{miniremark:intrinsic_metric} to obtain $\varrho_i$.  Defining
	$U_i = U \without ( X \without A_i )$, one obtains an increasing
	sequence of open subsets $U_i$ of $U$ such that $A_i = U_i \cap X$ for
	$i \in \nat$ and $U = \bigcup_{i=1}^\infty U_i$. Let
	\begin{equation*}
		V_i = V | \mathbf{2}^{U_i \times \grass \adim \vdim}, \quad
		W_i = W | \mathbf{2}^{(U_i \times U_i) \times \grass 
		{\rel^\adim \times \rel^\adim}{2 \vdim }}
	\end{equation*}
	for $i \in \nat$. Applying
	\ref{lemma:special_intrinsic_metric} with $U$, $V$, $X$, and $W$
	replaced by $U_i$, $V_i$, $A_i$, and $W_i$ yields that $\varrho_i$ are
	continuous metrics on $A_i$ such that
	\begin{gather*}
		\varrho_i \in \SWloc q ( W_i ) \quad \text{with} \quad \|
		\derivative {W_i} \varrho_i (z) \|_\nu \leq 1 \quad \text{for
		$\| W_i \|$ almost all $z$}, \\
		\varrho_i (a,\cdot) \in \SWloc q ( V_i ) \quad \text{with}
		\quad | \derivative {V_i} {(\varrho_i(a,\cdot))} (x) | \leq 1
		\quad \text{for $\| V_i \|$ almost all $x$}
	\end{gather*}
	whenever $1 \leq q < \infty$, $i \in \nat$, and $a \in A_i$.

	Thirdly, suppose $j \in \nat$. Then $\{ \varrho_i (a,\cdot) | K \with
	a \in A_j, j \leq i \in \nat \}$ is an equicontinuous family of
	functions whenever $K$ is a compact subset of $A_j$ by
	\ref{remark:strict_local_sobolev_space_inclusion} and \cite[4.8\,(1),
	13.1]{snulmenn:tv.v2}, hence one obtains that
	\begin{equation*}
		\varrho_i(a,\cdot) | A_j \downarrow \varrho(a,\cdot) | A_j
		\quad \text{locally uniformly as $i \to \infty$ for $a \in
		A_j$}
	\end{equation*}
	by the Ascoli theorem, see \cite[7.14, 7.18]{MR0370454}. Therefore
	$\varrho(a,\cdot)|A_j$ is continuous for $a \in A_j$, hence $\varrho |
	A_j \times A_j$ is continuous as $\varrho | A_j \times A_j$ is a
	metric and
	\begin{equation*}
		\varrho_i |A_j \times A_j \downarrow \varrho|A_j \times A_j
		\quad \text{locally uniformly as $i \to \infty$}
	\end{equation*}
	by Dini's theorem, see \cite[Problem~7.E]{MR0370454}. Consequently,
	\ref{lemma:weak_closedness_tv},
	\ref{remark:strict_local_sobolev_space_inclusion}, and
	\cite[8.14]{snulmenn:tv.v2} yield $\varrho | A_j \times A_j \in \trunc
	(W_j)$ and $\varrho(a,\cdot) |A_j \in \trunc (V_j)$ with
	\begin{gather*}
		\| \derivative {W_j}{(\varrho|A_j \times A_j)} (z) \rangle
		\|_\nu \leq 1 \quad \text{for $\| W_j \|$ almost all $z$}, \\
		| \derivative {V_j} {(\varrho(a,\cdot)|A_j)} (x) | \leq 1 \quad
		\text{for $\| V_j \|$ almost all $x$}
	\end{gather*}
	for $a \in A_j$. From \ref{miniremark:prop_lp_spaces},
	\ref{lemma:weak_closedness_tv}, \cite[2.5.7\,(ii)]{MR41:1976} and
	Alaoglu's theorem, see \cite[\printRoman{5}.4.2,
	\printRoman{5}.5.1]{MR0117523}, one obtains
	\begin{equation*}
		\lim_{i \to \infty} \tint{}{} \langle \theta, \derivative
		{W_j}{(\varrho_i|A_j \times A_j)} \rangle \ud \| W_j \| =
		\tint{}{} \langle \theta, \derivative {W_j}{(\varrho | A_j
		\times A_j)} \rangle \ud \| W_j \|
	\end{equation*}
	for $\theta \in \Lp 1 ( \| W_j \|, \rel^\adim \times \rel^\adim )$ and
	\begin{equation*}
		\lim_{i \to \infty} \tint{}{} \langle \theta, \derivative
		{V_j}{(\varrho_i(a,\cdot)|A_j)} \rangle \ud \| V_j \| =
		\tint{}{} \langle \theta, \derivative
		{V_j}{(\varrho(a,\cdot)|A_j)} \rangle \ud \| V_j \| \quad
	\end{equation*}
	for $\theta \in \Lp 1 ( \| V_j \|, \rel^\adim )$, so that
	\ref{remark:completeness_wqloc}, Mazur's lemma, see
	\cite[\printRoman{5}.3.14]{MR0117523}, and
	\cite[2.5.7\,(i)]{MR41:1976} in fact yield $\varrho|A_j \times A_j \in
	\SWloc q (W_j)$ and $\varrho(a,\cdot) | A_j \in \SWloc q (V_j)$ for $a
	\in A_j$ and $1 \leq q < \infty$.

	Finally, suppose $a \in X$. Then, by \cite[11.2,
	11.4\,(4)]{snulmenn:tv.v2}, $\| V \|$ almost all $x$ satisfy $x \in X$,
	$x \neq a$, and $\varrho(a,\cdot)$ is differentiable relative to $X$
	at $x$ and $|\Der (\varrho(a,\cdot))(x)| = | \derivative
	V(\varrho(a,\cdot))(x)|$. Consider such $x$, abbreviate $b = \varrho
	(a,x)$, and choose $g$ as in \ref{lemma:ex_mini_path} with $Y = \Clos
	X$. Therefore one obtains, as $b \in \Int g^{-1} \lIm X \rIm$,
	\begin{equation*}
		1 \leq \big ( \limsup_{\chi \to x} | \varrho
		(a,\chi)-\varrho(a,x)|/|\chi-x| \big ) \big ( \limsup_{\upsilon
		\to b-} |g(\upsilon)-g(b)|/|\upsilon-b| \big ).
	\end{equation*}
	and infers $1 \leq | \Der (\varrho(a,\cdot))(x)|$ by
	\ref{miniremark:relative_differential}.
\end{proof}
\begin{remark} \label{remark:connectedness}
	The connectedness hypothesis on $X$ is not as restrictive as it may
	seem since the theorem may otherwise be applied separately to $V
	\restrict C \times \grass \adim \vdim$ whenever $C$ is a connected
	component of $\spt \| V \|$ by \cite[6.14]{snulmenn:tv.v2}.
\end{remark}
\begin{comment}
	Finally, an example concerning the possible behaviour of the geodesic
	distance to a point is constructed which uses the following
	observation.
\end{comment}
\begin{miniremark} \label{miniremark:balls_in_balls}
	The following fact is of elementary geometric nature. \emph{If
	$0 < \delta_i < r_i < \infty$ for $i \in \{1,2\}$, $2 \leq \adim \in
	\nat$, $u \in \mathbf{S}^{\adim-1}$, $\delta_2 < \delta_1$, and $r_2^2
	- (r_2-\delta_2)^2 < r_1^2 - (r_1-\delta_1)^2$, then $\cball
	  {(\delta_2-r_2)u}{r_2} \cap \{ x \with x \bullet u \geq 0 \} \subset
	  \oball {(\delta_1-r_1)u}{r_1}$.}
\end{miniremark}
\begin{example} \label{example:geodesic_distance}
	Suppose $2 \leq \vdim \in \nat$, $\adim = \vdim+1$, $p = \vdim$, $U =
	\rel^\adim$, $\varepsilon > 0$, and $\omega : \{ t \with 0 < t \leq 1
	\} \to \{ t \with 0 < t < \infty \}$ satisfies $\lim_{t \to 0+} \omega
	(t) = 0$.

	Then there exist $T \in \grass \adim \vdim$ and $V$ related to
	$\vdim$, $\adim$, $p$, and $U$ as in \ref{miniremark:situation} with
	\begin{gather*}
		\| V \| \restrict \rel^\adim \without \cball 01 =
		\mathscr{H}^\vdim \restrict T \without \cball 01, \quad
		\measureball{ \| V \|}{\cball 01} \leq 7 \unitmeasure \vdim,
		\\
		\dist (x,T) \leq \varepsilon \quad \text{for $x \in \spt \| V
		\|$}, \qquad \tint{}{} | \mathbf{h} ( V, x ) |^\vdim \ud \| V
		\| \, x \leq \varepsilon, \\
		\big \| \project{\Tan^\vdim ( \| V \|, x )} - \project T \big
		\| \leq \varepsilon \quad \text{and} \quad \density^\vdim ( \|
		V \|, x) \leq 3 \qquad \text{for $\| V \|$ almost all $x$}, \\
		T \subset \spt \| V \|, \quad \text{$\spt \| V \|$ is
		connected}, \\
		\density^\vdim ( \mathscr{H}^\vdim \restrict \spt \| V \|, 0 )
		= \infty, \quad \Tan ( \spt \| V \|, 0 ) = \rel^\adim
	\end{gather*}
	such that the metric $\varrho$ on $X = \spt \| V \|$, see
	\ref{miniremark:intrinsic_metric} and
	\ref{thm:intrinsic_metric}\,\eqref{item:instrinsic_metric:product},
	satisfies
	\begin{equation*}
		\limsup_{x \to 0} \varrho(0,x)/\omega ( |x| ) = \infty \quad
		\text{and} \quad \varrho(0,\cdot) \notin \SWloc \infty (V).
	\end{equation*}
\end{example}
\begin{proof} [Construction]
	Assume $\omega (t) \geq t$ for $0 < t \leq 1$. The projections $\pp :
	\rel^\adim \to \rel^\vdim$ and $\qq : \rel^\adim \to \rel$ defined by
	\begin{equation*}
		\pp (x) = (x_1, \ldots, x_\vdim ) \quad \text{and} \quad \qq
		(x) = x_\adim
	\end{equation*}
	for $x = (x_1, \ldots, x_\adim ) \in \rel^\adim$ will be employed, see
	\cite[5.1.9]{MR41:1976}. Let $T = \im \pp^\ast$. Choose a
	nonincreasing, locally Lipschitzian function $\gamma : \{ t \with 0 <
	t \leq 1 \} \to \rel$ such that $\lim_{t \to 0+} \gamma(t) = \infty$,
	$\gamma(1)=0$, and $\zeta : T \cap \{ x \with 0 < |x| < 1 \} \to \rel$
	defined by $\zeta (x)= \gamma(|x|)$ whenever $x \in T$ and $0 < |x| <
	1$ satisfies
	\begin{equation*}
		\tint{T \cap \oball 01}{} | \Der \zeta |^\vdim \ud
		\mathscr{H}^\vdim \leq 2^{-\vdim} \varepsilon,
	\end{equation*}
	see for instance \cite[4.43]{MR2424078}. Let $s_0=1$ and choose a
	strictly decreasing sequence $s_i$ of positive numbers such that
	\begin{equation*}
		s_i^\vdim \leq 2^{-i} \quad \text{and} \quad \gamma (s_i )
		\geq \gamma (s_{i-1}) + 2 \quad \text{for $i \in \nat$}.
	\end{equation*}
	Abbreviating $\pi = \boldsymbol{\Gamma} (1/2)^2$ ($\approx 3.14$), see
	\cite[3.2.13]{MR41:1976}, next choose sequences $\delta_i$, $\alpha_i$
	and $r_i$ such that for $i \in \nat$ the following eight conditions
	are satisfied:
	\begin{gather*}
		0 < \delta_{i+1} < \delta_i \leq \varepsilon, \quad 0 \leq
		\alpha_i \leq \pi/4, \quad 0 < r_i < \infty, \\
		r_i^2 = (r_i-\delta_i)^2 + s_i^2, \quad \sin \alpha_i =
		s_i/r_i \leq \varepsilon, \\
		\delta_i \leq s_{2i+2}, \quad i \, \omega ( \delta_i ) \leq 2
		s_i, \quad 4 \unitmeasure \vdim \vdim^\vdim r_i^{-\vdim} \leq
		2^{-i-\vdim} \varepsilon;
	\end{gather*}
	in fact the first equation is equivalent to $r_i = ( \delta_i^2 +
	s_i^2)/(2 \delta_i)$, hence it is sufficient to inductively
	choose $\delta_i$ small enough. Notice that
	\begin{equation*}
		\text{$r_i > \delta_i$ and $\cos \alpha_i =
		(r_i-\delta_i)/r_i$ for $i \in \nat$}, \quad \lim_{i \to
		\infty} s_i = 0, \quad \lim_{i \to \infty} \delta_i = 0.
	\end{equation*}
	Let $S_i = T \cap \{ x \with |x|=s_i \}$ for $i \in \nat \cup \{ 0
	\}$.

	Let $V_0 \in \RVar_\vdim ( \rel^\adim)$ by defined by
	\begin{equation*}
		V_0 (k) = \tint{T \cap \{ x \with |x|>1 \}}{} k(x,T) \ud
		\mathscr{H}^\vdim \, x \quad \text{for $k \in \mathscr{K} (
		\rel^\adim \times \grass \adim \vdim )$},
	\end{equation*}
	hence one obtains for $\theta \in \mathscr{D} ( \rel^\adim, \rel^\adim
	)$ that
	\begin{equation*}
		( \delta V_0 ) ( \theta ) = - \tint{S_0}{} x \bullet \theta
		(x) \ud \mathscr{H}^{\vdim-1} \, x.
	\end{equation*}
	For $i \in \nat$ define $V_i \in \RVar_\vdim ( \rel^\adim )$ by
	\begin{equation*}
		V_i (k) = \tint{T \cap \oball 0{s_{i-1}} \without \oball
		0{s_i}}{} k (x,T) (1+ \inf \{ 2 \cos \alpha_i, \zeta
		(x)-\gamma (s_{i-1}) \} ) \ud \mathscr{H}^\vdim \, x
	\end{equation*}
	for $k \in \mathscr{K} ( \rel^\adim \times \grass \adim \vdim )$
	and compute for $\theta \in \mathscr{D} ( \rel^\adim, \rel^\adim )$
	that
	\begin{align*}
		( \delta V_i ) ( \theta ) & = \tint{S_{i-1}}{} |x|^{-1} x
		\bullet \theta (x) \ud \mathscr{H}^{\vdim-1} \, x -
		\tint{A_i}{} \langle \project T ( \theta(x)), \Der \zeta (x)
		\rangle \ud \mathscr{H}^\vdim \, x \\
		& \phantom = \ - (1+2 \cos \alpha_i ) \tint{S_i}{} |x|^{-1} x
		\bullet \theta (x) \ud \mathscr{H}^{\vdim-1} \, x,
	\end{align*}
	where $A_i = T \cap \oball 0{s_{i-1}} \cap \{ x \with \gamma (|x|) <
	\gamma (s_{i-1}) + 2 \cos \alpha_i \}$. Define $A =
	\bigcup_{i=1}^\infty A_i$.

	Let $u = \qq^\ast (1)$. Define the sets
	\begin{equation*}
		B_i = \{ x \with x \bullet u > 0, |x-(\delta_i-r_i)u|=r_i
		\}\cup \{ x \with x \bullet u < 0, |x-(r_i-\delta_i)u|=r_i \}.
	\end{equation*}
	for $i \in \nat$. Let $C_i$ for $i \in \nat$ denote the closed convex
	hull of $B_i$ and verify
	\begin{equation*}
		C_i = \cball{(\delta_i-r_i)u}{r_i} \cap
		\cball{(r_i-\delta_i)u}{r_i}.
	\end{equation*}
	Notice that $0 \in C_i$, $\Clos B_i = \Bdry C_i$, and by
	\ref{miniremark:balls_in_balls}, also
	\begin{equation*}
		C_{i+1} \subset \Int C_i \subset \cball 01 \quad \text{for $i
		\in \nat$},
	\end{equation*}
	in particular $B = \bigcup_{i=1}^\infty B_i$ is an $\vdim$ dimensional
	submanifold of class $\infty$ of $\rel^\adim$. The condition $\sin
	\alpha_i \leq \varepsilon$ for $i \in \nat$ implies
	\begin{equation*}
		\| \project{\Tan (B,x)} - \project T \| \leq \varepsilon \quad
		\text{for $x \in B$},
	\end{equation*}
	see for instance \cite[5.1\,(2)]{kol-men:decay.v2}.  For $i \in \nat$
	define $W_i \in \RVar_\vdim ( \rel^\adim )$ by
	\begin{equation*}
		W_i(k) = \tint{B_i}{} k(x,\Tan(B_i,x)) \ud \mathscr{H}^\vdim
		\, x \quad \text{for $k \in \mathscr{K} (\rel^\adim \times
		\grass \adim \vdim )$}
	\end{equation*}
	and compute for $\theta \in \mathscr{D} ( \rel^\adim, \rel^\adim )$
	that
	\begin{equation*}
		( \delta W_i ) ( \theta ) = - \tint{B_i}{} \mathbf{h} (B_i,x)
		\bullet \theta (x) \ud \mathscr{H}^\vdim \, x + 2 \cos \alpha_i
		\tint{S_i}{} |x|^{-1} x \bullet
		\theta (x) \ud \mathscr{H}^{\vdim-1} \, x.
	\end{equation*}

	Since $\sum_{i=1}^\infty \| V_i + W_i \| ( \rel^\adim ) \leq 7
	\unitmeasure \vdim$, one may define $V \in \RVar_\vdim ( \rel^\adim )$
	by $V = V_0 + \sum_{i=1}^\infty (V_i+W_i)$. One computes and estimates
	\begin{gather*}
		( \delta V ) ( \theta ) = - \tint B{} \mathbf{h} (B,x) \bullet
		\theta (x) \ud \mathscr{H}^\vdim \, x - \tint A{} \langle
		\project T ( \theta (x)), \Der \zeta (x) \rangle \ud
		\mathscr{H}^\vdim \,x, \\
		| ( \delta V ) ( \theta ) | \leq \big ( \tint{B}{} |
		\mathbf{h}(B,\cdot)|^\vdim \ud \mathscr{H}^\vdim \big
		)^{1/\vdim} + \tint{T \cap \oball 01}{} | \Der \zeta |^\vdim
		\ud \mathscr{L}^\vdim \big )^{1/\vdim} \leq
		\varepsilon^{1/\vdim}
	\end{gather*}
	whenever $\theta \in \mathscr{D} ( \rel^\adim, \rel^\adim )$ and
	$\Lpnorm{\| V\|}{\vdim/(\vdim-1)} { \theta } \leq 1$. Using
	\ref{miniremark:prop_lp_spaces} and \cite[2.4.16,
	2.5.7\,(i)]{MR41:1976}, one deduces that $V$ is related to $\vdim$,
	$\adim$, $p$, and $U$ as in \ref{miniremark:situation} with
	\begin{gather*}
		\tint{}{} | \mathbf{h} ( V, x ) |^\vdim \ud \| V \| \, x \leq
		\varepsilon, \quad \| V \| \restrict \rel^\adim \without
		\cball 01 = \mathscr{H}^\vdim \restrict T \without
		\cball 01, \\
		\text{$\spt \| V \|$ is connected}, \qquad \density^\vdim ( \|
		V \|, x ) \leq 3 \quad \text{for $\| V \|$ almost all $x$}.
	\end{gather*}
	Moreover, defining $X = \spt \| V \|$ and noting $X = T \cup B$, one
	infers
	\begin{equation*}
		\Tan(X,0) = \rel^\adim, \quad \density^\vdim (
		\mathscr{H}^\vdim \restrict X, 0 ) = \infty;
	\end{equation*}
	in fact, to prove the second equation notice that
	\begin{equation*}
		\card ( X \cap \{ x \with \text{$\pp (x) = \pp (a)$ and $|\qq
		(x) | \leq t$} \} ) \geq 2i+3
	\end{equation*}
	whenever $a \in T \cap \cball 0t$, $s_{2i+2} \leq t < s_{2i}$, and $i
	\in \nat$.

	The well known structure of length minimising geodesics on spheres,
	see for instance \cite[Chap.\,3, Example 2.11; Chap.\,3, Proposition
	3.6]{MR1138207}, implies the following statement. \emph{If $i \in
	\nat$, then
	\begin{equation*}
		\varrho (0,\delta_iu) = s_i + \alpha_i r_i,
	\end{equation*}
	and if $i \in \nat$, $v \in T$, $|v| = 1$, and $g : \{ t \with 0 \leq
	t \leq s_i + \alpha_ir_i \} \to X$ satisfies
	\begin{align*}
		g(t) & = tv \quad \text{if $t \leq s_i$}, \\
		g(t) & = \sin (\alpha_i + (s_i-t)/r_i ) r_i v + \big (
		\delta_i-r_i+\cos ( \alpha_i + (s_i - t)/r_i ) r_i \big ) u
		\quad \text{if $s_i < t$}
	\end{align*}
	whenever $0 \leq t \leq s_i+\alpha_i r_i$, then
	$g(0)=0$, $g(s_i+\alpha_ir_i) = \delta_i u$, $\Lip g = 1$
	and $|g'(t)| = 1$ for $\mathscr{L}^1$ almost all $t \in \dmn g$.}
	Consequently, $\varrho (0,\delta_iu) \geq i \, \omega ( | \delta_i u|
	)$ for $i \in \nat$ and
	\begin{equation*}
		\limsup_{x \to 0} \varrho(0,x)/\omega ( |x| ) = \infty.
	\end{equation*}

	Suppose $f : \rel^\adim \to \rel$ were a Lipschitzian function
	satisfying $f(0)=0$ and
	\begin{equation*}
		| \derivative V {(\varrho(0,\cdot))} (x) - \derivative Vf(x) |
		\leq 1/2 \quad \text{for $\| V \|$ almost all $x \in \cball
		01$}.
	\end{equation*}
	Then \ref{thm:intrinsic_metric}\,\eqref{item:intrinsic_metric:point}
	in conjunction with \ref{remark:strict_local_sobolev_space_inclusion}
	and \cite[11.4\,(4)]{snulmenn:tv.v2} would imply
	\begin{equation*}
		| \Der (\varrho(0,\cdot)) (x) - \Der (f|X)(x) | \leq 1/2 \quad
		\text{for $\| V \|$ almost all $x \in \cball 01$}.
	\end{equation*}
	Whenever $i \in \nat$ one could select $v \in T$ with $|v|=1$ such
	that the function $g$ associated to $i$ and $v$ in the statement of
	the preceding paragraph would satisfy
	\begin{equation*}
		| \Der ( \varrho (0,\cdot)) ( g(t)) - \Der (f|X) (g (t)) |
		\leq 1/2 \quad \text{for $\mathscr{L}^1$ almost all $t \in
		\dmn g$},
	\end{equation*}
	hence, noting $\varrho ( 0, \cdot ) \circ g = \id{\dmn g}$ and $\Lip (
	f \circ g ) < \infty$, integration would yield
	\begin{align*}
		| \varrho (0,\delta_iu) - f(\delta_iu) | & \leq \tint
		0{s_i+\alpha_ir_i} | \langle g'(t), \Der (\varrho (0,\cdot) -
		f|X) ( g(t)) \rangle | \ud \mathscr{L}^1 \, t \\
		& \leq \varrho(0,\delta_iu)/2.
	\end{align*}
	Consequently, one would obtain
	\begin{equation*}
		f ( \delta_i u ) \geq \varrho (0,\delta_i u) /2 \quad
		\text{for $i \in \nat$}
	\end{equation*}
	in contradiction to $\limsup_{x \to 0} |f(x)|/|x| < \infty$. Therefore
	$\varrho (0,\cdot) \notin \SWloc \infty ( V )$.
\end{proof}
\begin{remark}
	Considering large balls centred at $0$, the preceding example also
	shows that the Reifenberg type flatness result of Allard in
	\cite[8.8]{MR0307015} does not extend to the case of dimensionally
	critical mean curvature, $p = \vdim$ in \ref{miniremark:situation}. To
	which extent the behaviour of integral varifolds satisfying $p =
	\vdim$ in \ref{miniremark:situation} is more regular is only partially
	understood.  Properties of the density ratio specific to the integral
	case were obtained by Kuwert and Sch{\"a}tzle in
	\cite[Appendix~A]{MR2119722} and \cite[3.9]{snulmenn.poincare}, see
	also \cite[7.6]{snulmenn:tv.v2}. On the other hand nonuniqueness of
	tangent cones occurs naturally for $p = \vdim$ even for integral
	varifolds associated to Lipschitzian functions, see Hutchinson and
	Meier \cite{MR831410}.
\end{remark}
\section{Further implications of critical mean curvature} In this section the
study of varifolds satisfying a dimensionally critical summability condition
on the mean curvature and a lower bound on their densities as described in
\ref{miniremark:situation} with $p = \vdim$ will be continued.  Initially,
estimates for generalised weakly differentiable functions are derived, see
\ref{thm:one_dim}--\ref{thm:global_special_sob}. Subsequently, continuous and
compact embeddings of Sobolev spaces into Lebesgue spaces and spaces of
continuous functions along with topological properties of the various Sobolev
spaces are compiled.  These results follow readily from the corresponding
results on generalised weakly differentiable functions obtained in
\cite[\S\S\,8--10, \S\,13]{snulmenn:tv.v2}, Section \ref{sec:rellich}, and
\ref{thm:one_dim}--\ref{thm:global_special_sob}.  The treatment includes local
Sobolev spaces in
\ref{thm:replacement-w1q-loc}--\ref{remark:rellich-local-embedding}, an
intermediate space between $\SWSob q (V,Y)$ and $\SWloc q (V,Y)$ in
\ref{thm:replacement-w1q}--\ref{remark:quotient-sobolov-space}, and Sobolev
spaces with ``zero boundary values'' in
\ref{thm:sob_poin_summary}--\ref{remark:rellich_embedding}.

In implementing this study, the case of one dimensional varifolds needs
particular care since the present hypotheses permit that the variation measure
of the first variation is not absolutely continuous with respect to the weight
measure of the varifold.
\begin{comment}
	The first statement concerns a local Sobolev estimate near a single
	``boundary point'' for one dimensional varifolds.

	Its proof is based on the Sobolev Poincar{\'e} inequality with several
	medians, see \cite[10.9]{snulmenn:tv.v2}, and involves the concept of
	distributional boundary of a set with respect to a varifold as defined
	in \cite[5.1]{snulmenn:tv.v2}.
\end{comment}
\begin{theorem} \label{thm:one_dim}
	Suppose $\vdim$, $\adim$, $p$, $U$, and $V$ are as in
	\ref{miniremark:situation}, $p = \vdim = 1$, $\adim \leq M < \infty$,
	$\Lambda = \Gamma_{\textup{\cite[10.1]{snulmenn:tv.v2}}} ( 1+M )$, $0
	< r < \infty$, $A = \{ x \with \oball xr \subset U \}$, $a \in A$,
	\begin{equation*}
		\| V \| ( U ) \leq M \unitmeasure 1 r, \quad \| \delta V
		\| ( U \without \{ a \} ) \leq ( 2 + \Lambda )^{-1},
	\end{equation*}
	$Y$ is a finite dimensional normed vectorspace, and $f \in \trunc
	(V,Y)$.

	Then there holds
	\begin{equation*}
		\eqLpnorm{\| V \| \restrict A} \infty f \leq \Gamma \big (
		r^{-1} \Lpnorm{\| V \|} 1f + \Lpnorm{\| V \|} 1 { \derivative
		Vf} \big ),
	\end{equation*}
	where $\Gamma = 2^4 M ( 1 + \Lambda )$.
\end{theorem}
\begin{proof}
	Assume $r = 1$ and $\Lpnorm{\| V \|} 1 f + \Lpnorm{\| V \|} 1
	{\derivative Vf} < \infty$.  In view of \cite[8.16]{snulmenn:tv.v2}
	one may also assume $Y = \rel$.  Choose $N \in \nat$ such that $N \leq
	4 M \leq N+1$ and abbreviate $X = \{ x \with \cball x{1/2} \subset U
	\}$.  Applying \cite[10.9\,(2)]{snulmenn:tv.v2} with $M$, $Q$, $r$,
	and $X$ replaced by $1+M$, $1$, $1/2$, and $\{ a \}$, one obtains a
	subset $\Upsilon$ of $\rel$ such that $1 \leq \card \Upsilon \leq N+1$
	and
	\begin{equation*}
		\eqLpnorm{\| V \| \restrict X} \infty g \leq \Lambda
		\Lpnorm{\| V \|} 1 {\derivative Vf}, \quad \text{where $g =
		\dist ( \cdot, \Upsilon ) \circ f$}.
	\end{equation*}
	Let $s = \Lambda \Lpnorm{\| V \|} 1 {\derivative Vf}$ and $B = \bigcup
	\{ \cball \upsilon s \with \upsilon \in \Upsilon \}$.

	Suppose $E$ is a connected component of $B$ and let $F = f^{-1} \lIm E
	\rIm$. The proof will be concluded by showing that
	\begin{equation*}
		\eqLpnorm{\| V \| \restrict A \cap F} \infty f \leq \Gamma
		\big ( \Lpnorm {\| V \|} 1f + \Lpnorm{\| V \|} 1 {\derivative
		Vf} \big ).
	\end{equation*}
	Hence one may assume $A \cap \spt ( \| V \| \restrict F ) \neq
	\varnothing$.  In view of \cite[8.15]{snulmenn:tv.v2} applying
	\cite[8.29]{snulmenn:tv.v2} with $f$ and $y$ replaced by $\dist (
	\cdot, E) \circ f$ and $0$ yields $\| \boundary VF \| (X) = 0$.
	Abbreviating $W = V \restrict F \times \grass \adim 1$, this implies
	\begin{equation*}
		\| \delta W \| ( X ) < \infty \quad \text{and} \quad \| \delta
		W \| ( X \without \{a\}) \leq 1/2.
	\end{equation*}
	Choose $\chi \in A \cap \spt \| W \|$ such that $\chi = a$ if $a \in
	\spt \| W \|$.  By \cite[4.8\,(1)\,(4)]{snulmenn:tv.v2} with $U$,
	$V$, and $a$ replaced by $X$, $W | \mathbf{2}^{X \times \grass \adim
	1}$, and $\chi$ one obtains that
	\begin{equation*}
		\| W \| ( X ) \geq \measureball{\| W \|}{ \oball \chi {1/2}}
		\geq 1/4.
	\end{equation*}
	Since $\diam E \leq \mathscr{L}^1 (E) \leq 2^4 Ms$, integrating the
	inequality
	\begin{equation*}
		\eqLpnorm{\| W \| \restrict X} \infty f \leq |f(x)| + \diam E
		\quad \text{for $\| W \|$ almost all $x \in X$}
	\end{equation*}
	with respect to $\| W \| \restrict X$ now yields the conclusion.
\end{proof}
\begin{comment}
	In the ``interior'' the corresponding estimate is much simpler and
	obtainable in all dimensions by localising the Sobolev inequality
	\cite[10.1\,(2)]{snulmenn:tv.v2}.
\end{comment}
\begin{theorem} \label{thm:local_sobolev}
	Suppose $\vdim$, $\adim$, $p$, $U$, $V$, and $\psi$ are as in
	\ref{miniremark:situation}, $p = \vdim$, $\Lambda =
	\Gamma_{\textup{\cite[10.1]{snulmenn:tv.v2}}} ( \adim )$, $\psi ( U )
	\leq \Lambda^{-1}$, $0 < r < \infty$, $A = \{ x \with \oball xr
	\subset U \}$,
	\begin{enumerate}
		\item \label{item:local_sobolev:m=1} either $\vdim = q = 1$,
		$\alpha = \infty$, and $\kappa = \Lambda$,
		\item \label{item:local_sobolev:q<m} or $1 \leq q < \vdim$,
		$\alpha = \vdim q/(\vdim-q)$, and $\kappa =
		\Lambda(\vdim-q)^{-1}$,
		\item \label{item:local_sobolev:m<q} or $1 < \vdim < q$,
		$\alpha = \infty$, and $\kappa = \Lambda^{1/(1/\vdim-1/q)}
		\| V \| ( U )^{1/\vdim-1/q} < \infty$,
	\end{enumerate}
	and $Y$ is a finite dimensional normed vectorspace.

	Then there holds
	\begin{equation*}
		\eqLpnorm{\| V \| \restrict A} \alpha f \leq \kappa \big (
		\Lpnorm{\| V \|} q {\derivative Vf} + r^{-1} \Lpnorm{\| V \|}
		q f \big ) \quad \text{for $f \in \trunc (V,Y)$}.
	\end{equation*}
\end{theorem}
\begin{proof}
	Assume $Y = \rel$ and $f \geq 0$ by \cite[8.16]{snulmenn:tv.v2} and
	$\gamma = r^{-1} \Lpnorm{\| V \|}qf + \Lpnorm{\| V \|} q{\derivative
	Vf} < \infty$. Whenever $g : U \to \rel$ is a nonnegative Lipschitzian
	function with compact support, $\sup \im g \leq 1$ and $\Lip g \leq
	r^{-1}$, one infers from \cite[8.20\,(4), 9.2, 9.4]{snulmenn:tv.v2}
	that
	\begin{equation*}
		g f \in \trunc_{\Bdry U} (V) \quad \text{and} \quad \Lpnorm{\|
		V \|}q{\derivative V{(gf)}} \leq \gamma,
	\end{equation*}
	hence \cite[10.1\,(2b)\,(2c)\,(2d)]{snulmenn:tv.v2} implies
	$\Lpnorm{\| V \|}{\alpha}{ gf } \leq \kappa \gamma$.
\end{proof}
\begin{comment}
	Next, estimates of the Lebesgue seminorm of a generalised weakly
	differentiable function with respect to the variation measure of the
	first variation of the varifold will be studied.
\end{comment}
\begin{lemma} \label{lemma:special_sob}
	Suppose $\vdim$, $\adim$, $p$, $U$, $V$, and $\psi$ are as in
	\ref{miniremark:situation}, $p = \vdim$, $\Lambda =
	\Gamma_{\textup{\cite[10.1]{snulmenn:tv.v2}}} ( \adim )$, $1 \leq q
	\leq \infty$, $f \in \trunc_{\Bdry U} (V)$, and
	\begin{gather*}
		\| V \| ( \{ x \with f(x) > y \} ) < \infty \quad \text{for $0
		< y < \infty$}, \qquad \psi ( \{ x \with f(x)>0 \} ) \leq
		\Lambda^{-1}, \\
		\Lpnorm{\| V\|} qf + \Lpnorm{\| V \|}q{\derivative Vf} <
		\infty.
	\end{gather*}

	Then there holds
	\begin{equation*}
		\Lpnorm{\| \delta V \|} q {f} \leq \Gamma \Lpnorm{\| V \|} q
		{f}^{1-1/q} \Lpnorm{\| V \|}{q}{\derivative Vf}^{1/q},
	\end{equation*}
	where $\Gamma = 2(1+\Lambda)$ and $0^0 = 1$.
\end{lemma}
\begin{proof}
	One may assume $q < \infty$ by \cite[8.33]{snulmenn:tv.v2} and that
	$f$ is bounded by \cite[8.12, 8.13\,(4), 9.9]{snulmenn:tv.v2}.  Define
	$\alpha = \infty$ if $\vdim = 1$ and $\alpha = q\vdim/(\vdim-1)$ if
	$\vdim > 1$. Then H{\"o}lder's inequality, in case $\vdim = 1$ in
	conjunction with \cite[8.33]{snulmenn:tv.v2}, implies
	\begin{equation*}
		\Lpnorm{\| \delta V\|}qf \leq \psi (\{ x \with f(x)>0
		\})^{1/(\vdim q)} \Lpnorm{\| V \|} \alpha f \leq
		\Lambda^{-1/(\vdim q)} \Lpnorm{\| V \|} \alpha f
	\end{equation*}
	and applying \cite[10.1\,(2b)\,(2c)]{snulmenn:tv.v2} with $f$ and $q$
	replaced by $f^q$ and $1$ in conjunction with \cite[8.6,
	9.9]{snulmenn:tv.v2} yields
	\begin{equation*}
		\Lpnorm{\| V \|} \alpha{f} \leq (\Lambda q)^{1/q} \Lpnorm{\| V
		\|}{1}{f^{q-1} \derivative Vf}^{1/q}
	\end{equation*}
	and the conclusion follows from H{\"o}lder's inequality and the fact
	$q^{1/q} \leq 2$.
\end{proof}
\begin{remark} \label{remark:special_sob}
	In particular, one infers
	\begin{equation*}
		\Lpnorm{\| \delta V \|} q {f} \leq \Gamma \big ( \Lpnorm{\| V
		\|} q {f} + \Lpnorm{\| V \|}{q}{\derivative Vf} \big ).
	\end{equation*}
\end{remark}
\begin{comment}
	Localising the preceding estimate, one obtains the following theorem.
\end{comment}
\begin{theorem} \label{thm:special_sob}
	Suppose $\vdim$, $\adim$, $p$, $U$, $V$, and $\psi$ are as in
	\ref{miniremark:situation}, $p = \vdim$, $\Lambda =
	\Gamma_{\textup{\cite[10.1]{snulmenn:tv.v2}}} ( \adim )$, $\psi ( U )
	\leq \Lambda^{-1}$, $1 \leq q \leq \infty$, $Y$ is a finite
	dimensional normed vectorspace, $f \in \trunc (V,Y)$, $0 < r <
	\infty$, and $A = \{ x \with \oball xr \subset U \}$.

	Then there holds
	\begin{equation*}
		\eqLpnorm{\| \delta V \| \restrict A} q {f} \leq \Gamma \big (
		r^{-1/q} \Lpnorm{\| V \|} q {f} + r^{1-1/q} \Lpnorm{\| V
		\|}{q}{\derivative Vf} \big ),
	\end{equation*}
	where $\Gamma = 4(1+\Lambda)$.
\end{theorem}
\begin{proof}
	Assume firstly $r = 1$, secondly $Y = \rel$ and $f \geq 0$ by
	\cite[8.16]{snulmenn:tv.v2} and thirdly $\gamma = \Lpnorm{\| V \|}qf +
	\Lpnorm{\| V \|} q{\derivative Vf} < \infty$. Whenever $g : U \to
	\rel$ is a nonnegative Lipschitzian function with compact support,
	$\sup \im g \leq 1$ and $\Lip g \leq 1$, one infers from
	\cite[8.20\,(4), 9.2, 9.4]{snulmenn:tv.v2} that
	\begin{equation*}
		g f \in \trunc_{\Bdry U} (V) \quad \text{and} \quad \Lpnorm{\|
		V \|}q{gf} + \Lpnorm{\| V \|}q{\derivative V{(gf)}} \leq 2
		\gamma,
	\end{equation*}
	hence \ref{lemma:special_sob} and \ref{remark:special_sob} imply
	$\Lpnorm{\| \delta V \|}{q}{ gf } \leq \Gamma \gamma$.
\end{proof}
\begin{comment}
	If $U = \rel^\adim$, a global estimate under a weaker condition on
	$\psi$ will be proven.
\end{comment}
\begin{theorem} \label{thm:global_special_sob}
	Suppose $\vdim$, $\adim$, $p$, $U$, $V$, and $\psi$ are as in
	\ref{miniremark:situation}, $p = \vdim$, $U = \rel^\adim$, and $\psi (
	\rel^\adim ) < \infty$.

	Then there exists a positive, finite number $\Gamma$ with the
	following property.

	If $1 \leq q \leq \infty$ and $Y$ is a finite dimensional normed
	vectorspace, then
	\begin{equation*}
		\Lpnorm{\| \delta V \|} qf \leq \Gamma \big ( \Lpnorm{\| V \|}
		qf + \Lpnorm{\| V \|} q{\derivative Vf} \big ) \quad \text{for
		$f \in \trunc (V,Y)$}.
	\end{equation*}
\end{theorem}
\begin{proof}
	Define $\Delta_1 = \Gamma_{\textup{\cite[10.1]{snulmenn:tv.v2}}} (
	\adim)$.  Choose a nonempty compact subset $K$ of $\rel^\adim$ such
	that $\psi ( \rel^\adim \without K ) \leq \Delta_1^{-1}$. If $\vdim =
	1$ then let
	\begin{equation*}
		M = \sup \big ( \{ \adim \} \cup \{ r^{-1} \measureball{\| V
		\|}{ \oball a{2r} } \with a \in K, 0 < r \leq 1 \} \big ),
	\end{equation*}
	hence $M < \infty$ by \cite[4.8\,(1)]{snulmenn:tv.v2}, and
	define $\Delta_2 = 2+\Gamma_{\textup{\cite[10.1]{snulmenn:tv.v2}}} ( 1
	+ M)$. If $\vdim > 1$ then let $\Delta_2 = \Delta_1$.  Next, choose $k
	\in \nat$ and $a_j \in K$, $0 < r_j \leq 1$ for $j = 1, \ldots, k$
	such that
	\begin{gather*}
		\measureball{\| V \|}{\oball {a_j}{2r_j}} \leq M \unitmeasure 1
		r_j \quad \text{if $\vdim = 1$}, \\
		\psi ( \oball {a_j}{2r_j} \without \{ a_j \} ) \leq \inf \{
		\Delta_1^{-1}, \Delta_2^{-1} \}, \quad K \subset {\textstyle
		\bigcup \{ \oball{a_j}{r_j} \with j = 1, \ldots, k \}},
	\end{gather*}
	and let $A = \rel^\adim \without \bigcup_{j=1}^k \oball {a_j}{r_j}$.
	Abbreviating
	\begin{gather*}
		\begin{aligned}
			\Delta_3 & = \sup \big \{ 2^5 M^2 \Delta_2 (1 + \psi (
			\rel^\adim ) ), 4 ( 1 + \Delta_1 ) \big \}, && \quad
			\text{if $\vdim = 1$}, \\
			\Delta_3 & = 4(1+\Delta_1), && \quad \text{if $\vdim >
			1$},
		\end{aligned} \\
		r = \inf \big ( \{ \dist (a, A ) \with a \in K \} \cup \{ r_j
		\with j = 1, \ldots, k \} \big ),
	\end{gather*}
	define $\Gamma = ( k + 1 ) \Delta_3 r^{-1}$.

	Suppose $1 \leq q \leq \infty$ and $Y$ is a finite dimensional normed
	vectorspace.

	Then one obtains
	\begin{equation*}
		\eqLpnorm{\| \delta V \| \restrict \oball {a_j}{r_j}} q f \leq
		\Delta_3 \big ( r^{-1} \Lpnorm{\| V \|}qf + \Lpnorm{\| V \|} q
		{\derivative Vf} \big )
	\end{equation*}
	for $j=1,\ldots, k$ from \ref{thm:one_dim}, \cite[8.33]{snulmenn:tv.v2}
	and H{\"o}lder's inequality if $\vdim = 1$ and from
	\ref{thm:special_sob} if $\vdim > 1$.  Moreover, \ref{thm:special_sob}
	yields
	\begin{equation*}
		\eqLpnorm{\| \delta V \| \restrict A} qf \leq \Delta_3 \big (
		r^{-1} \Lpnorm{\| V \|} qf + \Lpnorm{\| V \|} q{\derivative
		Vf} \big).
	\end{equation*}
	Summing the preceding inequalities yields the conclusion.
\end{proof}
\begin{comment}
	Turning to local Sobolev spaces, firstly two properties of their
	topologies implied by the preceding estimates are gathered.
\end{comment}
\begin{theorem} \label{thm:replacement-w1q-loc}
	Suppose $\vdim$, $\adim$, $p$, $U$, and $V$ are as in
	\ref{miniremark:situation}, $p = \vdim$, $1 \leq q \leq \infty$, $Y$
	is a finite dimensional normed vectorspace, and $\sigma_K : \SWloc q
	(V,Y) \to \rel$ satisfy
	\begin{equation*}
		\sigma_K (f) = \eqLpnorm{\| V \| \restrict K} qf +
		\eqLpnorm{\| V \| \restrict K} q {\derivative Vf} \quad
		\text{for $f \in \SWloc q (V,Y)$}
	\end{equation*}
	whenever $K$ is a compact subset of $U$.

	Then the following two statements hold.
	\begin{enumerate}
		\item \label{item:replacement-w1q-loc:top} The topology of
		$\SWloc q (V,Y)$ is induced by the seminorms $\sigma_K$
		corresponding to all compact subsets $K$ of $U$.
		\item \label{item:replacement-w1q-loc:abs} The set of
		\begin{equation*}
			(f,F) \in \Lploc q ( \| V \|, Y ) \times \Lploc q ( \|
			V \|, \Hom ( \rel^\adim, Y ) )
		\end{equation*}
		such that some $g \in \SWloc q (V,Y)$ satisfies
		\begin{equation*}
			f(x) = g(x) \quad \text{and} \quad F(x) = \derivative
			Vg (x) \quad \text{for $\| V \|$ almost all $x$}
		\end{equation*}
		is closed in $\Lploc q ( \| V \|, Y ) \times \Lploc q ( \| V
		\|, \Hom ( \rel^\adim, Y ) )$.
	\end{enumerate}
\end{theorem}
\begin{proof}
	Noting \ref{miniremark:lcs-induced-by-seminorms} and
	\ref{remark:strict_local_sobolev_space_inclusion}, and in case $\vdim
	= 1$ also \cite[4.8\,(2)]{snulmenn:tv.v2}, one may employ
	\ref{thm:one_dim}, \cite[8.33]{snulmenn:tv.v2}, and H{\"o}lder's
	inequality if $\vdim = 1$ and \ref{thm:special_sob} if $\vdim > 1$ to
	verify \eqref{item:replacement-w1q-loc:top}. Moreover,
	\ref{remark:completeness_wqloc} and
	\eqref{item:replacement-w1q-loc:top} yield
	\eqref{item:replacement-w1q-loc:abs}.
\end{proof}
\begin{remark} \label{remark:replacement-w1q-loc}
	If $V$ satisfies the hypotheses of
	\ref{thm:replacement-w1q-loc}, $\| \delta V \|$ is not
	absolutely continuous with respect to $\| V \|$ and $\| \delta V \| (
	\{ x \} ) = 0$ for $x \in U$, then the set occurring in
	\eqref{item:replacement-w1q-loc:abs} contains pairs $(f,F)$ with $f$
	being $\| \delta V \|$ nonmeasurable by \cite[2.2.4,
	2.9.2]{MR41:1976}. Such $V$ do exist if and only if $\vdim = 1$, see
	\cite[12.3]{snulmenn:tv.v2}.
\end{remark}
\begin{comment}
	Continuous embeddings into local Lebesgue spaces closely resemble the
	behaviour of Sobolev spaces in the Euclidean case.
\end{comment}
\begin{theorem} \label{thm:local_sobolev_embedding}
	Suppose $\vdim$, $\adim$, $p$, $U$, $V$, and $\psi$ are as in
	\ref{miniremark:situation}, $p = \vdim$,
	\begin{enumerate}
		\item \label{item:local_sobolev_embedding:m=1} either $\vdim =
		q = 1$ and $\alpha = \infty$,
		\item \label{item:local_sobolev_embedding:q<m} or $1 \leq q <
		\vdim$ and $\alpha = \vdim q/(\vdim-q)$,
		\item \label{item:local_sobolev_embedding:m<q} or $1 < \vdim <
		q$ and $\alpha = \infty$,
	\end{enumerate}
	and $Y$ is a finite dimensional normed vectorspace.

	Then $\SWloc q (V,Y)$ embeds continuously into $\Lploc \alpha ( \| V
	\|, Y )$.
\end{theorem}
\begin{proof}
	In view of \ref{miniremark:lcs-induced-by-seminorms},
	\ref{miniremark:lcs-quotient}, \ref{remark:lploc-quotient},
	\ref{remark:swloc-quotient}, and \cite[\printRoman 2.1.14]{MR0117523},
	it is enough to show that bounded sets in $\SWloc q (V,Y)$ are bounded
	in $\Lploc \alpha ( \| V \|, Y )$.  Noting
	\ref{remark:strict_local_sobolev_space_inclusion} and in case $\vdim =
	1$ also \cite[4.8\,(2)]{snulmenn:tv.v2}, this is a consequence of
	\ref{thm:one_dim} and H{\"o}lder's inequality if $\vdim =1$ and of
	\ref{thm:local_sobolev}\,\eqref{item:local_sobolev:q<m}\,\eqref{item:local_sobolev:m<q}
	if $\vdim > 1$.
\end{proof}
\begin{remark}
	In case of \eqref{item:local_sobolev_embedding:m=1} or
	\eqref{item:local_sobolev_embedding:m<q} a continuous embedding of
	$\SWloc q (V,Y)$ into $\mathscr{C} ( \spt \| V \|, Y )$ will be
	constructed in \ref{thm:sob_continuous_emb}.
\end{remark}
\begin{comment}
	In combination with the Rellich type embedding result for generalised
	weakly differentiable functions, see \ref{thm:rellich-in-measure}, one
	directly infers a Rellich type embedding for local Sobolev functions.
\end{comment}
\begin{theorem} \label{thm:rellich-local-embedding}
	Suppose $\vdim$, $\adim$, $p$, $U$, $V$, and $\psi$ are as in
	\ref{miniremark:situation}, $p = \vdim$,
	\begin{enumerate}
		\item either $\vdim = 1$, $1 \leq q \leq \infty$, and $1 \leq
		\alpha < \infty$,
		\item or $1 \leq q < \vdim$ and $1 \leq \alpha < \vdim
		q/(\vdim-q)$,
	\end{enumerate}
	and $Y$ is a finite dimensional normed vectorspace.

	Then bounded subsets of $\SWloc q ( V,Y )$ have compact closure in
	$\Lploc \alpha (\| V \|,Y)$.
\end{theorem}
\begin{proof}
	Noting \ref{remark:strict_local_sobolev_space_inclusion} and
	\ref{thm:local_sobolev_embedding}\,\eqref{item:local_sobolev_embedding:m=1}\,\eqref{item:local_sobolev_embedding:q<m},
	the conclusion is a consequence of \ref{thm:rellich-in-measure} and
	H{\"o}lder's inequality.
\end{proof}
\begin{comment}
	The following embedding theorem into continuous functions rests on the
	oscillation estimate \cite[13.1]{snulmenn:tv.v2}; the absence of an
	embedding into locally H{\"o}lder continuous functions shows
	a notable difference to the Euclidean case, see
	\ref{remark:sob_continuous_emb}.
\end{comment}
\begin{theorem} \label{thm:sob_continuous_emb}
	Suppose $\vdim$, $\adim$, $p$, $U$, and $V$ are as in
	\ref{miniremark:situation}, $p = \vdim$, either $1 = \vdim
	\leq q$ or $1 < \vdim < q$, and $Y$ is a finite dimensional normed
	vectorspace.

	Then for every $f \in \SWloc q (V,Y)$ there exists $L(f) \in
	\mathscr{C}( \spt \| V \|,Y)$ uniquely characterised by
	\begin{equation*}
		L(f)(x) = f(x) \quad \text{for $\| V \| + \| \delta V \|$
		almost all $x$}
	\end{equation*}
	and $L$ is a continuous linear map. Moreover, if additionally $q >
	1$, then $L$ maps bounded subsets of $\SWloc q ( V,Y)$ onto sets with
	compact closure in $\mathscr{C}(\spt \| V \|,Y)$.
\end{theorem}
\begin{proof}
	To prove the existence of $L(f)$, suppose $f \in \SWloc q ( V,Y)$ and
	$K$ is a compact subset of $U$. Choose a compact subset $C$ of $U$
	with $K \subset \Int C$ and locally Lipschitzian functions $f_i : U
	\to Y$ with
    	\begin{equation*}
		\eqLpnorm{(\| V \| + \| \delta V \| )\restrict C}q{f_i-f} +
		\eqLpnorm{\| V \| \restrict C}{q}{\derivative V{(f_i-f)}} \to
		0 \quad \text{as $i \to \infty$}.
    	\end{equation*}
	By \cite[2.3.10]{MR41:1976} one may assume
	\begin{equation*}
		( \| V \| + \| \delta V \| ) ( K \without A ) = 0, \quad
		\text{where $A = K \cap \left \{x \with \lim_{i \to \infty}
		f_i(x) = f(x) \right \}$}.
    	\end{equation*}
	The family $\{ f_i | K \cap \spt \| V \| \with i \in \nat \}$ is
	equicontinuous by \cite[4.8\,(1), 13.1]{snulmenn:tv.v2}, hence $f | A
	\cap \spt \| V \|$ is uniformly continuous.

	The continuity of $L$ follows from
	\ref{thm:local_sobolev_embedding}\,\eqref{item:local_sobolev_embedding:m=1}\,\eqref{item:local_sobolev_embedding:m<q}
	as $\mathscr{C} ( \spt \| V \|, Y )$ is homeomorphically included in
	$\Lploc \infty ( \| V \|, Y )$.

	To prove the postscript, suppose $q>1$ and $B$ is a bounded subset of
	$\SWloc q (V,Y)$, hence $L \lIm B \rIm$ is bounded in $\mathscr{C} (
	\spt \| V \|, Y )$. Moreover,
	\ref{remark:strict_local_sobolev_space_inclusion} and \cite[4.8\,(1),
	13.1]{snulmenn:tv.v2} yield that the family $\{ L(f)|K \with f \in B
	\}$ is equicontinuous whenever $K$ is a compact subset of $\spt \| V
	\|$. Consequently, the conclusion follows from the Ascoli theorem, see
	\cite[7.14, 7.18]{MR0370454}.
\end{proof}
\begin{remark} \label{remark:sob_continuous_emb}
	Notice that some functions $L(f)$ may not be locally H{\"o}lder
	continuous with any exponent by
	\ref{thm:intrinsic_metric}\,\eqref{item:intrinsic_metric:point} and
	\ref{example:geodesic_distance}.
\end{remark}
\begin{remark} \label{remark:two-crossing-lines}
	Simple examples, see \cite[5.5, 5.6]{snulmenn:tv.v2}, show that not
	all members of $f \in \trunc (V,Y)$ with
	\begin{equation*}
		\eqLpnorm{(\| V \| + \| \delta V \| ) \restrict K}{q}{f} +
		\eqLpnorm{\| V\| \restrict K}{q}{\derivative Vf} < \infty
	\end{equation*}
	whenever $K$ is a compact subset of $U$ admit a continuous function $g
	: \spt \| V \| \to Y$ which is $\| V \| + \| \delta V \|$ almost equal
	to $f$.
\end{remark}
\begin{remark} \label{remark:rellich-local-embedding}
	\ref{thm:local_sobolev_embedding}, \ref{thm:rellich-local-embedding},
	and \ref{thm:sob_continuous_emb} are most useful in conjunction with
	the alternate description of the topology of $\SWloc q (V,Y)$ obtained
	in
	\ref{thm:replacement-w1q-loc}\,\eqref{item:replacement-w1q-loc:top}.
\end{remark}
\begin{comment}
	Sometimes, the seminormed space $E$ defined in the following theorem
	may function as alternate substitute to $\SWSob q (V,Y)$, for the
	Euclidean Sobolev space.
\end{comment}
\begin{theorem} \label{thm:replacement-w1q}
	Suppose $\vdim$, $\adim$, $p$, $U$, $V$, and $\psi$ are as in
	\ref{miniremark:situation}, $p = \vdim$, $1 \leq q \leq \infty$, $Y$
	is a finite dimensional normed vectorspace, $\sigma : \SWloc q ( V,Y )
	\to \overline \rel$ satisfies
	\begin{equation*}
		\sigma (f) = \Lpnorm{\| V \|} q f + \Lpnorm{\| V \|} q
		{\derivative Vf} \quad \text{for $f \in \SWloc q ( V,Y )$},
	\end{equation*}
	and $E = \{ f \with \sigma (f) < \infty \}$.

	Then the following six statements hold.
	\begin{enumerate}
		\item \label{item:replacement-w1q:complete} The vectorspace
		$E$ is $\sigma$ complete.
		\item \label{item:replacement-w1q:lip} The subspace $Y^U \cap
		\{ g \with \text{$g$ locally Lipschitzian, $\sigma (g) <
		\infty$} \}$ is $\sigma$ dense in $E$.
		\item \label{item:replacement-w1q:q<oo_dense} If $q < \infty$,
		then $\mathscr{E} (U,Y) \cap \{ g \with \text{$\sigma (g) <
		\infty$, $\spt g$ bounded} \}$ is $\sigma$ dense in $E$.
		\item \label{item:replacement-w1q:q<oo_separable} If $q <
		\infty$, then $E$ is $\sigma$ separable.
		\item \label{item:replacement-w1q:equal} If $U = \rel^\adim$
		and $\psi ( \rel^\adim ) < \infty$, then $E = \SWSob q(V,Y)$.
		\item \label{item:replacement-w1q:abs} The set of
		\begin{equation*}
			(f,F) \in \Lp q ( \| V \|, Y ) \times \Lp q ( \| V \|,
			\Hom ( \rel^\adim, Y ) )
		\end{equation*}
		such that there exists $g \in E$ satisfying
		\begin{equation*}
			f(x) = g (x) \quad \text{and} \quad F (x) =
			\derivative Vg(x) \quad \text{for $\| V \|$ almost all
			$x$}
		\end{equation*}
		is closed in $\Lp q ( \| V \|, Y ) \times \Lp q ( \| V \|,
		\Hom ( \rel^\adim, Y ) )$.
	\end{enumerate}
\end{theorem}
\begin{proof}
	\eqref{item:replacement-w1q:complete} follows from
	\ref{remark:completeness_wqloc} and
	\ref{thm:replacement-w1q-loc}\,\eqref{item:replacement-w1q-loc:top}.
	Replacing $\SWSob q (V,Y)$ and $\SWnorm qV\cdot$ by $E$ and $\sigma$
	in the argument of \ref{remark:density_sobolev_space}, a proof of
	\eqref{item:replacement-w1q:lip} and
	\eqref{item:replacement-w1q:q<oo_dense} results.
	\ref{miniremark:prop_lp_spaces} implies
	\eqref{item:replacement-w1q:q<oo_separable}.
	\ref{remark:strict_local_sobolev_space_inclusion} and
	\ref{thm:global_special_sob} yield
	\eqref{item:replacement-w1q:equal}.  \eqref{item:replacement-w1q:abs}
	follows from
	\ref{thm:replacement-w1q-loc}\,\eqref{item:replacement-w1q-loc:abs}.
\end{proof}
\begin{remark} \label{remark:quotient-sobolov-space}
	In view of \ref{miniremark:lcs-quotient} and
	\ref{thm:replacement-w1q}\,\eqref{item:replacement-w1q:complete}, the
	quotient space
	\begin{equation*}
		Q = E \big / \{ f \with \sigma (f) = 0 \}
	\end{equation*}
	is a Banach space normed by $\sigma \circ \pi^{-1}$, where $\pi : E
	\to Q$ denotes the canonical projection. If $1 < q < \infty$, then $Q$
	is reflexive by \ref{example:quotient-lp-spaces},
	\ref{thm:replacement-w1q}\,\eqref{item:replacement-w1q:abs} and
	\cite[\printRoman 2.3.23]{MR0117523}.
\end{remark}
\begin{comment}
	Turning to Sobolev spaces with ``zero boundary values'', firstly the
	corresponding Sobolev inequalities are stated.
\end{comment}
\begin{theorem} \label{thm:sob_poin_summary}
	Suppose $\vdim$, $\adim$, $p$, $U$, $V$, and $\psi$ are as in
	\ref{miniremark:situation}, $p = \vdim$, $\Lambda =
	\Gamma_{\textup{\cite[10.1]{snulmenn:tv.v2}}} ( \adim )$, $1 \leq q
	\leq \infty$, $Y$ is a finite dimensional normed vectorspace, $f \in
	\SWzero q ( V,Y )$, and $E = U \without \{ x \with f(x) = 0 \}$.

	Then the following three statements hold.
	\begin{enumerate}
		\item \label{item:sob_poin_summary:global:p=m=1} If $\vdim
		= 1$ and $\psi ( E ) \leq \Lambda^{-1}$, then
		\begin{equation*}
			\Lpnorm{\| V \|}{\infty}{f} \leq \Lambda \, \Lpnorm{\|
			V \|}{1}{ \derivative{V}{f} }.
		\end{equation*}
		\item \label{item:sob_poin_summary:global:q<m=p} If $1 \leq
		\alpha < \vdim$ and $\psi ( E ) \leq \Lambda^{-1}$, then
		\begin{equation*}
			\Lpnorm{\| V \|}{\vdim \alpha/(\vdim-\alpha)}{f} \leq
			\Lambda (\vdim-\alpha)^{-1} \Lpnorm{\| V \|}{\alpha}{
			\derivative{V}{f} }.
		\end{equation*}
		\item \label{item:sob_poin_summary:global:p=m<q} If $1 < \vdim
		< \alpha \leq \infty$ and $\psi ( E ) \leq \Lambda^{-1}$, then
		\begin{equation*}
			\Lpnorm{\| V \|}{\infty}{f} \leq
			\Lambda^{1/(1/\vdim-1/\alpha)} \| V \| (
			E)^{1/\vdim-1/\alpha} \Lpnorm{\| V \|}{\alpha}{
			\derivative{V}{f} },
		\end{equation*}
		here $0 \cdot \infty = \infty \cdot 0 = 0$.
	\end{enumerate}
\end{theorem}
\begin{proof}
	In view of \ref{remark:zero_finite_measure},
	\ref{thm:zero_implies_zero} and \cite[8.16]{snulmenn:tv.v2}, the
	conclusion is a consequence of
	\cite[10.1\,(2b)\,(2c)\,(2d)]{snulmenn:tv.v2} with $q$ replaced by
	$\alpha$.
\end{proof}
\begin{comment}
	In relation to the topology of $\SWzero q (V,Y)$, also the following
	consequence of the Sobolev inequalities becomes relevant.
\end{comment}
\begin{theorem} \label{thm:special-seminorm-control}
	Suppose $\vdim$, $\adim$, $p$, $U$, $V$, and $\psi$ are as in
	\ref{miniremark:situation}, $p = \vdim$, $\Lambda =
	\Gamma_{\textup{\cite[10.1]{snulmenn:tv.v2}}} ( \adim )$, $1 \leq q
	\leq \infty$, $Y$ is a finite dimensional normed vectorspace, $f \in
	\SWzero q (V,Y)$, and $\psi (U \without \{ x \with f(x)=0 \} ) \leq
	\Lambda^{-1}$.

	Then there holds
	\begin{equation*}
		\Lpnorm{\| \delta V \|} qf \leq \Gamma \big ( \Lpnorm{\| V \|}
		qf + \Lpnorm{\| V \|} q {\derivative Vf} \big ),
	\end{equation*}
	where $\Gamma = 2 ( 1 + \Lambda)$.
\end{theorem}
\begin{proof}
	In view of \ref{remark:zero_finite_measure},
	\ref{thm:zero_implies_zero}, and \cite[8.16]{snulmenn:tv.v2}, the
	conclusion is a consequence \ref{lemma:special_sob} and
	\ref{remark:special_sob}.
\end{proof}
\begin{comment}
	In the Euclidean case the following theorem is a direct consequence of
	a suitable Poincar{\'e} inequality; here also
	\ref{thm:special-seminorm-control} enters.
\end{comment}
\begin{theorem} \label{thm:alternate-seminorm-w0q}
	Suppose $\vdim$, $\adim$, $p$, $U$, $V$, and $\psi$ are as in
	\ref{miniremark:situation}, $p = \vdim$, $\| V \| ( U ) < \infty$,
	$\Lambda = \Gamma_{\textup{\cite[10.1]{snulmenn:tv.v2}}} ( \adim )$,
	$\psi ( U ) \leq \Lambda^{-1}$, $1 \leq q \leq \infty$, $Y$ is a
	finite dimensional normed vectorspace, and $\tau : \SWzero q (V,Y) \to
	\rel$ satisfies
	\begin{equation*}
		\tau (f) = \Lpnorm {\| V \|}q{\derivative Vf} \quad \text{for
		$f \in \SWzero q (V,Y)$}.
	\end{equation*}

	Then the topology of $\SWzero q (V,Y)$ is induced by $\tau$.
\end{theorem}
\begin{proof}
	This is a consequence of \ref{thm:sob_poin_summary},
	\ref{thm:special-seminorm-control}, and H{\"o}lder's inequality.
\end{proof}
\begin{comment}
	Finally, the Sobolev inequalities, see
	\ref{thm:sob_poin_summary}, and the Rellich type embedding for
	generalised weakly differentiable functions, see
	\ref{thm:rellich-in-measure}, combine to the following Rellich type
	embedding for Sobolev functions with ``zero boundary values''.
\end{comment}
\begin{theorem} \label{thm:rellich_embedding}
	Suppose $\vdim$, $\adim$, $p$, $U$, $V$, and $\psi$ are as in
	\ref{miniremark:situation}, $p = \vdim$, $\| V \| ( U ) < \infty$,
	$\Lambda = \Gamma_{\textup{\cite[10.1]{snulmenn:tv.v2}}} ( \adim )$,
	$\psi ( U ) \leq \Lambda^{-1}$,
	\begin{enumerate}
		\item \label{item:rellich_embedding:m=1} either $\vdim = 1$, $1
		\leq q \leq \infty$, and $1 \leq \alpha < \infty$,
		\item \label{item:rellich_embedding:m>1} or $1 \leq q <
		\vdim$ and $1 \leq \alpha < \vdim q/(\vdim-q)$,
	\end{enumerate}
	and $Y$ is a finite dimensional normed vectorspace.

	Then $\SWzero q(V,Y)$ embeds compactly into $\Lp \alpha (\| V \|,Y)$.
\end{theorem}
\begin{proof}
	Bounded subsets of $\SWzero q (V,Y)$ are bounded in $\Lp \infty ( \| V
	\|, Y )$ if $\vdim = 1$ by
	\ref{thm:sob_poin_summary}\,\eqref{item:sob_poin_summary:global:p=m=1}
	and bounded in $\Lp {\vdim q/(\vdim-q)} ( \| V \|, Y )$ if $\vdim > 1$
	by
	\ref{thm:sob_poin_summary}\,\eqref{item:sob_poin_summary:global:q<m=p}.
	Therefore \ref{thm:rellich-in-measure} and H{\"o}lder's inequality
	imply the conclusion.
\end{proof}
\begin{remark} \label{remark:rellich_embedding}
	The preceding corollary is most useful in conjunction with the
	alternate description of the topology of $\SWzero q (V,Y)$ obtained
	in \ref{thm:alternate-seminorm-w0q}.
\end{remark}
\section{Comparison to other Sobolev spaces} In this section, the notion of
Sobolev space developed in the present paper will be compared to the notion of
strong Sobolev space for finite Radon measures from Bouchitt{\'e}, Buttazzo
and Fragal{\`a} in \cite{MR1857850}, see
\ref{miniremark:setup_bss}--\ref{miniremark:prelim-comparison}, and the space
$\mathbf{W}(V,\rel)$ defined in \cite[8.28]{snulmenn:tv.v2} will be compared
to the weak Sobolev space for finite Radon measures introduced in
Bouchitt{\'e}, Buttazzo and Fragal{\`a} in \cite{MR1857850}, see
\ref{miniremark:comparison-weak-sobolev-spaces}.  In both cases, this mainly
amounts to relating the two notions of tangent spaces involved by means of the
results of Fragal{\`a} and Mantegazza in \cite{MR1686704}.
\begin{comment}
	The following space occurs in the pointwise variant of the definition
	of the tangent space of Bouchitt{\'e}, Buttazzo and Fragal{\`a} in
	\cite{MR1857850} given below.
\end{comment}
\begin{miniremark} \label{miniremark:lip_space}
	Suppose $X$ is a locally compact, separable metric space and
	\begin{equation*}
		E = \mathscr{C} (X) \cap \{ f \with \Lip f \leq 1 \}.
	\end{equation*}

	Then $\mathscr{C} (X)$ and $\rel^X$ endowed with the Cartesian product
	topology induce the same metrisable topology on $E$, see
	\ref{miniremark:lcs-induced-by-seminorms},
	\ref{remark:continuous-functions-complete-lcs}, and
	\cite[2.10.21]{MR41:1976}. Also, $E$ is separable by
	\ref{remark:kx_dense_in_cx} and \cite[2.2, 2.23]{snulmenn:tv.v2}.
	Consequently, the Borel family of $E$ is generated by the sets $E \cap
	\{ f \with f(x) < t \}$ corresponding to all $x \in X$ and $t \in
	\rel$.
\end{miniremark}
\begin{comment}
	For finite Radon measures the concept of tangent space of
	Bouchitt{\'e}, Buttazzo and Fragal{\`a} in \cite{MR1857850} is
	implemented as follows.
\end{comment}
\begin{miniremark} \label{miniremark:setup_bss}
	Suppose $\mu$ is a Radon measure over $\rel^\adim$ with $\mu (
	\rel^\adim ) < \infty$, $1 \leq q \leq \infty$, $1 \leq r \leq
	\infty$, $1/q+1/r=1$, and $C = \{ ( a, \cball{a}{s} ) \with \text{$a
	\in \rel^\adim$, $0 < s < \infty$} \}$. Define $Z$ to be the vector
	subspace of $\Lp r ( \mu, \rel^\adim)$ consisting of those $\theta \in
	\Lp{r} ( \mu, \rel^\adim)$ such that there exists $0 \leq \kappa <
	\infty$ satisfying
	\begin{equation*}
		\tint{}{} \langle \theta, \Der \zeta \rangle \ud \mu \leq
		\kappa \, \Lpnorm{\mu}{q}{\zeta} \quad \text{for $\zeta \in
		\mathscr{D} ( \rel^\adim , \rel )$}.
	\end{equation*}
	Define vector subspaces of $\rel^\adim$ by
	\begin{equation*}
		P(x) = \rel^\adim \cap \left \{ u \with \text{$u = ( \mu, C )
		\aplim_{\chi \to x} \theta ( \chi )$ for some $\theta \in Z$}
		\right \} \quad \text{for $x \in \rel^\adim$}.
	\end{equation*}
	Whenever $D$ is a nonempty countable $\mdistance \cdot \mu$ dense
	subset of $Z$ there holds
	\begin{equation*}
		P(x) = \Clos \{ \theta(x) \with \theta \in D \} \quad \text{for
		$\mu$ almost all $x$};
	\end{equation*}
	in fact, defining $E$ as in \ref{miniremark:lip_space} with $X =
	\rel^\adim$ and $g : \rel^\adim \to E$ by
	\begin{equation*}
		g(x)(u) = \inf \{ | u-\theta(x) | \with \theta \in D \} \quad
		\text{for $x,u \in \rel^\adim$},
	\end{equation*}
	$g$ is $\mu$ measurable and one observes that the equation holds at
	$x$ whenever $g$ and all members of $D$ are $(\mu,C)$ approximately
	continuous at $x$ which is the case for $\mu$ almost all $x$ by
	\cite[2.8.18, 2.9.13]{MR41:1976}. Consequently, in view of
	\cite[p.\,59]{MR0467310}, the function $P$ is a representative of the
	equivalence class introduced under the name ``$T_\mu^q$'' by
	Bouchitt\'e, Buttazzo and Seppecher in \cite[p.\,38]{MR1424348} in case
	$1 < q < \infty$ and by Bouchitt\'e, Buttazzo and Fragal{\`a} in
	\cite[p.\,403]{MR1857850} for all~$q$.
\end{miniremark}
\begin{comment}
	The next example illustrates the preceding concept of tangent plane.
\end{comment}
\begin{miniremark} \label{miniremark:tangent_planes}
	Suppose $1 < q < \infty$, $U = \rel \cap \oball 01$, $f : U \to \rel
	\cap \{ t \with t \geq 1 \}$ is a Borel function,
	\begin{equation*}
		\tint ab f^{1/(q-1)} \ud \mathscr{L}^1 = \infty \quad
		\text{whenever $-1 < a < b < 1$},
	\end{equation*}
	and $\mu$ is the Radon measure over $\rel$ satisfying $\mu (A) =
	\tint{A \cap U}{\ast} 1/f \ud \mathscr{L}^1$ for $A \subset \rel$.
	Then $\mu \leq \mathscr{L}^1 \restrict U$ and
	\begin{equation*}
		P(x) = \{ 0 \} \quad \text{for $\mu$ almost all $x$},
	\end{equation*}
	see \ref{miniremark:setup_bss}; in fact, if $\theta \in Z$, then $g =
	\theta / f$ satisfies
	\begin{equation*}
		\sup \big \{ \tint U{} \langle g, \Der \zeta \rangle \ud
		\mathscr{L}^1 \with \text{$\zeta \in \mathscr{D} ( U, \rel )$
		and $\eqLpnorm{\mathscr{L}^1 \restrict U} q \zeta \leq 1$}
		\big \} < \infty
	\end{equation*}
	and consequently $g$ is $\mathscr{L}^1 \restrict U$ almost equal to a
	continuous real valued function with domain $U$ by
	\ref{miniremark:prop_lp_spaces} and \cite[2.5.7\,(i),
	4.5.9\,(30)\,(\printRoman 1)\,(\printRoman 5), 4.5.16]{MR41:1976}
	implying
	\begin{gather*}
		\tint {U}{} g^{q/(q-1)} f^{1/(q-1)} \ud \mathscr{L}^1
		= \tint{}{} \theta^{q/(q-1)} \ud \mu < \infty, \\
		g(x) = 0 \quad \text{for $\mathscr{L}^1 \restrict U$ almost
		all $x$}, \qquad \theta (x) = 0 \quad \text{for $\mu$ almost
		all $x$}.
	\end{gather*}
	If additionally
	\begin{equation*}
		q < s < \infty \quad \text{and} \quad \tint U{}
		f^{1/(s-1)} \ud \mathscr{L}^1 < \infty,
	\end{equation*}
	then $\Lpnorm \mu {s/(s-1)} f < \infty$, hence $\theta f \in Z'$ for
	$\theta \in \mathscr{D} ( U, \rel )$ and, by \cite[2.8.18,
	2.9.13]{MR41:1976},
	\begin{equation*}
		P'(x) = \rel \quad \text{for $\mu$ almost all $x$},
	\end{equation*}
	where $Z'$ and $P'$ are defined as $Z$ and $P$ in
	\ref{miniremark:setup_bss} with $q$ replaced by $s$.  In any case,
	$\Tan^1 ( \mu, x ) = \rel$ for $\mu$ almost all $x$ by Allard
	\cite[3.5\,(1a)]{MR0307015}.
	
	Whenever $1 < q < s < \infty$ a function $f$ satisfying the
	conditions in the preceding paragraph may be constructed using a
	member of $\Lp 1 ( \mathscr{L}^1) \without \Lploc{(s-1)/(q-1)} (
	\mathscr{L}^1 )$ and the family of translations corresponding to a
	countable dense subset of $\rel$.
\end{miniremark}
\begin{remark} \label{remark:di_marino_speight}
	A very similar example was employed by Di~Marino and Speight, see
	\cite[Theorem~1]{MR3411142}, to show that ``the $p$-weak gradient
	depends on $p$''.  In fact, Di~Marino pointed out to the author that
	the fact that $P$ may depend on~$q$ for $\mu$ almost all $x$ could
	alternately be deduced from Di~Marino and Speight \cite{MR3411142}
	employing the equivalences established by Bouchitt{\'e}, Buttazzo, and
	Seppecher \cite{MR1424348} and Ambrosio, Gigli, and Savar{\'e}
	\cite{MR3090143}.
\end{remark}
\begin{miniremark} \label{miniremark:setup_bss_continued}
	Continuing \ref{miniremark:setup_bss}, results concerning the question
	which conditions on $V \in \RVar_\vdim ( \rel^\adim )$ guarantee
	\begin{equation*}
		P (x) = \Tan^\vdim ( \| V \|, x ) \quad \text{for $\| V \|$
		almost all $x$},
	\end{equation*}
	where $P$ is defined with reference to $\mu = \| V \|$, will be
	briefly summarised. Firstly, in view of Allard
	\cite[3.5\,(1)]{MR0307015}, Fragal\`a and Mantegazza
	\cite[2.4]{MR1686704} implies
	\begin{equation*}
		P (x) \subset \Tan^\vdim ( \| V \|, x ) \quad
		\text{for $\| V \|$ almost all $x$}
	\end{equation*}
	whenever $\mu = \| V \|$ for some $\adim \geq \vdim \in \nat$, $V \in
	\RVar_\vdim (\rel^\adim)$ with $\| V \| ( \rel^\adim ) < \infty$. If
	additionally $\| \delta V \|$ is a Radon measure and, in case $r>1$,
	also
	\begin{gather*}
		\mathbf{h}(V,\cdot) \in \Lploc r ( \| V \|, \rel^\adim ), \\
		( \delta V ) ( \theta ) = - \tint{}{} \mathbf{h} (V,x) \bullet
		\theta (x) \ud \| V \| \, x \quad \text{for $\theta \in
		\mathscr{D} ( \rel^\adim, \rel^\adim )$},
	\end{gather*}
	then equality holds; in fact, whenever $u \in \rel^\adim$ and $\varrho
	\in \mathscr{D} ( \rel^\adim, \rel )$ the function mapping $x \in
	\rel^\adim$ with $\Tan^\vdim ( \| V\|, x ) \in \grass \adim \vdim$
	onto $\varrho(x) \project{\Tan^\vdim ( \| V \|, x )} (u)$ belongs to
	$Z$. The latter argument is a variant of Fragal\`a and
	Mantegazza \cite[3.8]{MR1686704}.
\end{miniremark}
\begin{comment}
	Now, under suitable conditions, the presently introduced Sobolev space
	may be identified with the strong Sobolev space of Bouchitt{\'e},
	Buttazzo, and Fragal{\`a} \cite{MR1857850}.
\end{comment}
\begin{miniremark} \label{miniremark:prelim-comparison}
	If $\vdim$, $\adim$, $U$, $V$, and $\psi$ are as in
	\ref{miniremark:situation}, $p = \vdim$, $U = \rel^\adim$, $\| V \| (
	\rel^\adim ) < \infty$, $1 \leq q < \infty$,
	\begin{equation*}
		P (x) = \Tan^\vdim ( \| V \|, x ) \quad \text{for $\| V \|$
		almost all $x$},
	\end{equation*}
	where $P(x)$ is related to $\mu = \| V \|$ and $q$ as in
	\ref{miniremark:setup_bss}, and $\sigma$, $Q$, and $\pi$ are as in
	\ref{thm:replacement-w1q} and \ref{remark:quotient-sobolov-space},
	then $H_{\| V \|}^{1,q} ( \rel^\adim )$ with notion of derivative
	$\nabla_{\| V \|}$ and norm $\| \cdot \|_{1,q,\| V \|}$ defined by
	Bouchitt{\'e}, Buttazzo, and Fragal{\`a} in \cite[p.\,403]{MR1857850}
	is isometrically isomorphic to $Q$ with notion of derivative induced
	by $\derivative V{}$ and norm $\sigma \circ \pi^{-1}$ by
	\ref{remark:eq_classes} and
	\ref{thm:replacement-w1q}\,\eqref{item:replacement-w1q:complete}\,\eqref{item:replacement-w1q:q<oo_dense}.
	It appears to be unknown whether the condition on $P$ is redundant,
	see \ref{miniremark:setup_bss_continued}.
\end{miniremark}
\begin{comment}
	Finally, turning to the weak Sobolev space of Bouchitt{\'e},
	Buttazzo, and Fragal{\`a} \cite{MR1857850}, a comparison may be given
	as follows.
\end{comment}
\begin{miniremark} \label{miniremark:comparison-weak-sobolev-spaces}
	Suppose $\vdim, \adim \in \nat$, $\vdim \leq \adim$, $U$ is an
	open subset of $\rel^\adim$, $V \in \RVar_\vdim ( U )$, $\| V \| (
	\rel^\adim ) < \infty$, $\| \delta V \|$ is a Radon measure absolutely
	continuous with respect to $\| V \|$, and $\mathbf{h}(V,\cdot) \in \Lp
	\infty ( \| V \|, \rel^\adim )$.  In particular,
	\ref{miniremark:setup_bss_continued} implies
	\begin{equation*}
		P(x) = \Tan^\vdim ( \| V \|, x ) \quad \text{for $\| V \|$
		almost all $x$},
	\end{equation*}
	where $P(x)$ is related to $\mu = \| V \|$ and $q$ as in
	\ref{miniremark:setup_bss}.  Recalling
	\ref{example:quotient-lp-spaces}, \ref{remark:lploc-quotient}, and
	\cite[8.27]{snulmenn:tv.v2} and denoting by $D : \mathbf{W} (V,\rel)
	\to L_1^{\mathrm{loc}} ( \| V \|, \Hom ( \rel^\adim, \rel ) )$ the
	notion of derivative implicit there, define the quotient space
	\begin{align*}
		& Q = \big ( \mathbf{W} (V,\rel) \cap \Lp q ( \| V \| ) \cap
		\{ f \with D(f) \in L_q ( \| V \|, \Hom ( \rel^\adim, \rel ) )
		\} \big ) \big / W, \\
		& \qquad \text{where $W = \mathbf{W} (V,\rel) \cap \{ f \with
		\text{$f(x)=0$ for $\| V \|$ almost all $x$} \}$},
	\end{align*}
	with associated canonical projection $\pi$, and define the
	value of a norm on $Q$ at $f$ to be the sum of the $L_q ( \| V \|,
	\rel )$ norm of $f$ and the $L_q ( \| V \|, \Hom ( \rel^\adim, \rel )
	)$ norm of $(D \circ \pi^{-1})(f)$. Then $Q$ with notion of derivative
	$D \circ \pi^{-1}$ is isometrically isomorphic to the space $W_{\| V
	\|}^{1,q} ( U )$ with notion of derivative $D_{\| V \|}$ defined by
	Bouchitt{\'e}, Buttazzo, and Fragal{\`a} in \cite[p.\,403]{MR1857850}.
\end{miniremark}

\addcontentsline{toc}{section}{\numberline{}References}
\bibliography{UlrichMenne7v5}
\bibliographystyle{myalphaurl}

\medskip

{\small \noindent Max Planck Institute for Gravitational Physics (Albert
Einstein Institute) \newline Am M{\"u}hlen\-berg 1, \newline D-14476 Golm,
Germany \newline \texttt{Ulrich.Menne@aei.mpg.de}

\medskip \noindent University of Potsdam, Institute for Mathematics, \newline
OT Golm, Karl-Liebknecht-Stra{\ss}e 24--25 \newline D-14476 Potsdam, Germany
\newline \texttt{Ulrich.Menne@uni-potsdam.de} }

\end{document}